\documentclass[10pt,letterpaper]{article}
\usepackage[margin=2cm]{geometry}
\usepackage[T1]{fontenc}
\usepackage[utf8]{inputenc}
\usepackage[english]{babel}
\usepackage{amsmath,amssymb,mathtools,amsthm}
\usepackage[cal=boondoxo,scr=boondoxo]{mathalpha}
\usepackage{etoolbox,tcolorbox,needspace}
\tcbuselibrary{skins,breakable}
\usepackage{microtype}
\usepackage{array}
\usepackage[numbers,sort&compress]{natbib}
\usepackage{xurl}
\usepackage{hyperref}
\hypersetup{
  colorlinks=true,
  linkcolor=magenta,
  citecolor=green,
  urlcolor=blue,
  pdftitle={Second-order Wasserstein response for L\'evy laws},
  pdfauthor={Alexandre Autran}
}
\providecommand{\doi}[1]{\href{https://doi.org/#1}{doi:\nolinkurl{#1}}}
\allowdisplaybreaks
\definecolor{theoremframe}{rgb}{0,0,0.55}
\definecolor{theoremback}{rgb}{0.96,0.96,1}
\definecolor{propositionframe}{cmyk}{0.55,0,0,0.45}
\definecolor{propositionback}{cmyk}{0.08,0,0,0}
\definecolor{lemmaframe}{rgb}{0.20471,0.17177,0.32470}
\definecolor{lemmaback}{rgb}{0.94730,0.94290,0.96330}
\definecolor{corollaryframe}{rgb}{0.65059,0.35136,0.13452}
\definecolor{corollaryback}{rgb}{0.99294,0.95020,0.91922}

\tcbset{
  article statement/.style={
    enhanced,
    breakable,
    lines before break=6,
    boxrule=0.8pt,
    arc=1.4mm,
    outer arc=1.4mm,
    boxsep=0pt,
    left=2.5mm,
    right=2.5mm,
    top=3.6mm,
    bottom=2.3mm,
    before skip=9pt plus 2pt minus 1pt,
    after skip=5pt plus 1pt minus 1pt,
    fontupper=\normalfont\normalsize,
    fonttitle=\normalfont\normalsize\bfseries,
    coltitle=white,
    attach boxed title to top left={xshift=2mm,yshift=-2.4mm},
    boxed title style={
      boxrule=0pt,
      arc=0.8mm,
      outer arc=0.8mm,
      boxsep=0pt,
      left=2.8mm,
      right=2.8mm,
      top=1.2mm,
      bottom=1.2mm
    }
  }
}

\newcounter{innerthm}
\newtcolorbox{statementbox}[4]{
  article statement,
  before upper={\addtocounter{innerthm}{-1}\refstepcounter{innerthm}},
  colframe=#2,
  colback=#3,
  colbacktitle=#2,
  title={#1~\theinnerthm\ifstrempty{#4}{}{\enspace|\enspace#4}}
}
\newenvironment{theorem}[1][]{
  \stepcounter{innerthm}
  \begin{statementbox}{Theorem}{theoremframe}{theoremback}{#1}
}{\end{statementbox}}
\newenvironment{lemma}[1][]{
  \stepcounter{innerthm}
  \begin{statementbox}{Lemma}{lemmaframe}{lemmaback}{#1}
}{\end{statementbox}}
\newenvironment{corollary}[1][]{
  \stepcounter{innerthm}
  \begin{statementbox}{Corollary}{corollaryframe}{corollaryback}{#1}
}{\end{statementbox}}
\newenvironment{proposition}[1][]{
  \stepcounter{innerthm}
  \begin{statementbox}{Proposition}{propositionframe}{propositionback}{#1}
}{\end{statementbox}}
\theoremstyle{definition}
\newtheorem{innerassumption}[innerthm]{Assumption}
\newtheorem{innerexample}[innerthm]{Example}
\theoremstyle{remark}
\newtheorem*{remark*}{Remark}
\newtheorem{innerremark}[innerthm]{Remark}
\newenvironment{assumption}[1][]{\begin{innerassumption}[#1]}{\end{innerassumption}}
\newenvironment{remark}[1][]{\begin{innerremark}[#1]}{\end{innerremark}}
\newenvironment{example}[1][]{\begin{innerexample}[#1]}{\end{innerexample}}
\newcommand{\R}{\mathbf R}
\newcommand{\N}{\mathbf N}
\newcommand{\E}{\mathbf E}
\newcommand{\one}{\mathbf 1}
\newcommand{\dd}{\,\mathrm d}
\newcommand{\cL}{\mathcal L}
\newcommand{\cA}{\mathcal A}
\newcommand{\cP}{\mathcal P}
\newcommand{\cM}{\mathcal M}
\newcommand{\cS}{\mathcal S}
\newcommand{\BL}{\mathrm{BL}}
\newcommand{\FM}{\mathrm{FM}}
\newcommand{\norm}[1]{\left\lVert #1\right\rVert}
\DeclareMathOperator{\Law}{Law}
\DeclareMathOperator{\Lip}{Lip}
\DeclareMathOperator{\tr}{tr}
\newcommand{\mapdef}[5]{#1\colon\left|\begin{aligned}
  #2&\longrightarrow #3\\
  #4&\longmapsto #5
\end{aligned}\right.}
\newcommand{\byeq}[2]{
  \overset{\mathclap{\eqref{#1}}}{#2}
}
\title{Second-order Wasserstein response for L\'evy laws}
\author{Alexandre Autran\footnote{Department of Mathematics, École normale supérieure de Rennes, France. \href{mailto:alexandre.autran@ens-rennes.fr}{\nolinkurl{alexandre.autran@ens-rennes.fr}}}}
\date{}

\newlength{\equationmaxwidth}
\makeatletter
\everydisplay\expandafter{\the\everydisplay
  \ifboolexpr{
    test {\ifdefstring{\@currenvir}{equation}}
    or test {\ifdefstring{\@currenvir}{align}}
  }{}{\ifdim\displaywidth>\equationmaxwidth
      \advance\displayindent by \dimexpr(\displaywidth-\equationmaxwidth)/2\relax
      \displaywidth=\equationmaxwidth
    \fi
  }}
\makeatother
\begin{document}
\maketitle
\begin{abstract}
Moving an atom and splitting it symmetrically can give the same weak change in a L\'evy characteristic, yet produce Wasserstein errors of different orders. We explain this through first- and second-order expansions of the distribution function for infinitely divisible laws with a Gaussian factor. The expansions are uniform over Lipschitz tests and include distributional directions from moving atoms. When the first response is nonzero, analyticity gives the signed second coefficient of the Wasserstein--1 distance. When it vanishes, a finite second displacement moment gives a quadratic response, with a positive coefficient whenever the dispersion is nonzero. For local balanced remeshing, we prove matching bounds in terms of grid-alignment variance and construct a linear program with a controlled optimization gap. We also determine the sharp transition as Gaussian smoothing vanishes. The appendices give proofs, metric speed and length, and multivariate and state-dependent extensions.
\end{abstract}
\section{Introduction}\label{sec:intro}
Consider an infinitely divisible law on $\R$ with positive Gaussian variance and one nonzero atom in its variance-weighted jump measure. Moving this atom or splitting it equally to either side, while keeping its total mass fixed, gives the same bounded-Lipschitz discrepancy but different terminal Wasserstein errors: linear for the move and quadratic for the split. Section~\ref{sec:second} gives the coefficients. Starting from this example, we expand the terminal distribution function along changes in the mean and the variance-weighted L\'evy characteristic. The first response records the change in the mean, signed mass changes, and the mean displacement of each part of the characteristic, including the distributional tangents produced by moving atoms. The second response combines displacement dispersion with the composition of the first variation with itself. When the first response is nonzero, Gaussian analyticity gives a signed second coefficient for $W_1$. When it vanishes, only dispersion remains: a finite second displacement moment suffices for a quadratic limit, whose coefficient is positive whenever the dispersion is nonzero. Theorems~\ref{thm:two-jet}--\ref{thm:cancellation} state these results. We then use the response formulas to approximate terminal laws on grids. For compactly supported weighted measures, the best local balanced-grid error is comparable to the grid-alignment variance defined in Section~\ref{sec:remeshing}. This gives a quadratic rate for densities and for atomic sequences satisfying $\mathcal V_h\geq\nu h^2$ for some fixed $\nu>0$, while atoms already on the grid can be reproduced exactly. Minimizing the response norm over a class of local remeshing plans gives a terminal error at most $O(h^2)$ above the best error in that class, or $O(h^4)$ for balanced plans, where $h>0$ is the grid spacing. The linear programs obtained by quadrature converge to this objective. Finally, we describe what happens as Gaussian smoothing vanishes. For centered laws with common Gaussian variance $v\in[0,1]$ and jump variance at most one, the worst terminal error at bounded-Lipschitz discrepancy at most $\delta\in(0,1]$ is comparable to $\tfrac{\delta}{v+\delta^{\frac{2}{3}}}$. The response formulas reveal the cancellations that this worst-case bound does not capture.
\paragraph{Related work.} Weighted characteristic testing, external regularization, and forward--backward Wasserstein bounds are developed by \citet{BretonPrivault2022,BretonPrivault2024}. Exponent interpolation is given in \citep{HoudrePerezAbreuSurgailis1998}. These tools, including linear stability after common Gaussian smoothing, form the basis of our response calculus. \citet[Theorems~7.5--7.7]{Last2014} gives perturbation expansions for Poisson and Lévy functionals under domination conditions. No $\sigma$-finite measure dominates Dirac masses at every point of a continuum. Instead, we differentiate the terminal cumulative distribution function (CDF) in $L^1$ along distributional transport directions, uniformly over Lipschitz tests. \citet{PagesWilbertz2012} established intrinsic stationarity and second-order smooth-test quadrature. We use barycentric cancellation to prove terminal $W_1$ lower bounds under locality and grid-alignment conditions. \citet{GoldfeldEtAl2024} study limit distributions and directional differentiation of Gaussian-smoothed $p$-Wasserstein functionals for $p>1$. Here the variable is a Lévy characteristic and the response is measured in terminal $W_1$. \citet[Sections~2.7--2.8 and~6]{KangLim2025} construct optimal quadratic Markovian couplings and Wasserstein-type metrics on Lévy generators. \citet{LimOoi2026} extend optimality to immersion couplings. Those process constraints differ from our local, reference-dependent terminal response. Related work treats noise sensitivity \citep{GairingEtAl2015,GairingEtAl2018}, Lévy marginals \citep{MariucciReiss2018}, and Wasserstein evolution \citep{AlfonsiCorbettaJourdain2018}.
\paragraph{Organization.} Section~\ref{sec:terminal} introduces the method. Sections~\ref{sec:tangents} and~\ref{sec:second} give the response formulas and the moving-versus-splitting example. Section~\ref{sec:remeshing} applies them to grid approximation, and Section~\ref{sec:sharp} gives the transition as smoothing vanishes. The appendices collect the proofs and extensions by topic.
\section{Terminal laws and the response method}\label{sec:terminal}
\subsection{Notation and function spaces}\label{app:notation}
We write $\R$, $\mathbb C$, and $\mathbb Z$ for the real numbers, complex numbers, and integers, and set $\N\coloneqq \{1,2,\ldots\}$ and $\N_0\coloneqq \N\cup\{0\}$. For a topological space $E$, $\mathcal B(E)$ is its Borel sigma-field and $\cP(E)$ is the set of Borel probability measures on $E$. For $d\in\N$, $\cM_f(\R^d)$ and $\cM_f^+(\R^d)$ are the finite signed and finite positive Borel measures on $\R^d$. We use
\[
 \cP_1(\R^d)\coloneqq \left\{\lambda\in\cP(\R^d):\int_{\R^d}|x|\lambda(\dd x)<\infty\right\}.
\]
Here $x\cdot y\coloneqq x^\top y$ and $|x|\coloneqq (x\cdot x)^{\frac{1}{2}}$ are the Euclidean inner product and norm. The transpose is denoted by $\top$, and $\delta_x\in\cP(\R^d)$ is the unit mass at $x$. For $\kappa\in\cM_f(\R^d)$, $|\kappa|$ is its total variation measure and $\|\kappa\|_{\mathrm{TV}}\coloneqq |\kappa|(\R^d)$. For a $|\kappa|$-integrable Borel function $f$, write $\langle f,\kappa\rangle\coloneqq \int_{\R^d}f\dd\kappa$. For a measurable map $S:E\to F$ and a measure $\mu$ on $E$, the pushforward is $S_\#\mu(A)\coloneqq \mu(S^{-1}(A))$, $A\in\mathcal B(F)$.

\paragraph{Metrics.} For a real-valued function $f$ on $\R^d$, its Lipschitz seminorm is
$$\Lip(f)\coloneqq \sup_{\substack{x,y\in\R^d\\x\ne y}}\frac{|f(x)-f(y)|}{|x-y|}.$$
The space $\BL(\R^d)$ consists of bounded Lipschitz functions, with norm $\|f\|_{\BL}\coloneqq \|f\|_\infty+\Lip(f)$. For $\mu,\eta\in\cM_f(\R^d)$, set
\begin{subequations}\label{eq:metrics}
\begin{equation}
d_{\FM}^{(d)}(\mu,\eta)\coloneqq \sup_{\substack{f\in\BL(\R^d)\\\norm f_{\BL}\leq1}}|\langle f,\mu-\eta\rangle|,
\end{equation}
and
\begin{equation}
\norm\kappa_{\BL^*}\coloneqq d_{\FM}^{(d)}(\kappa,0),\qquad \kappa\in\cM_f(\R^d).
\end{equation}
\end{subequations}
We omit the superscript $d$ in one dimension. The symbol $\BL^*$ denotes the continuous dual of $\BL$; finite signed measures define elements of this dual by integration. For $\lambda_0,\lambda_1\in\cP_1(\R^d)$, Kantorovich duality \citep{Villani2009} gives
\begin{equation}\label{eq:KR}
 W_1(\lambda_0,\lambda_1)
 =\sup_{\substack{h:\R^d\to\R\\\Lip(h)\leq1,\ h(0)=0}}
 |\langle h,\lambda_0-\lambda_1\rangle|.
\end{equation}
The normalization $h(0)=0$ gives $|h(x)|\leq|x|$ without changing the supremum. For $\lambda\in\cP(\R)$, its distribution function is $F_\lambda(x)\coloneqq \lambda(( -\infty,x])$, $x\in\R$.

\paragraph{Function spaces and derivatives.} For $1\leq p<\infty$, $L^p(\R^d)$ consists of measurable functions modulo equality almost everywhere, with $\|f\|_p\coloneqq (\int_{\R^d}|f(x)|^p\dd x)^{\frac{1}{p}}<\infty$. The norm of $L^\infty$ is the essential supremum. These spaces use Lebesgue measure and real scalars, or complex scalars for Fourier transforms. An $L^1$ label above a convergence arrow denotes convergence in the $L^1$ norm. We write $f(t)$ for $f(t,\cdot)$ when a time-dependent function is viewed as a spatial function. For $k\in\N_0$ and $1\leq p\leq\infty$,
\[
 W^{k,p}(\R^d)\coloneqq \{f\in L^p(\R^d):D^\beta f\in L^p(\R^d),\ |\beta|\leq k\},
\]
with norm
\[
 \norm f_{W^{k,p}}\coloneqq \sum_{\substack{\beta\in\N_0^d\\|\beta|\leq k}}\norm{D^\beta f}_p.
\]
The derivatives are weak derivatives, with $\beta=(\beta_1,\ldots,\beta_d)\in\N_0^d$, $|\beta|\coloneqq \sum_{j=1}^d\beta_j$, and $D^\beta\coloneqq \partial_1^{\beta_1}\cdots\partial_d^{\beta_d}$. The spaces $C^k$, $C_b^k$, and $C_c^\infty$ consist of functions with $k$ continuous derivatives, functions in $C^k$ whose derivatives through order $k$ are bounded, and smooth compactly supported functions, respectively. We write $C_b\coloneqq C_b^0$ and use $C_b^\infty\coloneqq \bigcap_{k\in\N_0}C_b^k$ and $\|f\|_{C_b^k}\coloneqq \sum_{j=0}^k\|D^jf\|_\infty$. For vector-valued functions these definitions apply componentwise. The tensor $D^jf(x)$ carries the multilinear norm induced by Euclidean norms; $\nabla f$ and $\nabla^2f$ are the gradient and Hessian. In one dimension $D\coloneqq \partial_x$; for a unit vector $e\in\R^d$, $D_e\coloneqq e\cdot\nabla$. An omitted spatial domain is $\R$ in the one-dimensional sections and $\R^d$ in the multivariate sections; an omitted scalar codomain is $\R$. For a bounded linear map $A:X\to Y$ between normed spaces, $\|A\|_{X\to Y}\coloneqq \sup_{\substack{x\in X\\\|x\|_X\leq1}}\|Ax\|_Y$, and $I$ denotes the identity on the space in question.

\paragraph{Distributions and convolution.} The spaces $\mathcal D'(\R^d)$ and $\cS'(\R^d)$ are the continuous duals of $C_c^\infty(\R^d)$ and the Schwartz space $\cS(\R^d)$ of smooth rapidly decreasing functions. For a distribution $u$, $\langle\phi,D_ju\rangle\coloneqq -\langle\partial_j\phi,u\rangle$. For $f\in L^1(\R^d)$ and $\mu\in\cM_f(\R^d)$, our conventions are
\[
 (f*\mu)(x)\coloneqq \int_{\R^d}f(x-y)\mu(\dd y), 
\qquad
 \check\mu(A)\coloneqq \mu(-A),\qquad A\in\mathcal B(\R^d).
\]
For Fourier transforms, we use
\[
 \widehat f(\xi)\coloneqq \int_{\R^d}e^{-i\xi\cdot x}f(x)\dd x,
\qquad 
 \widehat\mu(\xi)\coloneqq \int_{\R^d}e^{-i\xi\cdot x}\mu(\dd x),\qquad \xi\in\R^d.
\]
Here $i^2=-1$. Convolution of finite measures is the pushforward of their product under $(x,y)\mapsto x+y$. We also use convolution for other functions whenever the integral is finite, and extend the Fourier transform to $\cS'$ by duality.

\paragraph{Probability and asymptotic notation.} Random variables are defined on a probability space $(\Omega,\mathcal F,\mathbf P)$. The symbols $\Law(X)$ and $\E$ denote the law of $X$ and expectation. For a Banach space $E$ and $1\leq p<\infty$, $L^p(\Omega;E)$ consists of strongly measurable $E$-valued variables with $\E[\|X\|_E^p]<\infty$; $L^\infty$ uses the essential supremum. The same convention with Lebesgue measure defines $L^p(I;E)$ on an interval $I$. For a square-integrable vector $X$, $\operatorname{Cov}(X)\coloneqq \E[(X-\E X)(X-\E X)^\top]$; in one dimension this is $\operatorname{Var}(X)$. We also write $\operatorname{Cov}(\lambda)$ for the covariance of a law $\lambda$. For $x\in\R$ and a set $A$, $x_+\coloneqq \max\{x,0\}$, $\one_A$ is the indicator of $A$, and $\operatorname{sgn}(x)$ is $-1,0,1$ according as $x<0,x=0,x>0$. The symbols $\wedge$, $\vee$, $\lfloor\cdot\rfloor$, and $\lceil\cdot\rceil$ denote minimum, maximum, floor, and ceiling. Constants $c,C>0$ may change between estimates; their dependence is stated locally. We use $a\lesssim b$ for $a\leq Cb$, $a\asymp b$ for two-sided comparison, and $a\sim b$ when the ratio $\tfrac{a}{b}$ tends to one in the stated limit. For a normed space $X$, $r_\varepsilon=o_X(|\varepsilon|^k)$ means $\tfrac{\|r_\varepsilon\|_X}{|\varepsilon|^k}\xrightarrow[\varepsilon\to0]{}0$; $O(|\varepsilon|^k)$ denotes a scalar remainder bounded in absolute value by a constant times $|\varepsilon|^k$. Without a subscript the absolute value is used. Limits of characteristic paths run through admissible parameter values.

\paragraph{Gaussian and L\'evy notation.} For $v>0$, the density of $N(0,v)$ is $g_v(x)\coloneqq (2\pi v)^{-\frac{1}{2}}e^{-\frac{x^2}{2v}}$ with $x\in\R$.  We write $\varphi\coloneqq g_1$ and, for $\sigma>0$, set $\varphi_\sigma\coloneqq g_{\sigma^2}$. A standard Gaussian variable is denoted by $G\sim N(0,1)$. A L\'evy measure $N$ on $\R^d$ is a positive Borel measure with $N(\{0\})=0$ and $\int_{\R^d}(1\wedge|y|^2)N(\dd y)<\infty$. A probability kernel from $E$ to $F$ is a map $K:E\times\mathcal B(F)\to[0,1]$ such that $K(x,\cdot)\in\cP(F)$ and $K(\cdot,A)$ is measurable for each Borel $A$. A L\'evy kernel assigns a L\'evy measure to each time or state, with measurable dependence on that parameter when evaluated on Borel sets. For a symmetric matrix $V\in\R^{d\times d}$, $V\succeq0$ means positive semidefinite. For $A,B\in\R^{d\times d}$, $A:B\coloneqq \operatorname{tr}(A^\top B)=\sum_{j,k=1}^d A_{jk}B_{jk}$.

\subsection{Terminal characteristics and interpolation}

We work with one-dimensional infinitely divisible laws of finite variance. Let $\lambda\in\cP_1(\R)$ be such a law, with mean $m\in\R$, Gaussian variance $v>0$, and L\'evy measure $N$ satisfying $\int_\R y^2N(\dd y)<\infty$. Its finite positive measures $\mu(\dd y)\coloneqq y^2N(\dd y)$ and $\Gamma\coloneqq v\delta_0+\mu$ combine Gaussian and jump variance. The continuous characteristic exponent $\Psi:\R\to\mathbb C$, normalized by $\Psi(0)=0$ and defined by $\int_\R e^{i\xi x}\lambda(\dd x)=e^{\Psi(\xi)}$, is
\[
 \Psi(\xi)\coloneqq im\xi-\frac v2\xi^2
 +\int_\R\frac{e^{i\xi y}-1-i\xi y}{y^2}\,\mu(\dd y),\qquad \xi\in\R,
\]
where the quotient at $y=0$ is $-\tfrac{\xi^2}{2}$. The density $p\coloneqq g_v*\rho$ belongs to $C_b^\infty(\R)\cap W^{k,1}(\R)$ for every $k\in\N_0$, where $\rho\in\cP_1(\R)$ is the jump part translated by the mean. For $\lambda_0,\lambda_1\in\cP_1(\R)$, the distance also has the form
\[
 W_1(\lambda_0,\lambda_1)=\norm{F_{\lambda_0}-F_{\lambda_1}}_1.
\]
This is why we expand distribution functions.

\paragraph{Interpolation.} For $f\in C^2(\R)$ with bounded $f'$, $f''$, the weighted Taylor formula writes the terminal generator as
\[
 \mathcal L f(x)\coloneqq mf'(x)+\int_\R\Phi_{f,x}(y)\Gamma(\dd y),\qquad x\in\R,
\]
where
\[
 \Phi_{f,x}(y)\coloneqq \int_0^1(1-r)f''(x+ry)\dd r,\qquad y\in\R.
\]
Thus $\mathcal Lf\in C_b(\R)$ and $\Phi_{f,x}\in C_b(\R)$. For infinitely divisible laws $\lambda_0,\lambda_1\in\cP_1(\R)$ of finite variance with exponents $\Psi_0,\Psi_1$, let $\pi_u\in\cP_1(\R)$ have exponent $(1-u)\Psi_1+u\Psi_0$, $u\in[0,1]$. For a Lipschitz function $h:\R\to\R$, define
\[
 Q_uh(x)\coloneqq \int_\R h(x+z)\pi_u(\dd z),\qquad x\in\R.
\]
These convolution operators preserve the class of Lipschitz functions. Let $\mathcal L_i$ be the terminal generator of $\lambda_i$, $i=0,1$. On $h\in\cS(\R)$, the interpolation identity reads
\[
 (Q_1-Q_0)h(x)=\int_0^1(\mathcal L_0-\mathcal L_1)Q_uh(x)\dd u,
 \qquad x\in\R.
\]
Repeated interpolation gives higher-order terms. A common Gaussian factor bounds the derivatives of $Q_uh$ uniformly over $\Lip(h)\leq1$. Appendix~\ref{app:terminal} proves these identities and the resulting stability bounds.

\paragraph{The affine response.} Let $\widetilde\lambda\in\cP_1(\R)$ have mean $\widetilde m\in\R$ and weighted characteristic $\widetilde\Gamma\in\cM_f^+(\R)$. Write $\delta m\coloneqq \widetilde m-m\in\R$ and $\kappa\coloneqq \widetilde\Gamma-\Gamma\in\cM_f(\R)$. The affine CDF response $H_{\delta m,\kappa}\in L^1(\R)$ is
\[
 H_{\delta m,\kappa}(x)\coloneqq -\delta m\,p(x)
 +\int_\R\int_0^1(1-r)p'(x-ry)\dd r\,\kappa(\dd y).
\]
Here $\widetilde\lambda$ is infinitely divisible with finite variance. When both Gaussian variances are at least $v_*>0$, Corollary~\ref{cor:metric-response} gives
\[
 \left|W_1(\widetilde\lambda,\lambda)-\norm{H_{\delta m,\kappa}}_1\right|
 \leq C(v_*)\bigl(|\delta m|+\norm\kappa_{\BL^*}\bigr)^2.
\]
The next section extends this response to moving and splitting atoms.
\section{Observable transport tangents}\label{sec:tangents}
Moving atoms give distributional directions for the characteristic. The next formula covers transport, splitting, and signed mass changes.
\subsection{A lifted perturbation and its observable derivative}
Fix a finite-variance terminal law $\lambda$ with mean $m$ and $\Gamma=v\delta_0+\mu$, where $v>0$ and $\mu=y^2N$. Let $\Pi$ be a finite positive measure on $\R_y\times\R_z$ with first marginal $\mu$. The variables $y$ and $z$ denote position and displacement velocity, respectively. Assume $M_2(\Pi)\coloneqq \int_{\R^2}z^2\Pi(\dd y,\dd z)<\infty$. For $\alpha\in\R$ and $\eta\in\cM_f(\R)$, set

\begin{subequations}\label{eq:transport-lift}
\begin{equation}
m_\varepsilon\coloneqq m+\varepsilon\alpha
\end{equation}
and
\begin{equation}
\Gamma_\varepsilon\coloneqq v\delta_0+(y+\varepsilon z)_\#\Pi+\varepsilon\eta.
\end{equation}
\end{subequations}
Here $(y+\varepsilon z)_\#\Pi$ is the pushforward under the map $\R^2\ni(y,z)\mapsto y+\varepsilon z\in\R$. A parameter $\varepsilon\in\R$ is admissible when $\Gamma_\varepsilon\geq v_*\delta_0$, for a fixed $0<v_*\leq v$; the inequality means that $\Gamma_\varepsilon-v_*\delta_0$ is a positive measure. Each admissible measure gives a triplet through $v_\varepsilon\coloneqq \Gamma_\varepsilon(\{0\})$ and $N_\varepsilon\coloneqq \one_{\{y\ne0\}}y^{-2}\Gamma_\varepsilon$. With this convention, mass moved to zero becomes Gaussian variance. Write $\lambda_\varepsilon$ for the resulting law and $\lambda$ for the reference law, with density $p$. Define $J\in\cM_f(\R)$ and $\tau\in\mathcal D'(\R)$ by
\begin{subequations}\label{eq:transport-tangent}
\begin{equation}
J(A)\coloneqq \int_{A\times\R}z\Pi(\dd y,\dd z)
\end{equation}
and
\begin{equation}
\tau\coloneqq \eta-D_yJ,
\end{equation}
so that
\begin{equation}
\langle f,\tau\rangle=\int_\R f(y)\eta(\dd y)+\int_\R f'(y)J(\dd y).
\end{equation}
\end{subequations}
Here $A\in\mathcal B(\R)$. The pairing is defined for $f\in C_c^\infty(\R)$ and extends to $f\in C_b^1(\R)$. Cauchy--Schwarz gives
\[
 \norm J_{\mathrm{TV}}\leq\int_{\R^2}|z|\Pi(\dd y,\dd z)
 \leq\mu(\R)^{\frac{1}{2}}M_2(\Pi)^{\frac{1}{2}}<\infty.
\]
\begin{theorem}[Transport response with a uniform quadratic remainder]
\label{thm:transport-response}
Under \eqref{eq:transport-lift}, define
\begin{equation}\label{eq:transport-cdf}
\begin{aligned}
 \mathsf T_p(\alpha,\eta,J)(x)
 &\coloneqq -\alpha p(x)+\int_\R\int_0^1(1-r)p'(x-ry)\dd r\,\eta(\dd y)\\
 &\quad-\int_\R\int_0^1r(1-r)p''(x-ry)\dd r\,J(\dd y).
\end{aligned}
\end{equation}
This defines a linear map $\mathsf T_p:\R\times\cM_f(\R)\times\cM_f(\R)\to L^1(\R)$. We omit its arguments when they are fixed. For every admissible $\varepsilon$,
\begin{equation}
\norm{F_{\lambda_\varepsilon}-F_\lambda
                -\varepsilon\mathsf T_p}_1\leq C\varepsilon^2.
                      \label{eq:transport-L1}
\end{equation}
Consequently, the Wasserstein distance satisfies
\begin{equation}
\left|W_1(\lambda_\varepsilon,\lambda)
                  -|\varepsilon|\norm{\mathsf T_p}_1\right|
                  \leq C\varepsilon^2.\label{eq:transport-W1}
\end{equation}
The constant $C$ depends on $v_*$, $|\alpha|$, $\|\eta\|_{\BL^*}$, and the first two displacement moments. Its value is given in \eqref{eq:transport-constant}.
\end{theorem}
\paragraph{Proof idea.} Taylor expansion of the displacement gives the first variation $\eta-D_yJ$. The affine response bound controls the remaining change in the law uniformly over Lipschitz tests. Integration by parts then gives \eqref{eq:transport-cdf}, and $W_1=\|F_{\lambda_\varepsilon}-F_\lambda\|_1$ gives the distance expansion. The proof is in Appendix~\ref{proof:thm:transport-response}. A moving atom has a distributional tangent even though its characteristic difference quotients need not converge in $\BL^*$; see \eqref{eq:non-BL-tangent}. Moving a weighted atom at fixed mass differs from moving a raw L\'evy atom at fixed intensity. Appendix~\ref{sec:raw-atom} gives the latter expansion.

\subsection{A tangent norm independent of the lift}
\begin{proposition}[Observable quotient and convolution contraction]
\label{prop:tangent-quotient}
For $\alpha\in\R$, $\eta,J\in\cM_f(\R)$, and the density $p$ of a Gaussian-smoothed infinitely divisible law $\lambda\in\cP_1(\R)$, $\mathsf T_p(\alpha,\eta,J)=0$ if and only if $\alpha=0$ and $\eta-D_yJ=0$ in $\mathcal D'(\R)$. Consequently
\begin{equation}\label{eq:observable-norm}
 \| (\alpha,\tau)\|_\lambda\coloneqq \norm{\mathsf T_p(\alpha,\eta,J)}_1
\end{equation}
defines a norm on $\R\times\{\eta-D_yJ:\eta,J\in\cM_f(\R)\}$, independent of the representation of $\tau$. If $\zeta\in\cP_1(\R)$ is infinitely divisible, then
\begin{equation*}
 \| (\alpha,\tau)\|_{\lambda*\zeta}
                  \leq\| (\alpha,\tau)\|_\lambda.
\end{equation*}
\end{proposition}
\begin{remark*}
Along \eqref{eq:transport-lift}, this norm is the first-order $W_1$ speed.
\end{remark*}
\paragraph{Proof idea.} The Fourier transform of $D_x\mathsf T_p$ is the product of the first exponent variation and $\widehat p$. Since $\widehat p$ never vanishes, a zero response forces $\alpha=0$ and $\eta-D_yJ=0$. Convolution contraction follows from Young's inequality. The proof is in Appendix~\ref{proof:prop:tangent-quotient}. For example, $\eta\coloneqq q'(y)\dd y$ and $J\coloneqq q(y)\dd y$, with $q\in C_c^\infty(\R)$, give zero first-order change even when both measures are nonzero. The quotient identifies these equivalent representations. Likewise, two splitting plans with the same $J$ and $\eta$ have the same first-order response: only their conditional mean displacement matters at this order, while the second displacement moment controls the remainder. The same norm gives metric speed and length along admissible curves. Appendix~\ref{app:geometry} develops this geometry for atomic charts and non-atomic lifts.
\section{Second-order response and cancellation}\label{sec:second}
The second response has two terms: displacement dispersion and the composition of the first variation. When the first tangent vanishes, the dispersion term determines the response.
\subsection{A second-order expansion for lifted paths}
For the lift \eqref{eq:transport-lift}, a Borel set $A\subset\R$, and an integer $j\geq1$, define the displacement moments $M_j\in[0,\infty]$ and the positive dispersion measure $S$ by
\begingroup
\begin{equation*}
S(A)\coloneqq \int_{A\times\R}z^2\Pi(\dd y,\dd z), 
\qquad
M_j\coloneqq \int_{\R^2}|z|^j\Pi(\dd y,\dd z).
\end{equation*}
\endgroup
When $M_2<\infty$, the measure $S\in\cM_f^+(\R)$ has mass $S(\R)=M_2$. For $\phi\in C_c^\infty(\R)$, Taylor expansion gives
\[\begin{aligned}
 \langle\phi,\Gamma_\varepsilon-\Gamma\rangle
 ={}&\varepsilon\left(\int_\R\phi(y)\eta(\dd y)+\int_{\R^2}z\phi'(y)\Pi(\dd y,\dd z)\right)\\
 &+\frac{\varepsilon^2}{2}\int_{\R^2}z^2\phi''(y)\Pi(\dd y,\dd z)+o(\varepsilon^2).
\end{aligned}\]
The distributional coefficients are $\tau=\eta-D_yJ\in\mathcal D'(\R)$ and $\tfrac{D_y^2S}{2}$. Define the linear maps $\mathcal A:C_b^3(\R)\to C_b(\R)$ and $\mathsf Q_p:\cM_f(\R)\to L^1(\R)$ by
\begin{equation*}
\mathcal A f(x)\coloneqq \alpha f'(x)+\int_\R\Phi_{f,x}(y)\eta(\dd y)
                           +\int_\R\partial_y\Phi_{f,x}(y)J(\dd y),
\end{equation*}
and
\begin{equation}
\mathsf Q_p(S)(x)\coloneqq \tfrac12\int_\R\!\int_0^1r^2(1-r)
                              p'''(x-ry)\dd r\,S(\dd y).
                                  \label{eq:second-observable}
\end{equation}
The spatial adjoint $\mathcal A^*:W^{3,1}(\R)\to L^1(\R)$ acts on $u\in W^{3,1}(\R)$ as
\begin{equation*}
\begin{aligned}
 \mathcal A^*u(x)=&-\alpha u'(x)+\int_\R\int_0^1(1-r)u''(x-ry)\dd r\,\eta(\dd y)\\
 &-\int_\R\int_0^1r(1-r)u'''(x-ry)\dd r\,J(\dd y).
\end{aligned}
\end{equation*}
In particular, $D_x\mathsf T_p=\mathcal A^*p$, and both $\mathsf T_p(\alpha,\eta,J)$ and $\mathsf Q_p(S)$ belong to $W^{k,1}(\R)$ for every $k\in\N_0$. Thus $\mathcal A^*\mathsf T_p$ is well defined in $L^1(\R)$.
\begin{theorem}[Observable second-order expansion]\label{thm:two-jet}
Assume $M_3<\infty$ and the uniform lower bound $v_*>0$ on the Gaussian variance in \eqref{eq:transport-lift}. Put
\begin{subequations}\label{eq:two-jet}
\begin{equation}
T_1\coloneqq \mathsf T_p(\alpha,\eta,J)
\end{equation}
and
\begin{equation}
T_2\coloneqq \mathsf Q_p(S)+\tfrac12\mathcal A^*T_1.
\end{equation}
\end{subequations}
Both coefficients belong to $L^1(\R)$. For every admissible $|\varepsilon|\leq1$, we have
\begin{equation}
\norm{F_{\lambda_\varepsilon}-F_\lambda
                -\varepsilon T_1-\varepsilon^2T_2}_1
                         \leq C_3|\varepsilon|^3.
                                      \label{eq:two-jet-L1}
\end{equation}
Consequently, the Wasserstein distance satisfies
\begin{equation}
\left|W_1(\lambda_\varepsilon,\lambda)
                -\norm{\varepsilon T_1+\varepsilon^2T_2}_1\right|
                         \leq C_3|\varepsilon|^3.
                                      \label{eq:two-jet-W1}
\end{equation}
The constant $C_3$ depends on $v_*$, $|\alpha|$, $\|\eta\|_{\mathrm{TV}}$, and $M_1,M_2,M_3$; see \eqref{eq:two-jet-constant} for its value.
\end{theorem}
\begin{remark*}
Only the displacement velocities require a third moment.
\end{remark*}
\paragraph{Proof idea.} Expand the generator difference through order two, then expand the terminal interpolation. The direct second variation gives $\mathsf Q_p(S)$; composing the first variation with itself gives $\tfrac{\mathcal A^*T_1}{2}$. Gaussian derivative bounds control the remainder uniformly over Lipschitz tests. The proof is in Appendix~\ref{proof:thm:two-jet}. Formula \eqref{eq:two-jet} separates the first and second characteristic variations. An additional order-$\varepsilon^2$ mean and characteristic change $(\beta,\zeta)\in\R\times\cM_f(\R)$, with a third-order remainder on the same smooth tests, adds $\mathsf T_p(\beta,\zeta,0)$ to $T_2$.
\Needspace{12\baselineskip}
\subsection{The scalar second coefficient of the Wasserstein distance}
\begin{theorem}[Scalar second-order response]\label{thm:scalar-second}
Under the assumptions of Theorem~\ref{thm:two-jet}, suppose $(\alpha,\tau)\ne(0,0)$. Then $\|T_1\|_1>0$ and, along admissible parameters of either sign,
\begin{equation}\label{eq:scalar-second}
 W_1(\lambda_\varepsilon,\lambda)
 =|\varepsilon|\norm{T_1}_1+
 \varepsilon|\varepsilon|\int_\R\operatorname{sgn}(T_1(x))T_2(x)\dd x
                              +o(\varepsilon^2).
\end{equation}
\end{theorem}
\paragraph{Proof idea.} Gaussian convolution makes $T_1$ analytic. Proposition~\ref{prop:tangent-quotient} shows that it is nonzero, so its real zeros have measure zero. Differentiating the $L^1$ norm of $T_1+\varepsilon T_2$ gives the signed second coefficient. The proof is in Appendix~\ref{proof:thm:scalar-second}. The coefficient $\int_\R\operatorname{sgn}(T_1(x))T_2(x)\dd x$ can have either sign. The leading term $\|T_1\|_1|\varepsilon|$ has slopes $-\|T_1\|_1$ and $\|T_1\|_1$ along admissible parameters approaching zero from the left and right, so \eqref{eq:scalar-second} describes a cusp at the reference law.
\subsection{Zero first tangent: a quadratic law}
\begin{theorem}[Cancellation and a nonzero second response]
\label{thm:cancellation}
In \eqref{eq:transport-lift}, suppose $M_2<\infty$, $\alpha=0$, and $\eta-D_yJ=0$. Then, along admissible parameters,
\begin{subequations}\label{eq:cancelled-expansion}
\begin{equation}
\norm{F_{\lambda_\varepsilon}-F_\lambda-\varepsilon^2\mathsf Q_p(S)}_1=o(\varepsilon^2).
\end{equation}
Consequently, the Wasserstein distance satisfies
\begin{equation}
W_1(\lambda_\varepsilon,\lambda)=\varepsilon^2\norm{\mathsf Q_p(S)}_1+o(\varepsilon^2).
\end{equation}
\end{subequations}
The coefficient is positive whenever $M_2>0$. The remainder is $O(|\varepsilon|^3)$ if $M_3<\infty$. If $\Pi$ is invariant under $(y,z)\mapsto(y,-z)$ and $M_4<\infty$, it is $O(\varepsilon^4)$.
\end{theorem}
\paragraph{Proof idea.} The condition $\alpha=\eta-D_yJ=0$ cancels the first generator variation. Taylor's formula and dominated convergence against $z^2\Pi$ then give the quadratic limit using only $M_2<\infty$. At low Fourier frequencies, the second response is nonzero whenever $S(\R)>0$. Symmetry cancels the cubic term and gives the fourth-order remainder. The proof is in Appendix~\ref{proof:thm:cancellation}.
\subsection{The same weak error can have two observable orders}
Fix $a\in\R\setminus\{0\}$, $M,v>0$, and $m\in\R$. Let $\lambda\in\cP_1(\R)$ have mean $m$, weighted characteristic $\Gamma=v\delta_0+M\delta_a$, and density $p$. Compare the paths
\begingroup
\begin{equation*}
\Gamma_\varepsilon^{\mathrm{move}}\coloneqq v\delta_0+M\delta_{a+\varepsilon}, 
\qquad
\Gamma_\varepsilon^{\mathrm{split}}\coloneqq v\delta_0+\frac M2(\delta_{a-\varepsilon}+\delta_{a+\varepsilon}).
\end{equation*}
\endgroup
Keep the mean $m$ fixed and take $0<|\varepsilon|<\tfrac{|a|}{2}$. Both paths preserve total variance and have the same weak error,
\begin{equation*}
 d_{\FM}(\Gamma_\varepsilon^{\mathrm{move}},\Gamma)
 =d_{\FM}(\Gamma_\varepsilon^{\mathrm{split}},\Gamma)
 =\frac{2M|\varepsilon|}{2+|\varepsilon|}.
\end{equation*}
For the split path, every $\|f\|_{\BL}\leq1$ gives
\[
 \left|\frac{f(a-\varepsilon)+f(a+\varepsilon)}2-f(a)\right|
 \leq\min\{2\norm f_\infty,|\varepsilon|\Lip(f)\}
 \leq\frac{2|\varepsilon|}{2+|\varepsilon|}.
\]
Equality is attained by the bounded Lipschitz function
\[
 f(y)\coloneqq \frac{\min\{|\varepsilon|,-|\varepsilon|+2|y-a|\}}{2+|\varepsilon|},\qquad y\in\R.
\]
The moving case follows from Lemma~\ref{lem:atoms}.
\begin{proposition}[Nonzero first and second observable coefficients]
\label{prop:split-order}
For the laws $\lambda_\varepsilon^{\mathrm{move}},\lambda_\varepsilon^{\mathrm{split}}\in\cP_1(\R)$ with these characteristics and common mean $m$,
\begin{equation}
W_1(\lambda_\varepsilon^{\mathrm{move}},\lambda)
 =|\varepsilon|C_1+O(\varepsilon^2),
\qquad 
W_1(\lambda_\varepsilon^{\mathrm{split}},\lambda)
 =\varepsilon^2C_2+O(\varepsilon^4),\label{eq:split-order}
\end{equation}
where
\[
 C_1\coloneqq M\left\|\int_0^1r(1-r)p''(\cdot-ra)\dd r\right\|_1>0
\]
and
\[
 C_2\coloneqq \frac M2\left\|\int_0^1r^2(1-r)p'''(\cdot-ra)\dd r\right\|_1>0.
\]
The absolute remainder in \eqref{eq:split-order} is bounded by $\varepsilon^4\left(\tfrac{MK_5}{720}+\tfrac{M^2K_7}{1152}\right)$, where $K_j\coloneqq \|g_v^{(j)}\|_1$.
\end{proposition}
\begin{proof}
For the moving path, $J=M\delta_a$ and $\eta=0$. Theorem~\ref{thm:transport-response} and Proposition~\ref{prop:tangent-quotient} show that $C_1>0$. For the symmetric split, $J=0$ and $S=M\delta_a$. Its second and fourth absolute displacement moments both equal $M$. Theorem~\ref{thm:cancellation} gives $C_2>0$, with the fourth-order remainder stated above.
\end{proof}
Equal weak discrepancies produce linear or quadratic terminal errors at the same reference law. The split path has $J=0$ and $M_2>0$, so no uniform local lower bound proportional to $d_{\FM}$ can hold. Section~\ref{sec:sharp} gives sharp worst-case bounds. Appendix~\ref{app:further-second} treats varying redistribution kernels, matching second variations, and fixed-intensity atoms.
\section{Balanced remeshing and optimization}\label{sec:remeshing}
\subsection{Balanced redistribution and second-order approximation}
\begin{theorem}[Barycentric remeshing of weighted jump measures]
\label{thm:barycentric}
Let $\mu,\widehat\mu\in\cM_f^+(\R)$ have a coupling $\pi\in\cM_f^+(\R^2)$, whose marginals are $\mu,\widehat\mu$, with finite quadratic displacement cost $\mathcal C_2\coloneqq \int_{\R^2}|z-y|^2\pi(\dd y,\dd z)$. Suppose that, for every bounded Borel function $f:\R\to\R$,
\begin{equation}\label{eq:barycentric}
 \int_{\R^2} (z-y)f(y)\pi(\dd y,\dd z)=0
.
\end{equation}
For the terminal laws $\lambda,\widehat\lambda\in\cP_1(\R)$ with common mean $m\in\R$ and joint characteristics $v\delta_0+\mu$, $v\delta_0+\widehat\mu$, where $v>0$,
\begin{equation}\label{eq:barycentric-bound}
 W_1(\lambda,\widehat\lambda)
          \leq\frac{\norm{g_v'''}_1}{24}\,\mathcal C_2.
\end{equation}
\end{theorem}
\begin{remark*}
The same bound holds if \eqref{eq:barycentric} is imposed with $f(z)$ instead of $f(y)$.
\end{remark*}
\paragraph{Proof idea.} Balance cancels the term linear in $z-y$ in the Taylor expansion of $\Phi_{f,x}$. The remaining term is bounded by $\tfrac{\mathcal C_2\|f^{(4)}\|_\infty}{24}$. Interpolation and Gaussian smoothing turn this into \eqref{eq:barycentric-bound}. The proof is in Appendix~\ref{proof:thm:barycentric}. For $h>0$, define a Borel probability kernel $K_h$ from $\R$ to $\R$ as follows: for $k\in\mathbb Z$ and $y\in[kh,(k+1)h]$, set
\[
 K_h(y,\dd z)\coloneqq \frac{(k+1)h-y}{h}\delta_{kh}(\dd z)
             +\frac{y-kh}{h}\delta_{(k+1)h}(\dd z).
\]
Its first two displacement moments are
\[
 \int_\R(z-y)K_h(y,\dd z)=0
\]
and
\[
 \int_\R(z-y)^2K_h(y,\dd z)=(y-kh)((k+1)h-y)\leq \frac{h^2}{4}.
\]
Set $\widehat\mu_h\coloneqq \int_\R K_h(y,\cdot)\mu(\dd y)\in\cM_f^+(\R)$, and let $\widehat\lambda_h\in\cP_1(\R)$ have mean $m$ and weighted characteristic $v\delta_0+\widehat\mu_h$. Then
\begingroup
\begin{equation*}
W_1(\lambda,\widehat\lambda_h)\leq\frac{\norm{g_v'''}_1\mu(\R)}{96}\,h^2, 
\qquad
\widehat N_h(\R)\leq\frac{\mu(\R)}{h^2}.
\end{equation*}
\endgroup
Here $\widehat v_h\coloneqq v+\widehat\mu_h(\{0\})$ and $\widehat N_h\coloneqq \one_{\{z\ne0\}}z^{-2}\widehat\mu_h$. Indeed,
\[
\widehat v_h+\int_\R z^2\widehat N_h(\dd z)=v+\widehat\mu_h(\R)=v+\mu(\R)
\]
and
\[
\widehat N_h(\R)=\int_{\R\setminus\{0\}}z^{-2}\widehat\mu_h(\dd z)
 \leq h^{-2}\widehat\mu_h(\R).
\]
The approximation preserves mean and total variance and has finite activity. Compact support of $\mu$ gives finitely many jump sizes. Finite jump variance suffices for the bound. Applying barycentric splitting \citep{PagesWilbertz2012} to $y^2N$ makes its second-order cancellation uniform over terminal Lipschitz observables.
\subsection{Sharpness of local balanced remeshing}
Let $\mu\in\cM_f^+(\R)$ be supported on $[-R,R]$, where $R\geq0$, and let $M\coloneqq \mu(\R)>0$, $v>0$, $m\in\R$, and $L\geq1$. The reference law $\lambda\in\cP_1(\R)$ has mean $m$ and characteristic $v\delta_0+\mu$. For $h>0$, define the grid maps $\ell_h,r_h:\R\to h\mathbb Z$ by
\[
 \ell_h(y)\coloneqq h\left\lfloor \frac{y}{h}\right\rfloor, 
\qquad
 r_h(y)\coloneqq \ell_h(y)+h.
\]
The alignment variance $\mathcal V_h\in[0,\infty)$ is then
\[
 \mathcal V_h\coloneqq \int_\R(y-\ell_h(y))(r_h(y)-y)\,\mu(\dd y).
\]
Let $\mathcal B_h\subset\cM_f^+(\R^2)$ be the balanced couplings with first marginal $\mu$, a second marginal $\widehat\mu\in\cM_f^+(\R)$ supported on $h\mathbb Z$, and $|z-y|\leq Lh$ for $\pi$-almost every $(y,z)$. Balance means \eqref{eq:barycentric}. The law $\widehat\lambda_\pi$ has mean $m$ and characteristic $v\delta_0+\widehat\mu$.
\begin{theorem}[Two-sided balanced-grid error]\label{thm:sharp-grid}
There are $h_0,c,C>0$, depending only on $v,M,R,L$, such that
\begin{equation*}
 c\mathcal V_h\leq
 \inf_{\pi\in\mathcal B_h}W_1(\lambda,\widehat\lambda_\pi)
 \leq C\mathcal V_h,\qquad 0<h\leq h_0.
\end{equation*}
Take $C\coloneqq \tfrac{\|g_v'''\|_1}{24}$ and $c\coloneqq \tfrac{\gamma}{2}$, where
\begin{subequations}\label{eq:grid-lower-constant}
\begin{equation}
\xi\coloneqq (1+R)^{-1}
\end{equation}
and
\begin{equation}
\gamma\coloneqq \frac{\xi^3}{24}e^{-\frac{(v+M)\xi^2}{2}}\cos(R\xi)>0.
\end{equation}
\end{subequations}
\end{theorem}
\begin{remark*}
If $\mathcal V_h\geq\nu h^2$ along a sequence, with $\nu>0$, we obtain matching $h^2$ upper and lower rates on that sequence.
\end{remark*}
\paragraph{Proof idea.} Every balanced grid coupling has squared displacement cost at least $\mathcal V_h$. Splitting between the two neighbouring grid points attains this cost. For the lower bound, compact support bounds the second response below by a fixed positive multiple of the squared displacement cost, while locality controls its remainder. The proof is in Appendix~\ref{proof:thm:sharp-grid}. For $\mu=M\delta_a$, $a>0$, taking $h_n\coloneqq \tfrac{a}{n+\theta}$, $n\in\N$, with fixed $0<\theta<1$ gives $\mathcal V_{h_n}=M\theta(1-\theta)h_n^2$. If $\mu(\dd y)=w(y)\dd y$ with compactly supported $w\in L^1(\R)$, set $q:\R\to\left[0,\tfrac{1}{4}\right]$, $q(t)\coloneqq \{t\}(1-\{t\})$, where $\{t\}\coloneqq t-\lfloor t\rfloor$ is the fractional part. For fixed $a,b\in\R$ with $a<b$,
\[
 \frac{\mathcal V_h}{h^2}=\int_\R w(y)q\left(\frac{y}{h}\right)\dd y.
\]
Since $q$ has period one, its mean and the interval estimate are
\[
 \int_0^1q(t)\dd t=\frac16,\qquad
 \left|\int_a^bq\left(\frac{y}{h}\right)\dd y-\frac{b-a}{6}\right|=O(h).
\]
Step functions with finitely many interval pieces satisfy the desired limit. Choose such functions $w_n\in L^1(\R)$ with $\|w-w_n\|_1\xrightarrow[n\to\infty]{}0$. Their approximation error is bounded uniformly in $h$ by $\tfrac{\|w-w_n\|_1}{4}$, proving $\tfrac{\mathcal V_h}{h^2}\xrightarrow[h\downarrow0]{}\tfrac{M}{6}$. Atomic sources supported on the grid have $\mathcal V_h=0$ and can be reproduced exactly.
\subsection{A law-adapted remeshing objective}
\label{sec:algorithm}
For the reference density $p$, define $A_p:\R^2\to\R$ by
\[
 A_p(y,x)\coloneqq \int_0^1(1-r)p'(x-ry)\dd r.
\]
Explicitly,
\begin{equation}\label{eq:response-dictionary}
 A_p(y,x)=\begin{cases}
 \displaystyle\frac{F_\lambda(x-y)-F_\lambda(x)+yp(x)}{y^2},&y\ne0,\\[4pt]
 \frac{p'(x)}{2},&y=0.
 \end{cases}
\end{equation}
For each $y\in\R$, $A_p(y,\cdot)\in W^{k,1}(\R)$ for every $k\in\N_0$, and
\[
\norm{D_x^kA_p(y,\cdot)}_1\leq\tfrac12\norm{p^{(k+1)}}_1.
\]
Its first derivative in $y$ is
\[
\partial_yA_p(y,x)=-\int_0^1r(1-r)p''(x-ry)\dd r,
\]
and its second derivative is
\[
\partial_y^2A_p(y,x)=2\mathsf Q_p(\delta_y)(x).
\]
For $\pi\in\cM_f^+(\R^2)$ with marginals $\mu,\widehat\mu$, keep $m,v$ fixed and define $H_\pi\in L^1(\R)$ and $\mathcal J_p(\pi)\in[0,\infty)$ by
\begingroup
\begin{equation*}
H_\pi(x)\coloneqq \int_{\R^2}[A_p(z,x)-A_p(y,x)]\pi(\dd y,\dd z)
\end{equation*}
\endgroup
and $\mathcal J_p(\pi)\coloneqq \|H_\pi\|_1.$ Since $H_\pi=\int_\R A_p(y,\cdot)(\widehat\mu-\mu)(\dd y)$, $\mathcal J_p$ depends only on the marginal difference. The coupling sets the constraints and controls how closely this response norm approximates $W_1$.
\begin{theorem}[Optimization of terminal response]
\label{thm:remeshing-oracle}
Let $\mu\in\cM_f^+(\R)$, $M\coloneqq \mu(\R)$, $m\in\R$, and $v,h>0$. Let $\lambda\in\cP_1(\R)$ have mean $m$ and joint characteristic $v\delta_0+\mu$, and let $p$ be its density. Let $\mathcal C_h\subset\cM_f^+(\R^2)$ be a nonempty family of couplings with first marginal $\mu$, all supported on $|z-y|\leq h$. For each $\pi\in\mathcal C_h$, $\widehat\mu$ is its second marginal and $\widehat\lambda_\pi\in\cP_1(\R)$ has joint characteristic $v\delta_0+\widehat\mu$ and mean $m$. Put $K_j\coloneqq \|g_v^{(j)}\|_1$ for $j\in\N$. For every $\pi\in\mathcal C_h$,
\begin{subequations}\label{eq:remeshing-remainder}
\begin{equation}
\left|W_1(\lambda,\widehat\lambda_\pi)-\mathcal J_p(\pi)\right|\leq R_h,
\end{equation}
where
\begin{equation}
R_h\coloneqq \tfrac12\left(\frac{K_3}{4}+\frac{K_4}{6}+\frac{K_5}{36}\right)M^2h^2.
\end{equation}
\end{subequations}
If every coupling is balanced as in \eqref{eq:barycentric}, the bound improves to
\begin{equation}\label{eq:balanced-oracle-remainder}
 R_h\coloneqq \frac{K_7M^2h^4}{1152}.
\end{equation}
If $\pi_*$ minimizes $\mathcal J_p$ over $\mathcal C_h$, then
\begin{equation}\label{eq:oracle}
 W_1(\lambda,\widehat\lambda_{\pi_*})
 \leq\inf_{\pi\in\mathcal C_h}W_1(\lambda,\widehat\lambda_\pi)+2R_h.
\end{equation}
\end{theorem}
\begin{remark*}
An approximate objective $\widehat{\mathcal J}:\mathcal C_h\to\R$ with uniform error at most $e_q\geq0$, minimized to tolerance $e_o\geq0$, adds $2e_q+e_o$ to the right-hand side.
\end{remark*}
\paragraph{Proof idea.} The affine response approximates the terminal distance uniformly over the allowed plans. Balance improves this remainder from order $h^2$ to order $h^4$. Applying the uniform bound before and after minimization gives the optimization gap. The proof is in Appendix~\ref{proof:thm:remeshing-oracle}. For balanced plans, the optimization gap within $\mathcal C_h$ is fourth order, while the terminal error can remain quadratic. For $I,J\in\N$, source points $y_i\in\R$, masses $M_i>0$, $1\leq i\leq I$, and $\mu\coloneqq \sum_{i=1}^{I}M_i\delta_{y_i}$, choose a finite grid $\{z_j:1\leq j\leq J\}\subset\R$, use $\pi\coloneqq (\pi_{ij})\in[0,\infty)^{I\times J}$ with
\begin{subequations}\label{eq:LP-constraints}
\begin{equation}
\sum_{j=1}^J\pi_{ij}=M_i,
\end{equation}
and, whenever $|z_j-y_i|>h$,
\begin{equation}
\pi_{ij}=0.
\end{equation}
In the balanced case, also require
\begin{equation}
\sum_{j=1}^J(z_j-y_i)\pi_{ij}=0.
\end{equation}
\end{subequations}
For $q\in\N$, quadrature nodes $x_\ell\in\R$, weights $\omega_\ell>0$, $1\leq\ell\leq q$, and auxiliary variables $t\in\R^q$, minimize $\sum_{\ell=1}^q\omega_\ell t_\ell$ subject to
\[
 -t_\ell\leq\sum_{i=1}^I\sum_{j=1}^J\pi_{ij}
       [A_p(z_j,x_\ell)-A_p(y_i,x_\ell)]\leq t_\ell.
\]
This is a linear program. The feasible transport polytope is closed and $0\leq\pi_{ij}\leq M_i$, hence compact. For two feasible matrices $\pi,\widetilde\pi$, the exact objective is continuous since
\[
 |\mathcal J_p(\pi)-\mathcal J_p(\widetilde\pi)|
 \leq\sum_{i=1}^I\sum_{j=1}^J|\pi_{ij}-\widetilde\pi_{ij}|
       \norm{A_p(z_j,\cdot)-A_p(y_i,\cdot)}_1.
\]
It therefore attains a minimum whenever the feasible polytope is nonempty. Two fixed bracketing neighbours determine the unique balanced split \citep{PagesWilbertz2012}. Additional nodes, grid selection, or further constraints create an optimization problem. For a fixed finite source and grid with a nonempty feasible polytope, midpoint quadrature of the exact response functions converges uniformly to this objective as the interval grows and the mesh tends to zero. The minimum values converge, and every cluster point of minimizers solves the exact problem. Appendix~\ref{app:quadrature} gives the proof and a quadrature error bound.
\section{As Gaussian smoothing vanishes}\label{sec:sharp}
For $v\geq0$, let $\mathcal I_v\subset\cP_1(\R)$ be the centered infinitely divisible laws with full-compensation triplet $(0,v,N)$ and $\int_\R y^2N(\dd y)\leq1$. For $\delta>0$, define the worst-case modulus $\mathfrak M(v,\delta)\in[0,\infty)$ by
\begin{equation}\label{eq:modulus-definition}
 \mathfrak M(v,\delta)\coloneqq \sup\left\{W_1(\lambda_0,\lambda_1):
 \lambda_i\in\mathcal I_v,\quad d_{\FM}(y^2N_0,y^2N_1)\leq\delta\right\}.
\end{equation}
Here $N_i$ is the L\'evy measure of $\lambda_i$, $i=0,1$.
\begin{theorem}[Sharp joint dependence on discrepancy and smoothing]
\label{thm:modulus}
There are numerical constants $0<c\leq C<\infty$ such that
\begin{equation*}
 c\,\frac{\delta}{v+\delta^{\frac{2}{3}}}
 \leq\mathfrak M(v,\delta)
 \leq C\,\frac{\delta}{v+\delta^{\frac{2}{3}}},
 \qquad 0\leq v\leq1,\quad 0<\delta\leq1.
\end{equation*}
\end{theorem}
\begin{remark*}
The lower bound holds for pairs with equal total variances.
\end{remark*}
\paragraph{Proof idea.} Adding a small Gaussian gives the upper bound. A pair of centered Poisson laws with matching variances gives the lower bound; the construction and proof are in Appendix~\ref{app:sharp}. The transition occurs at $v\asymp\delta^{\frac{2}{3}}$. Above this scale the error is of order $\tfrac{\delta}{v}$; below it the worst-case order is $\delta^{\frac{1}{3}}$.
\section{Conclusions and extensions}\label{sec:scope}
The first response combines changes in the mean, signed mass changes, and the mean displacement of the weighted characteristic. When it cancels, displacement dispersion determines the leading error. This gives both a quadratic approximation rate and a way to choose a grid approximation by minimizing its terminal response. For balanced local plans, the optimization gap is $O(h^4)$ even when the terminal error itself is of order $h^2$. The appendices also develop metric speed and length, multivariate responses, and state-dependent comparisons. The multivariate response theorem treats finite L\'evy measures with annular support, bounded displacement velocities, and a uniformly positive Gaussian covariance. Its expansion is in the Kantorovich--Rubinstein norm. The state-dependent expansion assumes bounded jump activity, compact displacement support, and a smooth initial law. Stability results cover factorizable state-dependent kernels and infinite-variance laws with stable smoothing.
\appendix
\setcounter{innerthm}{0}
\renewcommand{\theinnerthm}{\thesection.\arabic{innerthm}}
\numberwithin{innerthm}{section}
\numberwithin{equation}{section}
\section{Terminal interpolation and stability}\label{app:terminal}\label{app:extensions}
\subsection{From additive processes to terminal laws}
Fix $T>0$. For $i\in\{0,1\}$, let $x_i\in\R$, let $b_i:[0,T]\to\R$ and $a_i:[0,T]\to[0,\infty)$ be measurable, and let $\nu_i$ be a measurable L\'evy kernel: $\nu_i(t,\cdot)$ is a L\'evy measure for each $t$, and $t\mapsto\nu_i(t,B)$ is measurable for every $B\in\mathcal B(\R)$. Set
\begingroup
\begin{equation*}
m_i\coloneqq x_i+\int_0^Tb_i(t)\dd t,
\qquad
v_i\coloneqq\int_0^Ta_i(t)\dd t,
\end{equation*}
\noindent and
\begin{equation*}
\mu_i(\dd y)\coloneqq\int_0^Ty^2\nu_i(t,\dd y)\dd t,\qquad \text{for }i\in\{0,1\}.
\end{equation*}
\endgroup
The quantities $x_i$, $b_i$, and $a_i$ are the initial state, drift, and Gaussian variance rate. The integrability assumptions are given in \eqref{eq:add-int}. Integrating before weak testing preserves cancellation over time:
\begin{equation}\label{eq:time-cancellation}
 d_{\FM}(\mu_0,\mu_1)
 \leq\int_0^Td_{\FM}(y^2\nu_0(t),y^2\nu_1(t))\dd t.
\end{equation}
\subsection{Terminal characteristics}
For the characteristics above, assume
\begin{equation}\label{eq:add-int}
 \int_0^T\left(|b_i(t)|+a_i(t)+\int_\R y^2\nu_i(t,\dd y)\right)\dd t<\infty.
\end{equation}
For $f\in C^2(\R)$ with bounded $f',f''$, the full-compensation generator is
\begin{equation*}
 \cL_i(t)f(x)\coloneqq b_i(t)f'(x)+\frac{a_i(t)}2 f''(x)
 +\int_\R\bigl[f(x+y)-f(x)-yf'(x)\bigr]\nu_i(t,\dd y).
\end{equation*}
Let $X^i$ be the additive process starting at $x_i$ with these characteristics, and set $N_i\coloneqq\int_0^T\nu_i(t)\dd t$ and $\lambda_i\coloneqq\Law(X_T^i)\in\cP_1(\R)$. The process $X^i=(X_t^i)_{0\leq t\leq T}$ has independent increments and is continuous in probability. Condition \eqref{eq:add-int} makes its fully compensated jump integral square integrable. The terminal L\'evy--Khintchine exponent $\Psi_i:\R\to\mathbb C$ \citep{Sato1999,Applebaum2009} is
\begin{equation*}
 \Psi_i(\xi)\coloneqq im_i\xi-\frac{v_i}2\xi^2
 +\int_\R(e^{i\xi y}-1-i\xi y)N_i(\dd y).
\end{equation*}
We have $\E\left[e^{i\xi X_T^i}\right]=e^{\Psi_i(\xi)}$, $\xi\in\R$, and
\[
\mu_i=y^2N_i\in\cM_f^+(\R),
\qquad
\E\left[X_T^i\right]=m_i,\qquad \operatorname{Var}(X_T^i)=v_i+\mu_i(\R).
\]
$v_i$ is the Gaussian part of the variance. We compare terminal laws using the unit-time Lévy generator of the terminal law of $X^i$:
\begin{equation}\label{eq:terminal-operator}
 \cA_i f\coloneqq m_i f'+\frac{v_i}2 f''
 +\int_\R\bigl[f(\cdot+y)-f-yf'\bigr]N_i(\dd y).
\end{equation}
For $u\in[0,1]$, let $\pi_u$ be the infinitely divisible law with exponent $\psi_u:\R\to\mathbb C$. For a Borel function $h:\R\to\R$ of at most linear growth, define $Q_uh$ by
\begin{subequations}\label{eq:interpolation}
\begin{equation}
\psi_u\coloneqq(1-u)\Psi_1+u\Psi_0,
\end{equation}
\noindent and
\begin{equation}
\mapdef{Q_uh}{\R}{\R}{x}{\displaystyle\int_\R h(x+z)\pi_u(\dd z)}.
\end{equation}
\end{subequations}
The triplet is $((1-u)m_1+um_0,(1-u)v_1+uv_0,(1-u)N_1+uN_0)$. Hence $\pi_u\in\cP_1(\R)$, $\pi_0=\lambda_1$, and $\pi_1=\lambda_0$.
\begin{assumption}[Smoothing along the terminal interpolation]\label{ass:interpolation}
For almost every $u\in(0,1)$, $\pi_u$ has a density $\rho_u\in W^{2,1}(\R)$. The derivative norms are measurable in $u$, and
\begin{subequations}\label{eq:averaged-M}
\begin{equation}
M_1\coloneqq\int_0^1\norm{\rho_u'}_1\dd u<\infty
\end{equation}
\noindent and
\begin{equation}
M_2\coloneqq\int_0^1\norm{\rho_u''}_1\dd u<\infty.
\end{equation}
\end{subequations}
\end{assumption}
\subsection{Two estimates}
\begin{lemma}[Weighted Taylor testing]\label{lem:Taylor}
For $f\in C^3(\R)$ with bounded first three derivatives and $x\in\R$, set
\begin{equation}\label{eq:Phi}
 \mapdef{\Phi_{f,x}}{\R}{\R}{y}{\displaystyle\int_0^1(1-r)f''(x+ry)\dd r}.
\end{equation}
Then $f(x+y)-f(x)-yf'(x)=y^2\Phi_{f,x}(y)$ and
\begin{equation}\label{eq:Phi-norm}
 \norm{\Phi_{f,x}}_{\BL}
 \leq\tfrac12\norm{f''}_\infty+\tfrac16\norm{f'''}_\infty.
\end{equation}
The difference of the two jump terms in \eqref{eq:terminal-operator} is bounded by the right-hand side of \eqref{eq:Phi-norm} times $d_{\FM}(\mu_0,\mu_1)$.
\end{lemma}
\begin{proof}
For $x,y\in\R$, two applications of the fundamental theorem of calculus, followed by Fubini's theorem, give
\begin{align*}
 f(x+y)-f(x)-yf'(x)
 &=y\int_0^1\bigl(f'(x+sy)-f'(x)\bigr)\dd s\\
 &=y^2\int_0^1\int_0^sf''(x+ry)\dd r\dd s\\
 &=y^2\int_0^1(1-r)f''(x+ry)\dd r\ \ \byeq{eq:Phi}{=}\ \ y^2\Phi_{f,x}(y).
\end{align*}
Differentiating in $y$ yields
\[
\partial_y\Phi_{f,x}(y)=\int_0^1r(1-r)f'''(x+ry)\dd r,
\]
\noindent and
\[
\norm{\Phi_{f,x}}_\infty\leq\tfrac12\norm{f''}_\infty,
\qquad
\Lip(\Phi_{f,x})\leq\tfrac16\norm{f'''}_\infty.
\]
Thus, recalling that $\mu_i=y^2N_i$,
\begin{align*}
 \left|\int_\R y^2\Phi_{f,x}(y)(N_0-N_1)(\dd y)\right|
 =|\langle\Phi_{f,x},\mu_0-\mu_1\rangle|
 \ \ &\byeq{eq:metrics}{\leq}\ \ \norm{\Phi_{f,x}}_{\BL}d_{\FM}(\mu_0,\mu_1)\\
 \ \ &\byeq{eq:Phi-norm}{\leq}\ \
 \left(\frac{\norm{f''}_\infty}{2}
       +\frac{\norm{f'''}_\infty}{6}\right)d_{\FM}(\mu_0,\mu_1).
\end{align*}
Both jump integrals converge absolutely because $|y^2\Phi_{f,x}(y)|\leq \tfrac{y^2\|f''\|_\infty}{2}$.
\end{proof}
\begin{lemma}[Convolution regularization]\label{lem:convolution}
Let $\rho\in W^{2,1}(\R)$ be a probability density with finite first moment. For $h:\R\to\R$ with $\Lip(h)\leq1$, the function $q\coloneqq h*\check\rho$ belongs to $C^3(\R)$ and
\begin{subequations}\label{eq:conv-bounds}
\begin{equation}
\norm{q'}_\infty\leq1.
\end{equation}
\noindent For the second derivative,
\begin{equation}
\norm{q''}_\infty\leq\norm{\rho'}_1.
\end{equation}
\noindent For the third derivative,
\begin{equation}
\norm{q'''}_\infty\leq\norm{\rho''}_1.
\end{equation}
\end{subequations}
\end{lemma}
\begin{proof}
Center $h$ at zero. Then $h\in W^{1,\infty}_{\mathrm{loc}}(\R)$ and $Dh\in L^\infty(\R)$ with $\|Dh\|_\infty\leq1$. Moreover,
\[
 |q(x)|\leq\int_\R |x+y|\rho(y)\dd y
 \leq |x|+\int_\R |y|\rho(y)\dd y<\infty.
\]
For $k=0,1,2$, differentiation in distributions gives
\[
 D^{k+1}q=(Dh)*D^k\check\rho.
\]
Then Young's inequality, the bound $\|Dh\|_\infty\leq1$, and the change of variables $x\mapsto-x$ yield
\[
 \norm{D^{k+1}q}_\infty
 \leq\norm{Dh}_\infty\norm{D^k\check\rho}_1
 \leq\norm{D^k\check\rho}_1=\norm{D^k\rho}_1.
\]
If $a\to0$, each derivative satisfies
\[
 \sup_{x\in\R}|D^{k+1}q(x+a)-D^{k+1}q(x)|
 \leq\norm{Dh}_\infty
       \norm{D^k\check\rho(\cdot+a)-D^k\check\rho}_1\xrightarrow[a\to0]{}0.
\]
The weak derivatives are continuous and agree with the classical derivatives.
\end{proof}
\begin{lemma}[Terminal interpolation identity]\label{lem:interpolation}
Under \eqref{eq:add-int}, for every $h\in\cS(\R)$,
\begin{equation}\label{eq:interpolation-identity}
 (Q_1-Q_0)h(x)=\int_0^1(\cA_0-\cA_1)Q_uh(x)\dd u.
\end{equation}
The equality holds at every point. The operators $Q_u$ commute with translations, spatial derivatives, and the terminal operators on these tests.
\end{lemma}
\begin{proof}
Set $\delta\Psi\coloneqq\Psi_0-\Psi_1$. Fourier transformation gives $\widehat{Q_uh}=\widehat h e^{\psi_u}$ and $\partial_u\widehat{Q_uh}=\delta\Psi\widehat h e^{\psi_u}$. The identity $e^{iz}-1-iz=-z^2\int_0^1(1-r)e^{irz}\dd r$ implies
\[
|\Psi_i(\xi)|\leq|m_i||\xi|+\tfrac12(v_i+\mu_i(\R))|\xi|^2,
\]
\noindent and
\[
\operatorname{Re}\psi_u\leq0.
\]
Since $\widehat h\in\cS(\R)$, $C(|\xi|+|\xi|^2)|\widehat h(\xi)|\in L^1(\R)$ dominates the parameter derivative. Hence Fourier inversion, the fundamental theorem of calculus in $u$, and Fubini's theorem give
\begin{align*}
 (Q_1-Q_0)h(x)
 &=\tfrac1{2\pi}\int_\R e^{ix\xi}\widehat h(\xi)
                  \bigl(e^{\Psi_0(\xi)}-e^{\Psi_1(\xi)}\bigr)\dd\xi\\
 &=\int_0^1\tfrac1{2\pi}\int_\R
       e^{ix\xi}\widehat h(\xi)\delta\Psi(\xi)e^{\psi_u(\xi)}\dd\xi\dd u\\
 &\byeq{eq:terminal-operator}{=}\int_0^1(\cA_0-\cA_1)Q_uh(x)\dd u.
\end{align*}
Translations and derivatives commute with convolution. The same quadratic bound lets us apply Fubini to each jump operator and proves the remaining commutation identities.
\end{proof}
\subsection{The main theorem}
\begin{theorem}[Linear terminal stability]\label{thm:terminal}
Assume \eqref{eq:add-int} and Assumption~\ref{ass:interpolation}. Then
\begin{equation*}
 W_1(\lambda_0,\lambda_1)
 \leq |m_0-m_1|+\frac{M_1}{2}|v_0-v_1|
 +\left(\frac{M_1}{2}+\frac{M_2}{6}\right)d_{\FM}(\mu_0,\mu_1).
\end{equation*}
\end{theorem}
\begin{proof}
For $h\in\cS$ with $\Lip(h)\leq1$, put $q_u\coloneqq Q_uh$. Lemmas~\ref{lem:Taylor} and~\ref{lem:convolution} give
\begin{equation}\label{eq:calc-first-generator-bound}
 \begin{aligned}
 |(\cA_0-\cA_1)q_u(0)|
 &\byeq{eq:terminal-operator}{\leq}
     |m_0-m_1|\norm{q_u'}_\infty
     +\tfrac12|v_0-v_1|\norm{q_u''}_\infty
     +|\langle\Phi_{q_u,0},\mu_0-\mu_1\rangle|\\
 &\byeq{eq:Phi-norm}{\leq}|m_0-m_1|\norm{q_u'}_\infty
     +\tfrac12|v_0-v_1|\norm{q_u''}_\infty\\
 &\qquad+\left(\tfrac12\norm{q_u''}_\infty
       +\tfrac16\norm{q_u'''}_\infty\right)d_{\FM}(\mu_0,\mu_1)\\
 &\byeq{eq:conv-bounds}{\leq}|m_0-m_1|
     +\tfrac12\norm{\rho_u'}_1|v_0-v_1|\\
 &\qquad+\left(\tfrac12\norm{\rho_u'}_1
       +\tfrac16\norm{\rho_u''}_1\right)d_{\FM}(\mu_0,\mu_1).
 \end{aligned}
\end{equation}
Integrating in $u$ gives
\begin{align*}
 |\langle h,\lambda_0-\lambda_1\rangle|
 &\byeq{eq:interpolation-identity}{\leq}
       \int_0^1|(\cA_0-\cA_1)q_u(0)|\dd u\\
 &\byeq{eq:calc-first-generator-bound}{\leq}|m_0-m_1|
       +\frac{|v_0-v_1|}{2}\int_0^1\norm{\rho_u'}_1\dd u\\
 &\qquad\qquad+d_{\FM}(\mu_0,\mu_1)\int_0^1
       \left(\frac{\norm{\rho_u'}_1}{2}
             +\frac{\norm{\rho_u''}_1}{6}\right)\dd u\\
 &\byeq{eq:averaged-M}{=}|m_0-m_1|+\frac{M_1}{2}|v_0-v_1|
       +\left(\frac{M_1}{2}+\frac{M_2}{6}\right)d_{\FM}(\mu_0,\mu_1).
\end{align*}
For a centered Lipschitz test, Lemma~\ref{lem:approximation} provides locally uniform approximations bounded by $|x|$. The finite terminal first moments justify dominated convergence. Taking the supremum in \eqref{eq:KR} completes the proof.
\end{proof}
\begin{corollary}[Joint Gaussian and jump comparison]\label{cor:joint}
Under the assumptions of Theorem~\ref{thm:terminal}, define $\Gamma_i\coloneqq v_i\delta_0+\mu_i$. Then
\begin{equation}\label{eq:joint}
 W_1(\lambda_0,\lambda_1)
 \leq |m_0-m_1|+
 \left(\frac{M_1}{2}+\frac{M_2}{6}\right)d_{\FM}(\Gamma_0,\Gamma_1).
\end{equation}
\end{corollary}
\begin{proof}
The measures $\Gamma_i\in\cM_f^+(\R)$ satisfy
\[
\Phi_{f,x}(0)=\tfrac12f''(x),
\]
so
\[
\cA_i f(x)=m_i f'(x)+\langle\Phi_{f,x},\Gamma_i\rangle.
\]
In the preceding proof, replace the last two terms by this single pairing.
\end{proof}
The joint measure retains cancellation between Gaussian and small-jump variance.
\subsection{Gaussian constants and temporal cancellation}
Gaussian differentiation and $x=\sqrt v\,z$ give
\begin{subequations}\label{eq:gauss-L1}
\begin{equation}
\norm{g_v'}_1=\sqrt{\frac{2}{\pi}}v^{-\frac{1}{2}}
\end{equation}
\noindent and
\begin{equation}
\norm{g_v''}_1=\sqrt{\frac{8}{\pi e}}v^{-1}.
\end{equation}
\end{subequations}
Indeed,
\[
g_v'(x)=-\frac{x}{v}g_v(x),\qquad
g_v''(x)=\left(\frac{x^2}{v^2}-\frac1v\right)g_v(x),
\]
so the change of variables $x=\sqrt v\,z$ gives
\[
\norm{g_v'}_1=v^{-\frac{1}{2}}\E\left[|G|\right]=\sqrt{\frac{2}{\pi}}\,v^{-\frac{1}{2}},\qquad
\norm{g_v''}_1=v^{-1}\E\left[|G^2-1|\right].
\]
Since $\E\left[G^2-1\right]=0$, symmetry, the identity $(z\varphi(z))'=(1-z^2)\varphi(z)$, and $\varphi(1)=\tfrac{1}{\sqrt{2\pi e}}$ yield
\[
\begin{aligned}
 \E\left[|G^2-1|\right]
 &=2\int_{\{z\in\R:|z|>1\}}(z^2-1)\varphi(z)\dd z\\
 &=4\int_1^\infty(z^2-1)\varphi(z)\dd z
 =-4[z\varphi(z)]_1^\infty=\sqrt{\frac{8}{\pi e}}.
\end{aligned}
\]
\begin{corollary}[Positive terminal Gaussian variances]\label{cor:Gaussian}
Under \eqref{eq:add-int}, suppose $v_0,v_1>0$. Theorem~\ref{thm:terminal} and Corollary~\ref{cor:joint} hold with $M_1$ and $M_2$ replaced, respectively, by $\tfrac{2\sqrt{\frac{2}{\pi}}}{\sqrt{v_0}+\sqrt{v_1}}$ and
\begingroup
\begin{equation*}
\sqrt{\frac8{\pi e}}\times 
 \begin{cases}
 \frac{\log v_0-\log v_1}{v_0-v_1},&\text{if }v_0\ne v_1,\\
 v_0^{-1},&\text{if }v_0=v_1.
 \end{cases}
\end{equation*}
\endgroup
When $v_0=v_1=v>0$,
\begin{equation*}
 W_1(\lambda_0,\lambda_1)
 \leq |m_0-m_1|+
 \left(\frac{\sqrt{\tfrac{2}{\pi}}}{2\sqrt v}+\frac{\sqrt{\tfrac{8}{\pi e}}}{6v}\right)d_{\FM}(\mu_0,\mu_1).
\end{equation*}
\end{corollary}
\begin{proof}
Write $v_u\coloneqq(1-u)v_1+uv_0$ and $\rho_u=g_{v_u}*\eta_u$ for a probability law $\eta_u$. For $v_0\ne v_1$, Young's inequality, \eqref{eq:gauss-L1}, and the substitution $s=v_1+u(v_0-v_1)$ give
\[
\begin{aligned}
M_1\ \ \byeq{eq:averaged-M}{=}\ \  \int_0^1\norm{g_{v_u}'*\eta_u}_1\dd u
 &\leq\int_0^1\norm{g_{v_u}'}_1\dd u
 \ \ \byeq{eq:gauss-L1}{=}\ \  \sqrt{\frac{2}{\pi}}\int_0^1v_u^{-\frac{1}{2}}\dd u\\
 &=
 \frac{2\sqrt{\tfrac{2}{\pi}}(\sqrt{v_0}-\sqrt{v_1})}{v_0-v_1}
 =\frac{2\sqrt{\tfrac{2}{\pi}}}{\sqrt{v_0}+\sqrt{v_1}},
\end{aligned}
\]
\noindent and
\[
\begin{aligned}
M_2\ \ \byeq{eq:averaged-M}{=}\ \ \int_0^1\norm{g_{v_u}''*\eta_u}_1\dd u
 &\leq\int_0^1\norm{g_{v_u}''}_1\dd u
 \\ \ \  &\byeq{eq:gauss-L1}{=}\ \ \sqrt{\frac{8}{\pi e}}\int_0^1v_u^{-1}\dd u
 =
 \sqrt{\frac8{\pi e}}\,
       \frac{\log v_0-\log v_1}{v_0-v_1}.
\end{aligned}
\]
For $v_0=v_1$, the integrands are constant and the values are $v_0^{-\frac{1}{2}}$ and $v_0^{-1}$.
\end{proof}
\begin{example}[Degenerate chronological variance]
\label{ex:chronological}
Take $T=1$ and $a_0(t)\coloneqq2\one_{\left(\frac{1}{2},1\right]}(t)$ and $a_1(t)\coloneqq2\one_{\left[0,\frac{1}{2}\right]}(t)$. Both terminal variances equal one. The chronological mixture
\[
 V_t\coloneqq\int_0^t a_0(s)\dd s+\int_t^1a_1(s)\dd s
\]
satisfies
\[
V_t=\begin{cases}1-2t,&0\leq t\leq\tfrac12,\\
                  2t-1,&\tfrac12\leq t\leq1,
       \end{cases}
\]
so $\min_{0\leq t\leq1}V_t=0$, while $v_0=v_1=1.$ Thus Corollary~\ref{cor:Gaussian} applies even though the chronological variance vanishes. If $m_0=m_1$ and $N_0=N_1$, then $\Psi_0=\Psi_1$ and $\lambda_0=\lambda_1$.
\end{example}
\begin{proposition}[Variance-matched small-jump replacement]\label{ex:replacement}
Let $m\in\R$, $v>0$, and $r\geq0$. Let $N,K$ be L\'evy measures with $\int_\R y^2N(\dd y)<\infty$ and $K$ supported on $[-r,r]$. Compare the laws $\lambda_0,\lambda_1$ with terminal triplets $(m,v,N+K)$ and $(m,v+\delta,N)$, where $\delta\coloneqq\int_\R y^2K(\dd y)<\infty$. Then
\begin{equation*}
 W_1(\lambda_0,\lambda_1)
 \leq\frac{\sqrt{\tfrac{8}{\pi e}}}{6v}\int_\R |y|^3K(\dd y)
 \leq\frac{\sqrt{\tfrac{8}{\pi e}}}{6v}\,r\delta.
\end{equation*}
If also $\int_\R y^3K(\dd y)=0$, in particular if $K$ is symmetric, then
\begin{equation*}
 W_1(\lambda_0,\lambda_1)
 \leq\frac{\norm{g_v'''}_1}{24}\int_\R y^4K(\dd y)
 \leq\frac{\norm{g_v'''}_1}{24}\,r^2\delta.
\end{equation*}
\end{proposition}
\begin{proof}
Let $\mathcal D\coloneqq\cA_0-\cA_1$. Since $\delta=\int_\R y^2K(\dd y)$, Taylor's formula gives
\[
\begin{aligned}
\mathcal Df(x)
 &=
   \int_\R \left[f(x+y)-f(x)-yf'(x)-\frac{y^2}{2}f''(x)\right]K(\dd y)\\
 &=\tfrac12\int_\R y^3\int_0^1(1-r)^2f'''(x+ry)\dd r\,K(\dd y).
\end{aligned}
\]
Then the triangle inequality and $\int_0^1(1-r)^2\dd r=\tfrac{1}{3}$ imply
\[
\norm{\mathcal Df}_\infty
 \leq\frac{\norm{f'''}_\infty}{2}
          \int_\R |y|^3K(\dd y)\int_0^1(1-r)^2\dd r
 =\frac{\norm{f'''}_\infty}{6}\int_\R |y|^3K(\dd y).
\]
If $\int_\R y^3K(\dd y)=0$, one more Taylor term gives
\[
\mathcal Df(x)=\tfrac16\int_\R y^4\int_0^1(1-r)^3f^{(4)}(x+ry)
                  \dd r\,K(\dd y),
\]
hence
\[
\norm{\mathcal Df}_\infty\leq\frac{\norm{f^{(4)}}_\infty}{24}\int_\R y^4K(\dd y).
\]
Every interpolated law has a Gaussian factor of variance $v$, hence
\[
\norm{(Q_uh)'''}_\infty\leq\norm{g_v''}_1=\frac{\sqrt{\tfrac{8}{\pi e}}}{v},
\]
\noindent and
\[
\norm{(Q_uh)^{(4)}}_\infty\leq\norm{g_v'''}_1
 \quad\text{when }\Lip(h)\leq1.
\]
Apply Lemmas~\ref{lem:interpolation} and~\ref{lem:approximation}, then use $|y|^3\leq r y^2$ and $y^4\leq r^2y^2$ on the support of $K$.
\end{proof}
The joint metric also gives $d_{\FM}(y^2K,\delta\delta_0)\leq\int_\R |y|^3K(\dd y)$. The Taylor expansion removes the second-derivative term and, when the third moment vanishes, the third-derivative term. See also \citet{MariucciReiss2018} on Gaussian replacement.
\subsection{A quadratic remainder for the first-order response}
We begin with affine changes of the mean and joint characteristic. Section~\ref{sec:tangents} treats moving and splitting atoms.
\begin{theorem}[Uniform affine response]
\label{thm:response}
Assume \eqref{eq:add-int} and $\min(v_0,v_1)\geq v_*>0$. Set $\delta m\coloneqq m_0-m_1$ and $\kappa\coloneqq\Gamma_0-\Gamma_1$, and let $d\coloneqq\|\kappa\|_{\BL^*}$. For $j=1,\ldots,5$, set $K_j\coloneqq\|g_{v_*}^{(j)}\|_1$. For each one-Lipschitz test $h:\R\to\R$, define the first-order response at the reference law by
\begin{equation}\label{eq:response-definition}
 \mathcal R(h)\coloneqq\delta m\,(Q_0h)'(0)
                +\int_\R \Phi_{Q_0h,0}(y)\kappa(\dd y).
\end{equation}
Then
\begin{equation}\label{eq:response-bound}
 \begin{split}
 \sup_{\substack{h:\R\to\R\\\Lip(h)\leq1}}\bigl|\E\left[h(X_T^0)\right]-\E\left[h(X_T^1)\right]-\mathcal R(h)\bigr|
 \leq\tfrac12\bigg[&K_1|\delta m|^2
       +\left(K_2+\frac{K_3}{3}\right)|\delta m|d\\
       &+\left(\frac{K_3}{4}+\frac{K_4}{6}
                            +\frac{K_5}{36}\right)d^2\bigg].
 \end{split}
\end{equation}
\end{theorem}
\begin{proof}
Put $\mathcal D\coloneqq\cA_0-\cA_1$ and $\delta\Psi\coloneqq\Psi_0-\Psi_1$. For Schwartz tests, $\partial_u^j\widehat{Q_uh}=(\delta\Psi)^j \widehat h e^{\psi_u}$, $j=1,2$, with integrable majorants from the polynomial bounds. Taylor's formula in $u$ and Fourier inversion give
\begin{equation}\label{eq:second-interpolation}
 (Q_1-Q_0)h=\mathcal DQ_0h+
                   \int_0^1(1-u)\mathcal D^2Q_uh\dd u.
\end{equation}
For $f\in C^2(\R)$ with bounded second derivative, define $\mathcal B_\kappa f:\R\to\R$ by $\mathcal B_\kappa f(x)\coloneqq\langle\Phi_{f,x},\kappa\rangle$. For smooth $f$ with bounded derivatives through order six, $\mathcal D=\delta m D+\mathcal B_\kappa$, all these operators commute, and
\begin{equation}\label{eq:calc-first-insertion-bound}
 \norm{\mathcal B_\kappa D^jf}_\infty
 \ \ \byeq{eq:Phi-norm}{\leq}\ \ d\left(\tfrac12\norm{f^{(j+2)}}_\infty+
                  \tfrac16\norm{f^{(j+3)}}_\infty\right),\qquad\text{for }0\leq j\leq3.
\end{equation}
Applying this estimate twice gives
\begin{align*}
 \norm{\mathcal B_\kappa^2f}_\infty
 \ \ &\byeq{eq:calc-first-insertion-bound}{\leq}\ \
 d\left(\tfrac12\norm{\mathcal B_\kappa f''}_\infty
                     +\tfrac16\norm{\mathcal B_\kappa f'''}_\infty\right)\\
 \ \ &\byeq{eq:calc-first-insertion-bound}{\leq}\ \
 \frac{d^2}{2}\left(\tfrac12\norm{f^{(4)}}_\infty
                              +\tfrac16\norm{f^{(5)}}_\infty\right)
 +\frac{d^2}{6}\left(\tfrac12\norm{f^{(5)}}_\infty
                              +\tfrac16\norm{f^{(6)}}_\infty\right)\\
 &=d^2\left(\tfrac14\norm{f^{(4)}}_\infty
              +\tfrac16\norm{f^{(5)}}_\infty
              +\tfrac1{36}\norm{f^{(6)}}_\infty\right).
\end{align*}
Using $\mathcal D^2f=(\delta m)^2f''+ 2\delta m\mathcal B_\kappa f'+\mathcal B_\kappa^2f$ gives
\begin{equation}\label{eq:D2-bound}
 \begin{split}
 \norm{\mathcal D^2f}_\infty\leq{}&|\delta m|^2\norm{f''}_\infty
 +|\delta m|d\left(\norm{f'''}_\infty+\tfrac13\norm{f^{(4)}}_\infty\right)\\
 &+d^2\left(\tfrac14\norm{f^{(4)}}_\infty+
              \tfrac16\norm{f^{(5)}}_\infty+\tfrac1{36}\norm{f^{(6)}}_\infty\right).
 \end{split}
\end{equation}
Since $\pi_u=g_{v_*}*\eta_u$ with $\eta_u\in\cP_1$, $\|D^{j+1}Q_uh\|_\infty\leq K_j$ for $1\leq j\leq5$ and $\Lip(h)\leq1$. Explicitly,
\[
K_j=v_*^{-\frac{j}{2}}\E\left[|\operatorname{He}_j(G)|\right]<\infty,
\]
where
\[
\operatorname{He}_j(x)\coloneqq(-1)^je^{\frac{x^2}{2}}\frac{\dd^j}{\dd x^j}e^{-\frac{x^2}{2}}.
\]
Using $\int_0^1(1-u)\dd u=\tfrac{1}{2}$, we bound the remainder by
\begin{align*}
 \big|(Q_1-Q_0)h(0)-\mathcal R(h)\big|
 \ \ &\byeq{eq:second-interpolation}{\leq}\ \
       \int_0^1(1-u)\norm{\mathcal D^2Q_uh}_\infty\dd u
 \\ \ \ &\byeq{eq:D2-bound}{\leq}\ \ \int_0^1(1-u)\bigg[
      K_1|\delta m|^2+
       \left(K_2+\tfrac{K_3}{3}\right)|\delta m|d
+\left(\tfrac{K_3}{4}+\tfrac{K_4}{6}
                  +\tfrac{K_5}{36}\right)d^2\bigg]\dd u
 \\ &=
       \tfrac12\bigg[K_1|\delta m|^2+
       \left(K_2+\tfrac{K_3}{3}\right)|\delta m|d
+\left(\tfrac{K_3}{4}+\tfrac{K_4}{6}
                  +\tfrac{K_5}{36}\right)d^2\bigg].
\end{align*}
For the approximants of Lemma~\ref{lem:approximation}, $Dh_n\xrightarrow[n\to\infty]{}Dh$ weakly-$*$ in $L^\infty$: distributional convergence and the uniform derivative bound imply convergence against all $L^1$ tests. After Gaussian convolution, each derivative of $Q_0h_n$ converges. The uniform derivative bounds and the finiteness of $|\kappa|$ justify passage to the limit in the response integral. Finite terminal first moments give convergence of the expectations.
\end{proof}
\begin{corollary}[First-order Wasserstein response]\label{cor:metric-response}
Under Theorem~\ref{thm:response}, let $p\coloneqq\rho_0$ be the density of the second law $\lambda_1$. Define
\begin{equation*}
 H_{\delta m,\kappa}(z)\coloneqq-\delta m\,p(z)
   +\int_\R\int_0^1(1-r)p'(z-ry)\dd r\,\kappa(\dd y).
\end{equation*}
Then $H_{\delta m,\kappa}\in L^1(\R)$ and
\begin{equation}\label{eq:metric-response}
 \begin{split}
 \left|W_1(\lambda_0,\lambda_1)-\norm{H_{\delta m,\kappa}}_1\right|
 \ \ \byeq{eq:response-bound}{\leq}\ \ \tfrac12\bigg[&K_1|\delta m|^2
 +\left(K_2+\frac{K_3}{3}\right)|\delta m|\norm\kappa_{\BL^*}\\
 &+\left(\frac{K_3}{4}+\frac{K_4}{6}+\frac{K_5}{36}\right)
                         \norm\kappa_{\BL^*}^2\bigg].
 \end{split}
\end{equation}
On an admissible affine path with $(m_\varepsilon,\Gamma_\varepsilon)=(m_1,\Gamma_1) +\varepsilon(\dot m,\dot\Gamma)$ and Gaussian variance bounded below,
\begin{equation}\label{eq:directional-response}
 W_1(\lambda_\varepsilon,\lambda_1)
 =|\varepsilon|\|H_{\dot m,\dot\Gamma}\|_1+O(\varepsilon^2).
\end{equation}
The leading coefficient is positive whenever $(\dot m,\dot\Gamma)\ne(0,0)$.
\end{corollary}
\begin{proof}
Fubini gives $\|H_{\delta m,\kappa}\|_1\leq|\delta m|+ \tfrac{\|p'\|_1\|\kappa\|_{\mathrm{TV}}}{2}<\infty$. Integration by parts in \eqref{eq:response-definition} gives $\mathcal R(h)=-\int_\R h'(z)H_{\delta m,\kappa}(z)\dd z$. Lemma~\ref{lem:approximation} extends the identity to Lipschitz tests. Every measurable $w:\R\to\R$ with $|w|\leq1$ is the almost-everywhere derivative of $h(x)\coloneqq\int_0^xw(s)\dd s$. Hence $L^1$--$L^\infty$ duality gives
\[
\begin{aligned}
 \sup_{\substack{h:\R\to\R\\\Lip(h)\leq1}}|\mathcal R(h)|&=\sup_{\substack{h:\R\to\R\\\Lip(h)\leq1}}\left|\int_\R h'(z)H_{\delta m,\kappa}(z)\dd z\right|\\
 &=\sup_{\substack{w\in L^\infty(\R)\\\norm w_\infty\leq1}}\left|\int_\R w(z)H_{\delta m,\kappa}(z)\dd z\right|
 =\norm{H_{\delta m,\kappa}}_1.
\end{aligned}
\]
The reverse triangle inequality and Theorem~\ref{thm:response} prove \eqref{eq:metric-response}. Homogeneity gives \eqref{eq:directional-response}. Positive definiteness is the case $J=0$ of Proposition~\ref{prop:tangent-quotient}, whose Fourier proof uses the response formula.
\end{proof}
\subsection{Notation for the extensions}
Fix $d\in\N$. For $\sigma\in\cM_f(\R^d)$ with $\sigma(\R^d)=0$ and $\int_{\R^d}|x|\,|\sigma|(\dd x)<\infty$, set
\[
 \norm\sigma_{\mathrm{KR}}\coloneqq\sup_{\substack{h:\R^d\to\R\\\Lip(h)\leq1,\ h(0)=0}}
 \left|\int_{\R^d}h(x)\sigma(\dd x)\right|.
\]
Thus $W_1(\lambda,\widetilde\lambda)=\|\lambda-\widetilde\lambda\|_{\mathrm{KR}}$. We identify a density with its measure. For a symmetric positive definite matrix $V\in\R^{d\times d}$, the centered Gaussian density with covariance $V$ is
\[
 g_V(x)\coloneqq\frac{e^{-\frac{x^\top V^{-1}x}{2}}}{(2\pi)^{\frac{d}{2}}\sqrt{\det V}},\qquad\text{for }x\in\R^d.
\]
\subsection{Centered Lipschitz approximation}\label{app:analytic}
\begin{lemma}\label{lem:approximation}
For every one-Lipschitz $h:\R^d\to\R$ with $h(0)=0$, there exist $h_n\in C_c^\infty(\R^d)$ such that $h_n(0)=0$, $\Lip(h_n)\leq1$, $|h_n(x)|\leq|x|$, and $h_n\xrightarrow[n\to\infty]{}h$ locally uniformly.
\end{lemma}
\begin{proof}
Choose a smooth compactly supported probability density $\eta$ and a smooth cutoff $0\leq\chi\leq1$ equal to one on the unit ball. For $R,\varepsilon>0$, let $h_R\coloneqq\max\{-R,\min\{h,R\}\}$ and $\eta_\varepsilon(x)\coloneqq\varepsilon^{-d}\eta\left(\tfrac{x}{\varepsilon}\right)$, and define $f_{R,\varepsilon}\coloneqq h_R*\eta_\varepsilon-(h_R*\eta_\varepsilon)(0)$. Then $f_{R,\varepsilon}(0)=0$ and $\Lip(f_{R,\varepsilon})\leq1$. We also have $\|f_{R,\varepsilon}\|_\infty\leq2R$ and $\|f_{R,\varepsilon}-h_R\|_\infty\leq2\varepsilon\int_{\R^d}|z|\eta(z)\dd z$. For $n\in\N$, set
\[
 h_n(x)\coloneqq\frac{f_{n,n^{-1}}(x)\chi\left(\tfrac{x}{n^3}\right)}
                   {1+2n^{-2}\norm{\nabla\chi}_\infty}.
\]
The product rule and the Lipschitz and uniform bounds above give
\begin{align*}
 \norm{\nabla\left(f_{n,n^{-1}}\chi\left(\frac{\cdot}{n^3}\right)\right)}_\infty
 &\leq
   \norm{\nabla f_{n,n^{-1}}}_\infty\norm\chi_\infty
   +n^{-3}\norm{f_{n,n^{-1}}}_\infty\norm{\nabla\chi}_\infty\\
 &\leq
   1+n^{-3}\norm{f_{n,n^{-1}}}_\infty\norm{\nabla\chi}_\infty
 \leq
   1+2n^{-2}\norm{\nabla\chi}_\infty.
\end{align*}
Thus $\Lip(h_n)\leq1$, and $h_n(0)=0$ implies $|h_n(x)|\leq|x|$. On each fixed ball, $h_n\xrightarrow[n\to\infty]{}h$ uniformly because the cutoff eventually equals one, the denominator tends to one, and the mollification error is $O(n^{-1})$. For $\lambda\in\cP_1(\R^d)$, dominated convergence gives $$\int_{\R^d} h_n(x)\lambda(\dd x)\xrightarrow[n\to\infty]{}\int_{\R^d} h(x)\lambda(\dd x).$$
\end{proof}
\section{Proofs of the response formulas}\label{app:response-proofs}
\subsection{Proof of Theorem~\ref{thm:transport-response}.}\label{proof:thm:transport-response}
With $K_j\coloneqq\|g_{v_*}^{(j)}\|_1$, take
\begin{subequations}\label{eq:transport-constant}
\begin{equation}
C\coloneqq\tfrac12\left[K_1|\alpha|^2+
       \left(K_2+\frac{K_3}{3}\right)|\alpha|d_1+
       \left(\frac{K_3}{4}+\frac{K_4}{6}+\frac{K_5}{36}\right)d_1^2\right]
       +\frac{K_3}{24}M_2(\Pi),
\end{equation}
\noindent and
\begin{equation}
d_1\coloneqq\norm\eta_{\BL^*}+\int_{\R^2}|z|\Pi(\dd y,\dd z).
\end{equation}
\end{subequations}
\begin{proof}
Put $\kappa_\varepsilon\coloneqq\Gamma_\varepsilon-\Gamma$. For $\|f\|_{\BL}\leq1$, the triangle inequality gives
\begin{align*}
 |\langle f,\kappa_\varepsilon\rangle|
 &\byeq{eq:transport-lift}{=}
   \left|\int_{\R^2}[f(y+\varepsilon z)-f(y)]\Pi(\dd y,\dd z)
                   +\varepsilon\langle f,\eta\rangle\right|\\
 &\leq\int_{\R^2}|f(y+\varepsilon z)-f(y)|\Pi(\dd y,\dd z)
       +|\varepsilon|\,|\langle f,\eta\rangle|\\
 &\byeq{eq:metrics}{\leq}|\varepsilon|
     \left(\int_{\R^2}|z|\Pi(\dd y,\dd z)+\norm\eta_{\BL^*}\right),
\end{align*}
so $\|\kappa_\varepsilon\|_{\BL^*}\leq|\varepsilon|d_1$. For $q(x)\coloneqq\int_\R h(x+w)p(w)\dd w$ and $\Lip(h)\leq1$, differentiating the Taylor field with respect to its jump variable gives
\[
\partial_y\Phi_{q,0}(y)=\int_0^1r(1-r)q'''(ry)\dd r,
\qquad 
\norm{\partial_y^2\Phi_{q,0}}_\infty
                  \leq\tfrac1{12}\norm{q^{(4)}}_\infty\leq \frac{K_3}{12}.
\]
Taylor's formula gives, pointwise in $(y,z)$,
\begin{equation}\label{eq:calc-first-displacement-Taylor}
 \Phi_{q,0}(y+\varepsilon z)-\Phi_{q,0}(y)
 -\varepsilon z\partial_y\Phi_{q,0}(y)
 =\varepsilon^2z^2\int_0^1(1-t)
       \partial_y^2\Phi_{q,0}(y+t\varepsilon z)\dd t.
\end{equation}
Integrating against $\Pi$, using $\|\partial_y^2\Phi_{q,0}\|_\infty\leq \tfrac{K_3}{12}$, and then $\int_0^1(1-t)\dd t=\tfrac{1}{2}$, we obtain
\begin{align*}
 &\left|\langle\Phi_{q,0},\kappa_\varepsilon\rangle
 -\varepsilon\left(\int_\R \Phi_{q,0}(y)\eta(\dd y)
                  +\int_\R \partial_y\Phi_{q,0}(y)J(\dd y)\right)\right|\\
 &\qquad\qquad\byeq{eq:calc-first-displacement-Taylor}{\leq}
 \varepsilon^2\int_{\R^2}z^2\int_0^1(1-t)
       \left|\partial_y^2\Phi_{q,0}(y+t\varepsilon z)\right|\dd t\,\Pi(\dd y,\dd z)\\
 &\qquad\qquad\leq
       \frac{\varepsilon^2K_3}{12}\int_{\R^2}z^2\Pi(\dd y,\dd z)\int_0^1(1-t)\dd t
       =\frac{\varepsilon^2K_3M_2(\Pi)}{24}.
\end{align*}
Apply Theorem~\ref{thm:response} to $\lambda_\varepsilon,\lambda$. Its remainder is at most
\[
 \frac{\varepsilon^2}{2}\left[K_1|\alpha|^2+
 \left(K_2+\frac{K_3}{3}\right)|\alpha|d_1+
 \left(\frac{K_3}{4}+\frac{K_4}{6}+\frac{K_5}{36}\right)d_1^2\right].
\]
The resulting first-order functional is
\[
 \alpha q'(0)+\int_\R \Phi_{q,0}(y)\eta(\dd y)+
             \int_\R \partial_y\Phi_{q,0}(y)J(\dd y)
       \byeq{eq:transport-cdf}{=}-\int_\R h'(x)\mathsf T_p(\alpha,\eta,J)(x)\dd x.
\]
Indeed, for smooth $h$ and $j=1,2,3$,
\[
 q^{(j)}(t)=(-1)^{j-1}\int_\R h'(x)p^{(j-1)}(x-t)\dd x.
\]
Substituting into $\Phi_{q,0}$ and $\partial_y\Phi_{q,0}$ gives the three terms in \eqref{eq:transport-cdf}. Lemma~\ref{lem:approximation} extends the identity to Lipschitz $h$. Moreover,
\[
 \norm{\mathsf T_p}_1\leq|\alpha|
       +\tfrac12\norm{p'}_1\norm\eta_{\mathrm{TV}}
       +\tfrac16\norm{p''}_1\norm J_{\mathrm{TV}}<\infty.
\]
Since $F_{\lambda_\varepsilon}-F_\lambda\in L^1$ and every bounded measurable $w:\R\to\R$, $|w|\leq1$, is the almost-everywhere derivative of $h(x)\coloneqq\int_0^xw(t)\dd t$, the uniform test remainder is \eqref{eq:transport-L1}. The identity $W_1=\|F_{\lambda_\varepsilon}-F_\lambda\|_1$ and the reverse triangle inequality give \eqref{eq:transport-W1}.
\end{proof}
The conclusion also holds, with a larger constant for $|\varepsilon|\leq1$, if \eqref{eq:transport-lift} contains a remainder $r_\varepsilon\in\cM_f(\R)$ with $\|r_\varepsilon\|_{\BL^*}\leq C_0\varepsilon^2$ for a constant $C_0\geq0$ and the resulting characteristic remains admissible. In \eqref{eq:transport-constant}, replace $d_1$ by $d_1+C_0$ and add $C_0\left(\tfrac{K_1}{2}+\tfrac{K_2}{6}\right)$. The same bounded-Lipschitz pairing proves the extension. It covers, in particular, a raw atom of fixed intensity, $N_\varepsilon\coloneqq c\delta_{a+\varepsilon}$, with $c>0$ and $a\ne0$: its weighted tangent has $J=ca^2\delta_a$ and $\eta=2ca\delta_a$, and its residual is $2ca\varepsilon(\delta_{a+\varepsilon}-\delta_a)+ c\varepsilon^2\delta_{a+\varepsilon}=O(\varepsilon^2)$ in $\BL^*$. For $a\ne0$ and $\varepsilon>0$, set
\[
q_\varepsilon\coloneqq\frac{\delta_{a+\varepsilon}-\delta_a}{\varepsilon}, 
\qquad
f_\varepsilon(x)\coloneqq\left(\frac\varepsilon2-
                 \left|x-a-\frac\varepsilon2\right|\right)_+,
\qquad
t_+\coloneqq\max\{t,0\}.
\]
Then $\|f_\varepsilon\|_{\BL}=1+\tfrac{\varepsilon}{2}$ and $\left\langle f_\varepsilon,q_\varepsilon-q_{\frac{\varepsilon}{2}}\right\rangle=-1$, so
\begin{equation}\label{eq:non-BL-tangent}
 \norm{q_\varepsilon-q_{\frac{\varepsilon}{2}}}_{\BL^*}
       \geq\frac1{1+\tfrac{\varepsilon}{2}},\qquad\text{for }\varepsilon>0.
\end{equation}
Thus $q_\varepsilon$ has no limit in $\BL^*$, whereas $q_\varepsilon\xrightarrow[\varepsilon\downarrow0]{}-D_y\delta_a$ in $\mathcal D'$ and \eqref{eq:transport-L1} gives convergence of the terminal CDF quotients in $L^1$.
\subsection{Proof of Proposition~\ref{prop:tangent-quotient}.}\label{proof:prop:tangent-quotient}
\begin{proof}
Define the kernel $k:\R^2\to\mathbb C$ by
\[
 k(\xi,y)\coloneqq-\xi^2\int_0^1(1-r)e^{ir\xi y}\dd r
 =\begin{cases}\frac{e^{i\xi y}-1-i\xi y}{y^2},&y\ne0,\\
               -\frac{\xi^2}{2},&y=0,
   \end{cases}
\]
\noindent and define the first exponent variation $\dot\Psi:\R\to\mathbb C$ by
\[
 \dot\Psi(\xi)\coloneqq i\alpha\xi+\int_\R k(\xi,y)\eta(\dd y)
                                  +\int_\R \partial_y k(\xi,y)J(\dd y).
\]
The kernels satisfy $|k|\leq\tfrac{\xi^2}{2}$ and $|\partial_y k|\leq\tfrac{|\xi|^3}{6}$. Fourier transformation of \eqref{eq:transport-cdf} gives
\begin{subequations}\label{eq:tangent-Fourier}
\begin{equation}
i\xi\widehat{\mathsf T_p}(\xi) \byeq{eq:transport-cdf}{=}\dot\Psi(-\xi)\widehat p(\xi),
\end{equation}
\noindent and
\begin{equation}
\dot\Psi''(\xi) = -\int_\R e^{i\xi y}\eta(\dd y)-i\xi\int_\R e^{i\xi y}J(\dd y) \byeq{eq:transport-tangent}{=}-\widehat\tau(-\xi).
\end{equation}
\end{subequations}
Indeed $\partial_\xi^2k=-e^{i\xi y}$ and $\partial_\xi^2\partial_yk=-i\xi e^{i\xi y}$. These and the first $\xi$ derivatives are uniformly bounded in $y$ on compact $\xi$ intervals. Since $\eta,J$ are finite, differentiation under the integrals is valid. The second expression is the negative of the Fourier transform of the tempered distribution $\tau$, taken with a positive sign in the exponent. Since $\widehat p(\xi)=e^{\Psi(-\xi)}\ne0$, the identity $\mathsf T_p=0$ implies $\dot\Psi(-\xi)=0$ for all $\xi\ne0$ and hence for all $\xi$ by continuity. Thus $\alpha=\tfrac{\dot\Psi'(0)}i=0$ and $\widehat\tau(-\xi)=-\dot\Psi''(\xi)=0$, and Fourier uniqueness in $\cS'(\R)$ gives $\tau=0$. Conversely, if $\alpha=0$ and $\tau=0$, then $\dot\Psi''=0$, $\dot\Psi(0)=0$, and $\dot\Psi'(0)=i\alpha=0$, so $\dot\Psi=0$. Equation~\eqref{eq:tangent-Fourier} gives $\widehat{\mathsf T_p}(\xi)=0$ for $\xi\ne0$. Continuity gives the value at zero, and Fourier uniqueness in $L^1$ gives $\mathsf T_p=0$. The kernel characterization makes the norm independent of the representation and positive definite. Linearity gives the remaining norm properties. Finally, Fubini's theorem, Young's inequality, and $\zeta(\R)=1$ give
\[
 \| (\alpha,\tau)\|_{\lambda*\zeta} \byeq{eq:observable-norm}{=}\norm{\mathsf T_{p*\zeta}(\alpha,\eta,J)}_1
 =\norm{\mathsf T_p(\alpha,\eta,J)*\zeta}_1
 \leq\norm{\mathsf T_p(\alpha,\eta,J)}_1\,\zeta(\R) \byeq{eq:observable-norm}{=}\| (\alpha,\tau)\|_\lambda.
\]
\end{proof}
\subsection{Proof of Theorem~\ref{thm:two-jet}.}\label{proof:thm:two-jet}
With $K_j\coloneqq\|g_{v_*}^{(j)}\|_1$, take
\begin{equation}\label{eq:two-jet-constant}
 \begin{aligned}
 C_3\coloneqq{}&\frac{K_4M_3}{120}
 +M_2\left(\frac{|\alpha|K_4}{24}
          +\frac{\norm\eta_{\mathrm{TV}}K_5}{48}
          +\frac{M_1K_6}{144}\right)\\
 &+\frac{|\alpha|^3K_2}{6}
   +\frac{|\alpha|^2\norm\eta_{\mathrm{TV}}K_3}{4}
 +\left(\frac{|\alpha|^2M_1}{12}
         +\frac{|\alpha|\norm\eta_{\mathrm{TV}}^2}{8}\right)K_4
 +\left(\frac{|\alpha|\norm\eta_{\mathrm{TV}}M_1}{12}
         +\frac{\norm\eta_{\mathrm{TV}}^3}{48}\right)K_5\\
 &+\left(\frac{|\alpha|M_1^2}{72}
         +\frac{\norm\eta_{\mathrm{TV}}^2M_1}{48}\right)K_6
   +\frac{\norm\eta_{\mathrm{TV}}M_1^2K_7}{144}
   +\frac{M_1^3K_8}{1296}.
 \end{aligned}
\end{equation}
\begin{proof}
Let $\Psi,\Psi_\varepsilon$ be the terminal characteristic exponents, and let $Q_u$ have exponent $(1-u)\Psi+u\Psi_\varepsilon$, $0\leq u\leq1$. The generator difference is
\[
 \mathcal D_\varepsilon f(x)\coloneqq\varepsilon\alpha f'(x)
 +\varepsilon\int_\R\Phi_{f,x}(y)\eta(\dd y)
 +\int_{\R^2}[\Phi_{f,x}(y+\varepsilon z)-\Phi_{f,x}(y)]\Pi(\dd y,\dd z).
\]
\noindent and we set
\[
 \mathcal Bf(x)\coloneqq\tfrac12\int_\R\partial_y^2\Phi_{f,x}(y)S(\dd y).
\]
For $f\in C_b^{j+2}(\R)$,
\[
\partial_y^j\Phi_{f,x}(y)
 =\int_0^1r^j(1-r)f^{(j+2)}(x+ry)\dd r,
\]
\noindent and
\[
\int_0^1r^j(1-r)\dd r=\frac1{(j+1)(j+2)}.
\]
Taylor's formula in $y$ and the total variation bound for $\eta$ give the operator bounds
\[
\norm{\mathcal D_\varepsilon f}_\infty
 \leq|\varepsilon|\left(|\alpha|\norm{f'}_\infty
       +\frac{\norm\eta_{\mathrm{TV}}}{2}\norm{f''}_\infty
       +\frac{M_1}{6}\norm{f'''}_\infty\right),
\]
\noindent and
\[
\norm{\mathcal Af}_\infty
 \leq|\alpha|\norm{f'}_\infty
       +\frac{\norm\eta_{\mathrm{TV}}}{2}\norm{f''}_\infty
       +\frac{M_1}{6}\norm{f'''}_\infty.
\]
\noindent The first-order remainder satisfies
\begin{equation}\label{eq:calc-second-operator}
\norm{(\mathcal D_\varepsilon-\varepsilon\mathcal A)f}_\infty
 \leq\frac{\varepsilon^2M_2}{24}\norm{f^{(4)}}_\infty.
\end{equation}
\noindent At second order, we obtain
\[
\norm{(\mathcal D_\varepsilon-\varepsilon\mathcal A-\varepsilon^2\mathcal B)f}_\infty
 \leq\frac{|\varepsilon|^3M_3}{120}\norm{f^{(5)}}_\infty.
\]
By translation invariance, these operators commute with spatial derivatives and with each other. Hence
\[
\mathcal D_\varepsilon^2-\varepsilon^2\mathcal A^2
 =(\mathcal D_\varepsilon-\varepsilon\mathcal A)
   (\mathcal D_\varepsilon+\varepsilon\mathcal A).
\]
Then the quadratic remainder bound and the first-order estimates yield
\[
\begin{aligned}
\norm{(\mathcal D_\varepsilon^2-\varepsilon^2\mathcal A^2)f}_\infty
 &\byeq{eq:calc-second-operator}{\leq}
 \frac{\varepsilon^2M_2}{24}
 \norm{(\mathcal D_\varepsilon+\varepsilon\mathcal A)f^{(4)}}_\infty\\
 &\leq|\varepsilon|^3M_2\left(
 \frac{|\alpha|}{12}\norm{f^{(5)}}_\infty+
 \frac{\norm\eta_{\mathrm{TV}}}{24}\norm{f^{(6)}}_\infty+
 \frac{M_1}{72}\norm{f^{(7)}}_\infty\right).
\end{aligned}
\]
Applying the first operator bound three times also gives
\begin{align*}
 \norm{\mathcal D_\varepsilon^3f}_\infty
 \leq |\varepsilon|^3\bigg[&
 |\alpha|^3\norm{f^{(3)}}_\infty
 +\frac{3|\alpha|^2\norm\eta_{\mathrm{TV}}}{2}\norm{f^{(4)}}_\infty
 +\left(\frac{|\alpha|^2M_1}{2}
       +\frac{3|\alpha|\norm\eta_{\mathrm{TV}}^2}{4}\right)
                                      \norm{f^{(5)}}_\infty\\
 &+\left(\frac{|\alpha|\norm\eta_{\mathrm{TV}}M_1}{2}
       +\frac{\norm\eta_{\mathrm{TV}}^3}{8}\right)
                                      \norm{f^{(6)}}_\infty
 +\left(\frac{|\alpha|M_1^2}{12}
       +\frac{\norm\eta_{\mathrm{TV}}^2M_1}{8}\right)
                                      \norm{f^{(7)}}_\infty\\
 &+\frac{\norm\eta_{\mathrm{TV}}M_1^2}{24}\norm{f^{(8)}}_\infty
 +\frac{M_1^3}{216}\norm{f^{(9)}}_\infty\bigg].
\end{align*}
For $h\in\cS(\R)$, differentiating the interpolated Fourier multiplier three times and using Taylor's integral formula gives
\[
 (Q_1-Q_0)h=\mathcal D_\varepsilon Q_0h
 +\tfrac12\mathcal D_\varepsilon^2Q_0h
 +\tfrac12\int_0^1(1-u)^2\mathcal D_\varepsilon^3Q_uh\dd u.
\]
The multiplier derivatives are polynomially bounded, so $\widehat h\in\cS(\R)$ justifies differentiation and Fourier inversion. Every interpolation law has a Gaussian factor of variance $v_*$. For $\Lip(h)\leq1$ and $j\geq1$,
\[
 \norm{(Q_uh)^{(j)}}_\infty
 \leq\norm{h'}_\infty\norm{g_{v_*}^{(j-1)}}_1
 \leq\norm{g_{v_*}^{(j-1)}}_1.
\]
Subtracting $\varepsilon\mathcal AQ_0h+ \varepsilon^2\left(\mathcal B+\tfrac{\mathcal A^2}{2}\right)Q_0h$, applying the three bounds above, and using $\int_0^1(1-u)^2\dd u=\tfrac{1}{3}$ gives a uniform error at most $C_3|\varepsilon|^3$, with the constant in \eqref{eq:two-jet-constant}. The adjoint identities, commutation with $D_x$, and integration by parts give
\[
\begin{aligned}
\mathcal AQ_0h(0)
 =\int_\R h(x)(\mathcal A^*p)(x)\dd x
 \byeq{eq:transport-cdf}{=}\int_\R h(x)D_xT_1(x)\dd x
 =-\int_\R h'(x)T_1(x)\dd x.
\end{aligned}
\]
The second characteristic variation gives
\[
\begin{aligned}
\mathcal BQ_0h(0)
 =\int_\R h(x)(\mathcal B^*p)(x)\dd x
 \byeq{eq:second-observable}{=}\int_\R h(x)D_x\mathsf Q_p(S)(x)\dd x
 =-\int_\R h'(x)\mathsf Q_p(S)(x)\dd x.
\end{aligned}
\]
The composition of the first variation with itself gives
\[
\begin{aligned}
\mathcal A^2Q_0h(0)
 =\int_\R h(x)((\mathcal A^*)^2p)(x)\dd x
 =\int_\R h(x)D_x(\mathcal A^*T_1)(x)\dd x
 =-\int_\R h'(x)(\mathcal A^*T_1)(x)\dd x.
\end{aligned}
\]
These follow from $\mathcal B^*p=D_x\mathsf Q_p(S)$ and $\mathcal A^*p=D_xT_1$. Gaussian convolution gives $p^{(j)}\in L^1(\R)$ for every $j\geq0$, so $T_1,T_2\in L^1(\R)$. Centered Lipschitz approximation and CDF duality, as in Theorem~\ref{thm:transport-response}, yield \eqref{eq:two-jet-L1}. Finally, $|\|f\|_1-\|g\|_1|\leq\|f-g\|_1$ yields \eqref{eq:two-jet-W1}.
\end{proof}
\subsection{Proof of Theorem~\ref{thm:scalar-second}.}\label{proof:thm:scalar-second}
The remainder after these two terms is bounded by
\begin{equation}\label{eq:scalar-sign-remainder}
 C_3|\varepsilon|^3+2\varepsilon^2
  \int_{\{x\in\R:0<|T_1(x)|\leq|\varepsilon T_2(x)|\}}|T_2(x)|\dd x.
\end{equation}
\begin{proof}
For $a,b,t\in\R$ and $a\ne0$,
\[
 0\leq |a+tb|-|a|-t\operatorname{sgn}(a)b
 \leq2|tb|\one_{\{0<|a|\leq|tb|\}}.
\]
For real $f,g\in L^1(\R)$, set $Z\coloneqq\{x\in\R:f(x)=0\}$ and $\operatorname{sgn}(0)\coloneqq0$. Integrating this inequality and treating $Z$ separately gives
\begin{subequations}\label{eq:L1-scalar}
\begin{equation}
\norm{f+t g}_1=\norm f_1+t\int_{Z^c}\operatorname{sgn}(f(x))g(x)\dd x
       +|t|\int_Z|g(x)|\dd x+r(t),
\end{equation}
where
\begin{equation}
|r(t)|\leq2|t|\int_{\{x\in\R:0<|f(x)|\leq|tg(x)|\}}|g(x)|\dd x.
\end{equation}
\end{subequations}
The last integrand converges to zero almost everywhere and is bounded by $|g|$, so the integral tends to zero. Write $p=g_v*\rho$, where $v>0$ and $\rho\in\cP_1(\R)$. For $x,b,u\in\R$ and $z\coloneqq x+ib\in\mathbb C$,
\[
 |g_v(z-u)|=(2\pi v)^{-\frac{1}{2}}
 \exp\!\left(\frac{b^2-(x-u)^2}{2v}\right).
\]
Each complex derivative is this Gaussian times a polynomial in $z-u$. For every compact $K\subset\mathbb C$ and integer $j\geq0$,
\[
 \sup_{z\in K,\ u\in\R}|g_v^{(j)}(z-u)|<\infty.
\]
These bounds remain uniform after the real shifts $ry$ in \eqref{eq:transport-cdf}. Integrating against $\rho$, $\eta$, and $J$ extends $T_1$ to an entire function. By Proposition~\ref{prop:tangent-quotient}, $(\alpha,\tau)\ne(0,0)$ implies $T_1\not\equiv0$. Its real zeros are isolated, hence have Lebesgue measure zero, and $\|T_1\|_1>0$. Apply \eqref{eq:L1-scalar} with $(f,g,t)=(T_1,T_2,\varepsilon)$, multiply by $|\varepsilon|$, and use \eqref{eq:two-jet-W1} to obtain \eqref{eq:scalar-second} and \eqref{eq:scalar-sign-remainder}.
\end{proof}
\subsection{Proof of Theorem~\ref{thm:cancellation}.}\label{proof:thm:cancellation}
If $M_3<\infty$, the $L^1$ remainder is at most
\begin{equation}\label{eq:cancelled-third}
 \frac{K_4M_3}{120}|\varepsilon|^3+
                        \frac{K_7M_2^2}{1152}\varepsilon^4.
\end{equation}
If $\Pi$ is invariant under $(y,z)\mapsto(y,-z)$ and $M_4<\infty$, then $J=0$, $\eta=0$, and the remainder is at most
\begin{equation}\label{eq:cancelled-fourth}
 \varepsilon^4\left(\frac{K_5M_4}{720}+
                              \frac{K_7M_2^2}{1152}\right).
\end{equation}
Here $K_j\coloneqq\|g_{v_*}^{(j)}\|_1$.
\begin{proof}
The identity $\eta-D_yJ=0$ means $$\int_\R\phi(y)\eta(\dd y)+\int_\R\phi'(y)J(\dd y)=0$$ for $\phi\in C_c^\infty(\R)$. Choose $\chi\in C_c^\infty(\R)$ with $0\leq\chi\leq1$ and $\chi=1$ near zero. For $R>0$, put $\chi_R(y)\coloneqq\chi\left(\tfrac{y}{R}\right)$ and apply this identity to $\phi(y)\coloneqq\chi_R(y)\Phi_{f,x}(y)$:
\[
\begin{aligned}
 \int_\R\chi_R(y)\Phi_{f,x}(y)\eta(\dd y)
 +\int_\R\chi_R(y)\partial_y\Phi_{f,x}(y)J(\dd y)
 =-\int_\R\chi_R'(y)\Phi_{f,x}(y)J(\dd y)\xrightarrow[R\to\infty]{}0.
\end{aligned}
\]
The right-hand side is bounded by $R^{-1}\|\chi'\|_\infty\|\Phi_{f,x}\|_\infty\|J\|_{\mathrm{TV}}$. Dominated convergence applies to the left-hand side. Hence $\mathcal A=0$. Taylor expansion then gives
\[
\norm{\mathcal D_\varepsilon f}_\infty
 \leq\frac{\varepsilon^2M_2}{24}\norm{f^{(4)}}_\infty,
\qquad 
\norm{\mathcal D_\varepsilon^2f}_\infty
 \leq\frac{\varepsilon^4M_2^2}{576}\norm{f^{(8)}}_\infty.
\]
For $\Lip(h)\leq1$, the interpolation identity, the triangle inequality, and the Gaussian derivative bound give
\begin{align*}
 \norm{(Q_1-Q_0)h-\mathcal D_\varepsilon Q_0h}_\infty
 \ \ &\byeq{eq:second-interpolation}{=}\ \ 
 \left\|\int_0^1(1-u)\mathcal D_\varepsilon^2Q_uh\dd u\right\|_\infty\\
 &\leq
 \int_0^1(1-u)\norm{\mathcal D_\varepsilon^2Q_uh}_\infty\dd u\\
 &\leq
 \frac{\varepsilon^4K_7M_2^2}{576}\int_0^1(1-u)\dd u
 =\frac{\varepsilon^4K_7M_2^2}{1152}.
\end{align*}
For $q\coloneqq Q_0h$ and $F(y)\coloneqq\Phi_{q,0}(y)$,
\[
 F(y+u)-F(y)-uF'(y)-\frac{u^2}{2}F''(y)
 =u^2\int_0^1(1-t)[F''(y+tu)-F''(y)]\dd t.
\]
Using separately $|F''(y+tu)-F''(y)|\leq2\|F''\|_\infty$ and $|tu|\|F'''\|_\infty$, together with $\|F''\|_\infty\leq \tfrac{K_3}{12}$ and $\|F'''\|_\infty\leq \tfrac{K_4}{20}$, bounds the remainder in the Taylor expansion of the generator difference by
\begin{equation}\label{eq:min-remainder}
 \varepsilon^2\int_{\R^2} z^2
       \min\left\{\frac{K_3}{12},\ \frac{|\varepsilon z|K_4}{120}\right\}\,\Pi(\dd y,\dd z).
\end{equation}
After dividing \eqref{eq:min-remainder} by $\varepsilon^2$, dominated convergence with respect to the finite measure $z^2\Pi$ shows that the limit is zero. This proves \eqref{eq:cancelled-expansion}. Using the second term in the minimum gives \eqref{eq:cancelled-third}. If $\Pi$ is symmetric in $z$, every integrable odd function of $z$ integrates to zero. Thus $J=0$, $\eta=0$, and the cubic Taylor term vanishes. The fourth Taylor remainder is bounded by
\[
 \frac{\varepsilon^4M_4}{4!}
 \norm{\partial_y^4\Phi_{q,0}}_\infty
 \leq\frac{\varepsilon^4M_4K_5}{720},
\]
which gives \eqref{eq:cancelled-fourth}. Finally, Fourier transformation gives
\begin{equation}\label{eq:calc-second-fourier}
\widehat{\mathsf Q_p(S)}(\xi)
 \byeq{eq:second-observable}{=}
 \tfrac12(i\xi)^3\widehat p(\xi)
       \int_0^1r^2(1-r)\widehat S(r\xi)\dd r.
\end{equation}
Then continuity at the origin and dominated convergence imply
\[
\int_0^1r^2(1-r)\widehat S(r\xi)\dd r
 \xrightarrow[\xi\to0]{}
 M_2\int_0^1r^2(1-r)\dd r=\frac{M_2}{12}.
\]
Since $\widehat p(0)=1$, this Fourier transform is nonzero for sufficiently small nonzero $\xi$ when $M_2>0$. Thus $\|\mathsf Q_p(S)\|_1>0$.
\end{proof}
\section{Geometry and second-order responses}\label{app:geometry}
\subsection{Continuity and finite-dimensional local geometry}
The response varies continuously with the reference law and determines the metric speed on finite-dimensional parameter charts.
\begin{proposition}[Reference continuity]\label{prop:reference-continuity}
Let $p,q$ be Gaussian-smoothed infinitely divisible reference densities. For $\alpha\in\R$ and any representation $\tau=\eta-D_yJ$ with $\eta,J\in\cM_f(\R)$,
\begin{equation}\label{eq:reference-continuity}
 \big|\| (\alpha,\tau)\|_p-\| (\alpha,\tau)\|_q\big| \leq{}|\alpha|\norm{p-q}_1 +\tfrac12\norm\eta_{\mathrm{TV}}\norm{p'-q'}_1 +\tfrac16\norm J_{\mathrm{TV}}\norm{p''-q''}_1.
\end{equation}
A density subscript denotes its law. If $p=g_v*\rho$ and $q=g_v*\sigma$, with $\rho,\sigma\in\cP_1$ and $v>0$, then the right-hand side is at most
\begin{subequations}\label{eq:reference-W1}
\begin{equation}
\left(|\alpha|K_1+\tfrac12\norm\eta_{\mathrm{TV}}K_2+
             \tfrac16\norm J_{\mathrm{TV}}K_3\right)W_1(\rho,\sigma),
\end{equation}
where
\begin{equation}
K_j\coloneqq\norm{g_v^{(j)}}_1.
\end{equation}
\end{subequations}
\end{proposition}
\begin{proof}
Subtract the two formulas in \eqref{eq:transport-cdf}. Fubini, translation invariance of $L^1$, and the reverse triangle inequality give the first bound. For real random variables $U,V$ with laws $\rho,\sigma$, respectively, and $j=0,1,2$, Fubini's theorem, the fundamental theorem of calculus, and translation invariance of the $L^1$ norm give
\begin{align*}
 \norm{p^{(j)}-q^{(j)}}_1
 &\leq\E\left[\norm{g_v^{(j)}(\cdot-U)-g_v^{(j)}(\cdot-V)}_1\right]\\
 &\leq\E\left[|U-V|\int_0^1
       \norm{g_v^{(j+1)}(\cdot-V-r(U-V))}_1\dd r\right]\\
 &=
       \norm{g_v^{(j+1)}}_1\E\left[|U-V|\right]
 =K_{j+1}\E\left[|U-V|\right].
\end{align*}
Taking the infimum over couplings proves \eqref{eq:reference-W1}.
\end{proof}
Take the infimum in \eqref{eq:reference-continuity} over all representations by finite signed measures $\eta,J$. With
\[
 \| (\alpha,\tau)\|_{\rm rep}
 \coloneqq|\alpha|+\inf_{\substack{\eta,J\in\cM_f(\R)\\\eta-D_yJ=\tau}}
             (\norm\eta_{\mathrm{TV}}+\norm J_{\mathrm{TV}}),
\]
the observable norm is bounded by $\max\left\{1,\tfrac{\|p'\|_1}{2},\tfrac{\|p''\|_1}{6}\right\}\| (\alpha,\tau)\|_{\rm rep}$. The representation cost is a norm: if representations $\tau=\eta_n-D_yJ_n$ have $\|\eta_n\|_{\mathrm{TV}}+ \|J_n\|_{\mathrm{TV}}\xrightarrow[n\to\infty]{}0$, then, for every $f\in C_c^\infty(\R)$,
\[
 |\langle f,\tau\rangle|
 \leq\norm f_\infty\norm{\eta_n}_{\mathrm{TV}}
       +\norm{f'}_\infty\norm{J_n}_{\mathrm{TV}}\xrightarrow[n\to\infty]{}0,
\]
hence $\tau=0$. Homogeneity and the triangle inequality follow by scaling and adding representations.
\begin{theorem}[Metric speed and local identification on atomic charts]
\label{thm:atomic-geometry}
For $k\in\N$, let
\[
 \theta\coloneqq(m,v,w_1,\ldots,w_k,a_1,\ldots,a_k)\in\R^{2k+2}
\]
parameterize laws $\lambda_\theta$ with mean $m$, density $p_\theta$, and
\[
\Gamma_\theta\coloneqq v\delta_0+\sum_{j=1}^kw_j\delta_{a_j},
\]
with $v>0$ and $w_j>0$, where the nonzero $a_j$ are distinct and ordered. On every compact subset of this parameter domain there are $c,C,r>0$ such that, whenever both parameters are in the subset and $|\theta'-\theta|\leq r$,
\begin{equation}\label{eq:chart-bilip}
 c|\theta'-\theta|\leq W_1(\lambda_{\theta'},\lambda_\theta)
                         \leq C|\theta'-\theta|.
\end{equation}
For a $C^2$ curve $[0,1]\ni s\mapsto\theta_s$ in this domain, write $\lambda_s\coloneqq\lambda_{\theta_s}$ and $p_s\coloneqq p_{\theta_s}$. Dots denote derivatives with respect to $s$. Set
\[
\eta_s\coloneqq\dot v_s\delta_0+\sum_{j=1}^k\dot w_{j,s}\delta_{a_{j,s}},
\qquad 
J_s\coloneqq\sum_{j=1}^kw_{j,s}\dot a_{j,s}\delta_{a_{j,s}},
\qquad 
\tau_s\coloneqq\eta_s-D_yJ_s.
\]
Its Wasserstein metric speed and length are
\begin{subequations}\label{eq:metric-length}
\begin{equation}
|\dot\lambda_s|_{W_1}=\| (\dot m_s,\tau_s)\|_{\lambda_s}
\end{equation}
\noindent and
\begin{equation}
\operatorname{Len}_{W_1}(\lambda_{[0,1]})
        =\int_0^1\| (\dot m_s,\tau_s)\|_{\lambda_s}\dd s.
\end{equation}
\end{subequations}
The integral also bounds the endpoint distance.
\end{theorem}
\begin{proof}
Fix a compact parameter set $K$ in the stated domain. Its Gaussian variances and masses are bounded below by positive constants, and its atom positions are bounded, separated, and bounded away from zero. For $u\coloneqq\theta'-\theta\in\R^{2k+2}$, its coordinates are denoted by $(u_m,u_v,u_{w_1},\ldots,u_{w_k},u_{a_1},\ldots,u_{a_k})$. Define $\eta_\theta(u),J_\theta(u)\in\cM_f(\R)$ and $\mathcal T_\theta:\R^{2k+2}\to L^1(\R)$ by
\[
\eta_\theta(u)\coloneqq u_v\delta_0+\sum_{j=1}^ku_{w_j}\delta_{a_j},
\qquad 
J_\theta(u)\coloneqq\sum_{j=1}^kw_ju_{a_j}\delta_{a_j},
\qquad 
\mathcal T_\theta u\coloneqq\mathsf T_{p_\theta}
                   (u_m,\eta_\theta(u),J_\theta(u))\in L^1(\R).
\]
The difference between the characteristics is
\[
\Gamma_{\theta'}-\Gamma_\theta
 =u_v\delta_0+\sum_{j=1}^kw_j(\delta_{a_j+u_{a_j}}-\delta_{a_j})
             +\sum_{j=1}^ku_{w_j}\delta_{a_j}+r_u,
\]
where
\[
r_u\coloneqq\sum_{j=1}^ku_{w_j}(\delta_{a_j+u_{a_j}}-\delta_{a_j}).
\]
Since $\|f\|_{\BL}\leq1$ implies $\Lip(f)\leq1$, the inequality $2|ab|\leq a^2+b^2$ gives
\[
\begin{aligned}
 \norm{r_u}_{\BL^*} &\byeq{eq:metrics}{=}\sup_{\substack{f\in\BL(\R)\\\norm f_{\BL}\leq1}} \left|\sum_{j=1}^ku_{w_j}[f(a_j+u_{a_j})-f(a_j)]\right|\\
 &\leq\sum_{j=1}^k|u_{w_j}|\,|u_{a_j}|\leq \tfrac12\sum_{j=1}^k(|u_{w_j}|^2+|u_{a_j}|^2) \leq\frac{|u|^2}{2}.
\end{aligned}
\]
Apply Theorem~\ref{thm:transport-response} with $\varepsilon=|u|$, velocities $\tfrac{u_{a_j}}{|u|}$, and its weak-remainder extension. Uniformly for $\theta,\theta'\in K$ and $|u|\leq1$,
\[
 \norm{F_{\lambda_{\theta'}}-F_{\lambda_\theta}
                     -\mathcal T_\theta u}_1\leq B_K|u|^2
\]
for a finite constant $B_K$. The bound also holds for $u=0$. The direction $(u_m,\eta_\theta(u)-D_yJ_\theta(u))$ is zero only when $u=0$. Indeed, smooth test functions supported near each atom can prescribe their value and first derivative there, proving linear independence of $\delta_0,\delta_{a_j},D_y\delta_{a_j}$. Also, $w_j>0$. By Proposition~\ref{prop:tangent-quotient}, $\|\mathcal T_\theta u\|_1>0$ for $u\ne0$. The map $(\theta,u)\mapsto\mathcal T_\theta u$ is continuous in $L^1$: the jump intensities $\tfrac{w_j}{a_j^2}$ and compensation shifts $-\tfrac{w_j}{a_j}$ vary continuously, their compound-Poisson laws vary in $W_1$, and Gaussian derivatives and their translations vary in $L^1$. Compactness gives
\[
a_K\coloneqq\min_{\substack{\theta\in K,\ u\in\R^{2k+2}\\|u|=1}}\norm{\mathcal T_\theta u}_1>0,
\]
\noindent and
\[
A_K\coloneqq\max_{\substack{\theta\in K,\ u\in\R^{2k+2}\\|u|=1}}\norm{\mathcal T_\theta u}_1<\infty.
\]
Consequently
\[
 (a_K-B_K|u|)|u|\leq W_1(\lambda_{\theta'},\lambda_\theta)
                   \leq(A_K+B_K|u|)|u|.
\]
Choose $r\leq1$ with $B_Kr\leq \tfrac{a_K}{2}$ to obtain \eqref{eq:chart-bilip}. For a $C^2$ curve, the same expansion gives
\[
 \frac{F_{\lambda_{s+t}}-F_{\lambda_s}}{t}
 \xrightarrow[t\to0]{L^1}\mathcal T_{\theta_s}\dot\theta_s
 =\mathsf T_{p_s}(\dot m_s,\eta_s,J_s).
\]
The derivative is continuous, so the anchored curve $s\mapsto F_{\lambda_s}-F_{\lambda_0}$ is $C^1$ in $L^1$. By definition,
\[
|\dot\lambda_s|_{W_1}\coloneqq\lim_{t\to0}
       \frac{W_1(\lambda_{s+t},\lambda_s)}{|t|},
\]
\noindent and
\[
\operatorname{Len}_{W_1}(\lambda_{[0,1]})
 \coloneqq\sup_{\substack{n\in\N,\ (s_0,\ldots,s_n)\in[0,1]^{n+1}\\0=s_0<\cdots<s_n=1}}
             \sum_{j=1}^nW_1(\lambda_{s_j},\lambda_{s_{j-1}}).
\]
The CDF isometry gives the speed. The fundamental theorem of calculus bounds every partition sum by the integral of the derivative norm. Riemann sums and continuity of the derivative give the reverse bound. This gives \eqref{eq:metric-length}.
\end{proof}
\subsection{Length of admissible non-atomic curves}
The next lift allows non-atomic weighted measures and unbounded initial positions.
\begin{theorem}[Observable speed for absolutely continuous lifts]
\label{thm:lift-length}
Let $(E,\mathcal E,\gamma)$ be a finite positive measure space and $w_0\in L^\infty(E,\mathcal E,\gamma;\R)$, $y_0:E\to\R$ be measurable, and, for $s\in[0,1]$, set
\[
w_s\coloneqq w_0+\int_0^s a_t\dd t,
\qquad 
y_s\coloneqq y_0+\int_0^s b_t\dd t,
\]
with
\[
a,b\in L^1([0,1];L^\infty(\gamma)),
\]
where $L^\infty(\gamma)\coloneqq L^\infty(E,\mathcal E,\gamma;\R)$ and $L^1([0,1];L^\infty(\gamma))$ consists of strongly measurable maps whose $L^\infty(\gamma)$ norm is integrable in $s$. The integrals are Bochner integrals, and $w_s\geq0$ holds $\gamma$-almost everywhere. Let $m,v:[0,1]\to\R$ be absolutely continuous, with $v_s\geq v_*>0$, and let $\lambda_s$ have mean $m_s$ and joint characteristic $\Gamma_s\coloneqq v_s\delta_0+(y_s)_\#(w_s\gamma)$. Write $p_s$ for its density, and let $\dot m_s,\dot v_s$ denote the almost-everywhere derivatives. For an integrable $f:E\to\R$, $f\gamma$ denotes the signed measure $A\mapsto\int_Af\dd\gamma$ on $(E,\mathcal E)$. For almost every $s$, put
\[
\eta_s\coloneqq\dot v_s\delta_0+(y_s)_\#(a_s\gamma),
\qquad 
J_s\coloneqq(y_s)_\#(w_sb_s\gamma),
\qquad 
\tau_s\coloneqq\eta_s-D_yJ_s.
\]
Then $F_{\lambda_s}-F_{\lambda_0}$ is absolutely continuous in $L^1$, with Bochner derivative $\mathsf T_{p_s}(\dot m_s,\eta_s,J_s)$ almost everywhere. In particular,
\begin{subequations}\label{eq:general-length}
\begin{equation}
|\dot\lambda_s|_{W_1}=\|(\dot m_s,\tau_s)\|_{\lambda_s}
 \quad\text{a.e.},
\end{equation}
\noindent and
\begin{equation}
\operatorname{Len}_{W_1}(\lambda_{[0,1]})
 =\int_0^1\|(\dot m_s,\tau_s)\|_{\lambda_s}\dd s.
\end{equation}
\end{subequations}
The integral bounds $W_1(\lambda_0,\lambda_1)$.
\end{theorem}
\begin{proof}
Set
\[
W\coloneqq\sup_{s\in[0,1]}\norm{w_s}_\infty
 \leq\norm{w_0}_\infty+\int_0^1\norm{a_t}_\infty\dd t<\infty,
\]
and $M_\gamma\coloneqq\gamma(E).$ For $0\leq s<t\leq1$ and $\|f\|_{\BL}\leq1$, decompose the integrand and use $w_s\leq W$ together with the Bochner integral representations:
\begin{equation}\label{eq:calc-first-lift-increments}
 \begin{aligned}
 \left|\int_E\bigl(w_tf(y_t)-w_sf(y_s)\bigr)\dd\gamma\right|
 &=
 \left|\int_E\bigl[(w_t-w_s)f(y_t)
               +w_s(f(y_t)-f(y_s))\bigr]\dd\gamma\right|\\
 &\quad\leq
       \int_E|w_t-w_s|\dd\gamma+\int_Ew_s|y_t-y_s|\dd\gamma\\
 &\quad\leq
       M_\gamma\bigl(\norm{w_t-w_s}_\infty
                         +W\norm{y_t-y_s}_\infty\bigr)\\
 &\quad\leq
       M_\gamma\int_s^t(\norm{a_r}_\infty+W\norm{b_r}_\infty)\dd r.
 \end{aligned}
\end{equation}
Thus, with $K_j\coloneqq\|g_{v_*}^{(j)}\|_1$ and $L\coloneqq \tfrac{K_1}{2}+\tfrac{K_2}{6}$, absolute continuity of $m$ and $v$ yields
\begin{align*}
 W_1(\lambda_t,\lambda_s)
 \ \ &\byeq{eq:joint}{\leq}\ \ |m_t-m_s|+L\norm{\Gamma_t-\Gamma_s}_{\BL^*}\\
 \ \ &\byeq{eq:calc-first-lift-increments}{\leq}\ \
       |m_t-m_s|+L|v_t-v_s|
       +LM_\gamma\int_s^t
                   (\norm{a_r}_\infty+W\norm{b_r}_\infty)\dd r\\
 &\leq\int_s^t\left[|\dot m_r|+L|\dot v_r|
         +LM_\gamma(\norm{a_r}_\infty+W\norm{b_r}_\infty)\right]\dd r.
\end{align*}
The anchored curve $s\mapsto F_{\lambda_s}-F_{\lambda_0}$ is absolutely continuous in $L^1$. At a common Lebesgue point $s$ of the four velocities, write $\Delta w\coloneqq w_{s+u}-w_s$ and $\Delta y\coloneqq y_{s+u}-y_s$. Bochner differentiation gives
\[
\norm{\Delta w-ua_s}_\infty=o(|u|),\qquad
\norm{\Delta y-ub_s}_\infty=o(|u|),
\]
\noindent and
\[
m_{s+u}-m_s=u\dot m_s+o(|u|), \qquad
v_{s+u}-v_s=u\dot v_s+o(|u|).
\]
For a centered one-Lipschitz $h$, let $q_s(x)\coloneqq\int_\R h(x+z)p_s(z)\dd z$ and $\Phi\coloneqq\Phi_{q_s,0}$. Its first two $y$ derivatives are bounded uniformly over $h$ and $s$. The decomposition
\[
(w_s+\Delta w)\Phi(y_s+\Delta y)-w_s\Phi(y_s)
 ={}\Delta w\,\Phi(y_s)+w_s\Phi'(y_s)\Delta y+R,
\]
where
\[
|R|\leq{}\norm{\Phi'}_\infty|\Delta w|\,|\Delta y|
       +\tfrac12\norm{\Phi''}_\infty w_s|\Delta y|^2
\]
has $\int_E|R|\dd\gamma=O(u^2)$. Therefore, using the pushforward definitions of $\eta_s$ and $J_s$, the affine response between the actual endpoints satisfies, uniformly over these tests,
\begin{align*}
 (m_{s+u}-m_s)q_s'(0)&+\langle\Phi,\Gamma_{s+u}-\Gamma_s\rangle
 \\ &=u\left[\dot m_s q_s'(0)+\dot v_s\Phi(0)+\int_E\bigl(a_s\Phi(y_s)+w_sb_s\Phi'(y_s)\bigr)\dd\gamma\right]+o(|u|)\\
 &=u\left[\dot m_s q_s'(0)+\int_\R \Phi(y)\eta_s(\dd y)+\int_\R \Phi'(y)J_s(\dd y)\right]+o(|u|)\\
 &\byeq{eq:transport-cdf}{=}-u\int_\R h'(x)\mathsf T_{p_s}(\dot m_s,\eta_s,J_s)(x)\dd x+o(|u|).
\end{align*}
Since $|m_{s+u}-m_s|+\|\Gamma_{s+u}-\Gamma_s\|_{\BL^*}=O(|u|)$, Theorem~\ref{thm:response} adds an $O(u^2)$ remainder. Duality gives
\[
F_{\lambda_{s+u}}-F_{\lambda_s}
 =uV_s+o_{L^1}(|u|),
\]
where
\[
V_s\coloneqq\mathsf T_{p_s}(\dot m_s,\eta_s,J_s).
\]
Moreover,
\[
 \norm{V_s}_1\leq|\dot m_s|
 +\frac{K_1}{2}\bigl(|\dot v_s|+M_\gamma\norm{a_s}_\infty\bigr)
 +\frac{K_2}{6}M_\gamma W\norm{b_s}_\infty\in L^1(0,1).
\]
As an almost-everywhere limit of strongly measurable $L^1$ difference quotients, $V$ is strongly measurable, hence Bochner integrable. For any $\ell\in(L^1)^*=L^\infty$, the scalar absolutely continuous function $s\mapsto\ell(F_{\lambda_s}-F_{\lambda_0})$ has derivative $\ell(V_s)$ almost everywhere. Its fundamental theorem of calculus, followed by separation by $L^\infty$, yields
\[
 F_{\lambda_t}-F_{\lambda_s}=\int_s^tV_r\dd r\quad\text{in }L^1.
\]
The length of an indefinite Bochner integral equals the integral of the derivative norm. The triangle inequality gives the upper bound. Approximating $V$ by step functions gives the lower bound. The CDF isometry now proves \eqref{eq:general-length}.
\end{proof}
Mass at $y_s=0$ contributes to the Gaussian variance. Equation \eqref{eq:general-length} gives the length of this curve.
\subsection{Second-order responses}\label{app:further-second}
\subsubsection{Varying redistribution kernels and a Lindeberg condition}
The same limit holds for varying redistribution kernels under a Lindeberg condition on their actual displacements.
\begin{theorem}[Second-order limit for a triangular family]
\label{thm:lindeberg}
Let $\mu\in\cM_f^+(\R)$, $v>0$, and $\pi_\varepsilon\in\cM_f^+(\R_y\times\R_U)$, $\varepsilon\in\R\setminus\{0\}$, have first marginal $\mu$ and satisfy $\int_{\R^2} U f(y)\pi_\varepsilon(\dd y,\dd U)=0$ for every bounded Borel function $f:\R\to\R$. Let $\lambda$ and $\lambda_\varepsilon$ have common mean $m\in\R$ and characteristics $v\delta_0+\mu$ and $\Gamma_\varepsilon\coloneqq v\delta_0+(y+U)_\#\pi_\varepsilon$. For Borel $A\subset\R$, put
\[
S_\varepsilon(A)\coloneqq\varepsilon^{-2}
       \int_{A\times\R}U^2\pi_\varepsilon(\dd y,\dd U),
\]
\noindent and
\[
C_{2,\varepsilon}\coloneqq\int_{\R^2} U^2\pi_\varepsilon(\dd y,\dd U).
\]
Write $p$ for the density of $\lambda$. Suppose $d_{\FM}(S_\varepsilon,S)\xrightarrow[\varepsilon\to0]{}0$ for $S\in\cM_f^+(\R)$, and
\begin{equation}\label{eq:lindeberg-condition}
 \varepsilon^{-2}\int_{\R^2} U^2\one_{\{|U|>r\}}\pi_\varepsilon(\dd y,\dd U)\xrightarrow[\varepsilon\to0]{}0
                      \quad\text{for every }r>0.
\end{equation}
Then \eqref{eq:cancelled-expansion} holds with this $S$. With $K_j\coloneqq\|g_v^{(j)}\|_1$, its $L^1$ remainder is at most
\begin{equation}\label{eq:lindeberg-remainder}
\begin{aligned}
 &\varepsilon^2\left(\frac{K_3}{24}+\frac{K_4}{40}\right)d_{\FM}(S_\varepsilon,S)
 +\frac{K_7C_{2,\varepsilon}^2}{1152}\\
 &\qquad{}+\int_{\R^2} U^2\min\left\{\frac{K_3}{12},\ \frac{|U|K_4}{120}\right\}\pi_\varepsilon(\dd y,\dd U).
\end{aligned}
\end{equation}
\end{theorem}
\begin{remark*}
The coordinates $y$ and $U$ are the initial position and actual displacement.
\end{remark*}
\begin{proof}
Use the preceding interpolation argument with displacement $U$. Its quadratic term is $\varepsilon^2\mathsf Q_p(S_\varepsilon)$. For $\psi\in L^\infty(\R)$ with $\|\psi\|_\infty\leq1$, define $q_\psi:\R\to\R$ by
\[
 q_\psi(y)\coloneqq\tfrac12\int_\R\psi(x)
       \int_0^1r^2(1-r)p'''(x-ry)\dd r\dd x.
\]
Fubini's theorem, differentiation under the integral, and the Gaussian bounds $\|p^{(j)}\|_1\leq K_j$ yield
\[
\norm{q_\psi}_\infty
 \leq\frac{\norm\psi_\infty}{2}
   \int_0^1r^2(1-r)\norm{p'''}_1\dd r
 \leq\frac{K_3}{2}
   \int_0^1r^2(1-r)\dd r=\frac{K_3}{24},
\]
\noindent and
\[
\Lip(q_\psi)
 \leq\frac{\norm\psi_\infty}{2}
   \int_0^1r^3(1-r)\norm{p^{(4)}}_1\dd r
 \leq\frac{K_4}{2}
   \int_0^1r^3(1-r)\dd r=\frac{K_4}{40}.
\]
Then $L^1$ duality, Fubini's theorem, and the preceding bounds give
\begin{align*}
 \norm{\mathsf Q_p(S_\varepsilon)-\mathsf Q_p(S)}_1
 &=
 \sup_{\substack{\psi\in L^\infty(\R)\\\norm\psi_\infty\leq1}}
 \left|\int_\R q_\psi(y)(S_\varepsilon-S)(\dd y)\right|\\
 &\byeq{eq:metrics}{\leq}
 \sup_{\substack{\psi\in L^\infty(\R)\\\norm\psi_\infty\leq1}}\norm{q_\psi}_{\BL}
                  d_{\FM}(S_\varepsilon,S)\\
 &\leq
 \left(\frac{K_3}{24}+\frac{K_4}{40}\right)d_{\FM}(S_\varepsilon,S).
\end{align*}
The other two terms come from the Taylor and interpolation remainders in \eqref{eq:lindeberg-remainder}. Since $\|1\|_{\BL}=1$,
\[
 \left|\varepsilon^{-2}C_{2,\varepsilon}-S(\R)\right|
 \leq d_{\FM}(S_\varepsilon,S)\xrightarrow[\varepsilon\to0]{}0.
\]
Hence $C_{2,\varepsilon}=O(\varepsilon^2)$. Let $R_\varepsilon$ denote the final integral in \eqref{eq:lindeberg-remainder}. For any $r>0$, split it at $|U|=r$ and use $|U|^3\leq rU^2$ on $\{|U|\leq r\}$. Thus
\begin{align*}
 \frac{R_\varepsilon}{\varepsilon^2}
 &\leq
 \frac{K_4}{120\varepsilon^2}
       \int_{\R\times[-r,r]}|U|^3\pi_\varepsilon(\dd y,\dd U)
 +\frac{K_3}{12\varepsilon^2}
       \int_{\R\times(\R\setminus[-r,r])}U^2\pi_\varepsilon(\dd y,\dd U)\\
 &\leq
 \frac{rK_4}{120}\frac{C_{2,\varepsilon}}{\varepsilon^2}
 +\frac{K_3}{12\varepsilon^2}
       \int_{\R\times(\R\setminus[-r,r])}U^2\pi_\varepsilon(\dd y,\dd U).
\end{align*}
Since $\tfrac{C_{2,\varepsilon}}{\varepsilon^2}\xrightarrow[\varepsilon\to0]{}S(\R)$, it follows that
$$\limsup_{\varepsilon\to0}\frac{R_\varepsilon}{\varepsilon^2}
 \ \ \byeq{eq:lindeberg-condition}{\leq}\ \ 
 \frac{rK_4}{120}S(\R)\xrightarrow[r\downarrow0]{}0.$$
Every term in \eqref{eq:lindeberg-remainder} is $o(\varepsilon^2)$.
\end{proof}
As an example of a rare macroscopic displacement, take $M>0$, $a\in\R$, and $b\ne0$. For $0<|\varepsilon|\leq1$, set
\[
 \pi_\varepsilon\coloneqq M\delta_a(\dd y)
 \left[(1-\varepsilon^2)\delta_0(\dd U)
       +\frac{\varepsilon^2}{2}(\delta_{-b}+\delta_b)(\dd U)\right].
\]
Then $S_\varepsilon=Mb^2\delta_a$, whereas for $0<r<|b|$, $\varepsilon^{-2}\int_{\R^2} U^2\one_{\{|U|>r\}}\pi_\varepsilon(\dd y,\dd U)=Mb^2$. Using the response functions in \eqref{eq:response-dictionary},
\[
 \varepsilon^{-2}(F_{\lambda_\varepsilon}-F_\lambda)
 \xrightarrow[\varepsilon\to0]{L^1}\frac M2[A_p(a-b,\cdot)+A_p(a+b,\cdot)-2A_p(a,\cdot)].
\]
This finite difference need not equal $\mathsf Q_p(Mb^2\delta_a)$.
\subsubsection{Matching second variations of two positive approximations}
Two approximations with the same first tangent differ at second order only through their dispersion measures.
\begin{proposition}[Second-response norm and matched lifts]
\label{prop:second-matching}
Extend \eqref{eq:second-observable} linearly to $S\in\cM_f(\R)$. Then $\|S\|_{\lambda,2}\coloneqq\|\mathsf Q_p(S)\|_1$ is a norm on $\cM_f(\R)$ and contracts under common independent infinitely divisible convolution. For two lifted paths at the same reference law, assume their first pairs $(\alpha,\tau)$ agree and their third displacement moments are finite. If their dispersion measures are $S_0,S_1$, then
\begin{equation}\label{eq:matched-response}
 W_1(\lambda_\varepsilon^0,\lambda_\varepsilon^1)
       =\varepsilon^2\|S_0-S_1\|_{\lambda,2}+O(|\varepsilon|^3).
\end{equation}
If also $S_0=S_1$, the error is $O(|\varepsilon|^3)$. For two balanced symmetric lifts with equal $S$ and finite fourth velocity moments, this improves to $O(\varepsilon^4)$.
\end{proposition}
\begin{proof}
The norm is finite since $\|\mathsf Q_p(S)\|_1\leq \tfrac{\|p'''\|_1\|S\|_{\mathrm{TV}}}{24}$. Let $k(\xi,y)\coloneqq-\xi^2\int_0^1(1-r)e^{ir\xi y}\dd r$ and $b(\xi)\coloneqq\tfrac12\int_\R\partial_y^2k(\xi,y)S(\dd y)$. Fourier transformation gives $i\xi\widehat{\mathsf Q_p(S)}(\xi)=b(-\xi)\widehat p(\xi)$. Also $b''(\xi)=\tfrac{\xi^2\int_\R e^{i\xi y}S(\dd y)}{2}$. Indeed, the zero initial values at $\xi=0$ give
\[
\partial_y^2k(\xi,y)=\int_0^\xi(\xi-t)t^2e^{ity}\dd t,
\]
hence
\[
|\partial_y^2k(\xi,y)|\leq\frac{|\xi|^4}{12}.
\]
Since $|S|(\R)<\infty$, differentiation of this integral twice is valid, uniformly in $y$ on compact $\xi$ intervals. Since $\widehat p$ never vanishes, a zero response forces $b=0$. It follows that $\widehat S=0$ away from the origin and hence, by continuity, everywhere. Uniqueness of Fourier transforms gives $S=0$. Linearity proves the remaining norm axioms, and $\mathsf Q_{p*\zeta}(S)=\mathsf Q_p(S)*\zeta$ gives contraction. Equal first tangents give the same $\mathcal A$ and $T_1$. Subtract \eqref{eq:two-jet-L1} for the two paths: the composed term $\tfrac{\mathcal A^*T_1}{2}$ cancels, leaving $\varepsilon^2\mathsf Q_p(S_0-S_1)$ and a cubic remainder. CDF duality proves \eqref{eq:matched-response}. For balanced symmetric lifts, we use the same subtraction with the fourth-order remainders in \eqref{eq:cancelled-fourth}.
\end{proof}
For example, use the velocity laws $\kappa_0,\kappa_1\in\cP(\R)$ $\kappa_0\coloneqq\tfrac{\delta_{-1}+\delta_1}{2}$
and $\kappa_1\coloneqq\tfrac34\delta_0+\tfrac18(\delta_{-2}+\delta_2).$ Both satisfy $\int_\R z\kappa_i(\dd z)=0$ and $\int_\R z^2\kappa_i(\dd z)=1$, while $\int_\R z^4\kappa_0(\dd z)=1$ and $\int_\R z^4\kappa_1(\dd z)=4$. For the lifts $\Pi_i\coloneqq\mu\otimes\kappa_i$, $S_0=S_1=\mu$: each nontrivial approximation has quadratic error, but the distance between them is $O(\varepsilon^4)$.
\subsubsection{A fixed-intensity atom and the role of the weight}
\label{sec:raw-atom}
Moving an atom while keeping its raw Lévy intensity fixed changes both the location and the mass of the variance-weighted atom.
\begin{proposition}[Raw fixed-intensity transport]\label{prop:raw-atom}
Fix $c,v>0$, $a\in\R\setminus\{0\}$, and a mean $m\in\R$. Let the reference law have raw L\'evy measure $N\coloneqq c\delta_a$ and density $p$. Keep the mean and Gaussian variance fixed and set $N_\varepsilon\coloneqq c\delta_{a+\varepsilon}$, with $|\varepsilon|<\tfrac{|a|}{2}$. Then
\begin{equation}\label{eq:raw-first}
T_1(x)=c[p(x)-p(x-a)],
\end{equation}
\noindent and
\begin{equation}\label{eq:raw-second}
T_2(x)=\frac c2p'(x-a)+\frac{c^2}{2}
                  [p'(x)-2p'(x-a)+p'(x-2a)].
\end{equation}
With $K_2\coloneqq\|g_v''\|_1$, we have
\begin{equation*}
 \norm{F_{\lambda_\varepsilon}-F_\lambda
                 -\varepsilon T_1-\varepsilon^2T_2}_1
 \leq K_2\left(\frac c6+c^2+\frac{4c^3}{3}\right)|\varepsilon|^3.
\end{equation*}
The first Wasserstein coefficient $c\|p-p(\cdot-a)\|_1$ is positive.
\end{proposition}
\begin{proof}
For $f\in C^3(\R)$ with bounded first three derivatives, the generator difference is
\[
\mathcal D_\varepsilon f
   \coloneqq c[f(\cdot+a+\varepsilon)-f(\cdot+a)-\varepsilon f'].
\]
\noindent Its first two coefficients are
\[
\mathcal Af\coloneqq c[f'(\cdot+a)-f'],\qquad
\mathcal Bf\coloneqq\frac c2f''(\cdot+a).
\]
Taylor expansion gives
\[
\norm{\mathcal D_\varepsilon f}_\infty
 \leq 2c|\varepsilon|\norm{f'}_\infty,\qquad
\norm{(\mathcal D_\varepsilon-\varepsilon\mathcal A)f}_\infty
 \leq \frac{c\varepsilon^2}{2}\norm{f''}_\infty.
\]
\noindent The second-order remainder satisfies
\[
\norm{(\mathcal D_\varepsilon-\varepsilon\mathcal A-\varepsilon^2\mathcal B)f}_\infty
 \leq \frac{c|\varepsilon|^3}{6}\norm{f'''}_\infty.
\]
Then, using commutation and these bounds, we obtain
\[
\norm{(\mathcal D_\varepsilon^2-\varepsilon^2\mathcal A^2)f}_\infty
 \leq 2c^2|\varepsilon|^3\norm{f'''}_\infty.
\]
\noindent The third iterate satisfies
\[
\norm{\mathcal D_\varepsilon^3f}_\infty
 \leq 8c^3|\varepsilon|^3\norm{f'''}_\infty.
\]
The third-order interpolation identity and $\|(Q_uh)'''\|_\infty\leq K_2\Lip(h)$ yield $K_2|\varepsilon|^3\left(\tfrac{c}{6}+c^2+\tfrac{4c^3}{3}\right)$. The spatial adjoint on $u\in W^{1,1}(\R)$ is $\mathcal A^*u=c[u'-u'(\cdot-a)]$. Its first CDF response is \eqref{eq:raw-first}. Adding the primitive of $\mathcal B^*p$ and $\tfrac{\mathcal A^*T_1}{2}$ gives \eqref{eq:raw-second}. If $p=p(\cdot-a)$ almost everywhere, the positive integral over one period would be repeated on every translated period, contradicting $\int_\R p(x)\dd x=1$. Hence $T_1\ne0$.
\end{proof}
The weighted characteristic is $\Gamma_\varepsilon=v\delta_0+c(a+\varepsilon)^2\delta_{a+\varepsilon}$. Its first distributional coefficient is
\[
 \tau=2ca\delta_a-ca^2D_y\delta_a.
\]
The second coefficient is
\[
 c\delta_a-2caD_y\delta_a+\tfrac12ca^2D_y^2\delta_a.
\]
For $\phi\in C_c^\infty(\R)$, these coefficients follow from
\[
\begin{aligned}
 c(a+\varepsilon)^2\phi(a+\varepsilon)
 &=ca^2\phi(a)+c\varepsilon[2a\phi(a)+a^2\phi'(a)]\\
 &\quad{}+c\varepsilon^2\left[\phi(a)+2a\phi'(a)+\frac{a^2\phi''(a)}{2}\right]
 +O_\phi(|\varepsilon|^3).
\end{aligned}
\]
\section{Remeshing proofs and numerical details}\label{app:remesh-proofs}
\subsection{Proof of Theorem~\ref{thm:barycentric}.}\label{proof:thm:barycentric}
\begin{proof}
Let $\mathcal D$ be the terminal generator of $\widehat\lambda$ minus that of $\lambda$. For $f\in C_b^4(\R)$, the common mean cancels the drift difference, while balance cancels the first derivative term in the jump coordinate. Taylor's integral formula then gives
\begin{align*}
 \mathcal Df(x)
 &=\int_{\R^2}[\Phi_{f,x}(z)-\Phi_{f,x}(y)]\pi(\dd y,\dd z)\\
 &\byeq{eq:barycentric}{=}
 \int_{\R^2}[\Phi_{f,x}(z)-\Phi_{f,x}(y)
              -(z-y)\partial_y\Phi_{f,x}(y)]\pi(\dd y,\dd z)\\
 &=\int_{\R^2}(z-y)^2\int_0^1(1-t)
       \partial_y^2\Phi_{f,x}(y+t(z-y))\dd t\,\pi(\dd y,\dd z).
\end{align*}
Since $\|\partial_y^2\Phi_{f,x}\|_\infty \leq\tfrac{\|f^{(4)}\|_\infty}{12}$,
\[
 \norm{\mathcal Df}_\infty
 \leq\frac{\mathcal C_2}{24}\norm{f^{(4)}}_\infty.
\]
For a real one-Lipschitz test $h$, every terminal interpolation law retains Gaussian variance $v$, so $$\|(Q_uh)^{(4)}\|_\infty\leq\|g_v'''\|_1\Lip(h).$$ The triangle inequality and the Taylor bound give
\begin{align*}
 |\langle h,\widehat\lambda-\lambda\rangle|
 \ &\byeq{eq:interpolation-identity}{=}\ 
 \left|\int_0^1\mathcal DQ_uh(0)\dd u\right|
 \\ &\leq
 \int_0^1\norm{\mathcal DQ_uh}_\infty\dd u\\
 &\leq
 \frac{\mathcal C_2}{24}\int_0^1\norm{(Q_uh)^{(4)}}_\infty\dd u
 \leq
 \frac{\mathcal C_2}{24}\norm{g_v'''}_1\Lip(h).
\end{align*}
Lemmas~\ref{lem:interpolation} and~\ref{lem:approximation} justify the interpolation and approximation steps. Kantorovich duality completes the proof. Reversing the coupling proves the version with conditional balance given $z$.
\end{proof}
\subsection{Proof of Theorem~\ref{thm:sharp-grid}.}\label{proof:thm:sharp-grid}
\begin{proof}
Fix a plan $\pi\in\mathcal B_h$ and use $\xi,\gamma$ from \eqref{eq:grid-lower-constant}. Write $p$ and $\Psi$ for the density and characteristic exponent of $\lambda$. For $z\in h\mathbb Z$, $(z-\ell_h(y))(z-r_h(y))\geq0$. Expand
\[
 (z-\ell_h)(z-r_h)
 =(z-y)^2+(2y-\ell_h-r_h)(z-y)-(y-\ell_h)(r_h-y).
\]
Since $|2y-\ell_h-r_h|\leq h$, balance removes the middle term after integration. Hence
\begin{equation}\label{eq:calc-second-grid-cost}
 C_2\coloneqq \int_{\R^2}(z-y)^2\pi(\dd y,\dd z)\geq\mathcal V_h.
\end{equation}
Splitting between the two bracketing neighbours attains this minimum and belongs to $\mathcal B_h$. The upper bound now follows from Theorem~\ref{thm:barycentric}. Define $S_\pi\in\cM_f^+(\R)$ by $S_\pi(A)\coloneqq\int_{A\times\R}(z-y)^2\pi(\dd y,\dd z)$ for Borel $A\subset\R$. For every $S\in\cM_f^+(\R)$ supported on $[-R,R]$, positivity and $R\xi<1<\tfrac{\pi}{2}$ give, for $0\leq r\leq1$,
\[
 \operatorname{Re}\widehat S(r\xi)
 =\int_\R\cos(r\xi y)S(\dd y)
 \geq\cos(R\xi)S(\R).
\]
The inequality $\cos t\geq1-\tfrac{t^2}{2}$ gives
\begin{equation}\label{eq:calc-second-characteristic-bound}
 \begin{aligned}
 \operatorname{Re}\Psi(\xi)
 =-\frac{v\xi^2}{2}+\int_\R\frac{\cos(\xi y)-1}{y^2}\mu(\dd y)
 \geq-\frac{(v+M)\xi^2}{2}.
 \end{aligned}
\end{equation}
At $y=0$, the quotient is interpreted as $-\tfrac{\xi^2}{2}$. The $L^1$ norm bounds the modulus of the Fourier transform, which in turn bounds its positive real part. Using \eqref{eq:calc-second-fourier} and \eqref{eq:calc-second-characteristic-bound}, we obtain
\begin{align*}
 \norm{\mathsf Q_p(S)}_1
 &\geq|\widehat{\mathsf Q_p(S)}(\xi)|
 =
 \frac{\xi^3|\widehat p(\xi)|}{2}
 \left|\int_0^1r^2(1-r)\widehat S(r\xi)\dd r\right|\\
 &\geq
 \frac{\xi^3|\widehat p(\xi)|\cos(R\xi)}{24}S(\R)\\
 &\geq
 \frac{\xi^3e^{-\frac{(v+M)\xi^2}{2}}\cos(R\xi)}{24}S(\R)
 \byeq{eq:grid-lower-constant}{=}\gamma S(\R).
\end{align*}
The same interpolation calculation as \eqref{eq:cancelled-third}, applied to $U\coloneqq z-y$, yields
\[
 \norm{F_{\widehat\lambda_\pi}-F_\lambda-\mathsf Q_p(S_\pi)}_1
 \leq \frac{K_4C_3}{120}+\frac{K_7C_2^2}{1152},
\]
where
\[
 C_3\coloneqq\int_{\R^2}|U|^3\pi(\dd y,\dd z),\qquad
 K_j\coloneqq\norm{g_v^{(j)}}_1\quad(j\in\N).
\]
Since $C_3\leq LhC_2$ and $C_2\leq ML^2h^2$, choose $h_0$ so that $\tfrac{K_4Lh_0}{120}+\tfrac{K_7ML^2h_0^2}{1152}\leq\tfrac{\gamma}{2}$. CDF duality, the reverse triangle inequality, these bounds and the choice $h\leq h_0$ give, for every feasible plan,
\begin{align*}
 W_1(\lambda,\widehat\lambda_\pi)
 =\norm{F_{\widehat\lambda_\pi}-F_\lambda}_1
 &\geq
 \norm{\mathsf Q_p(S_\pi)}_1
 -\norm{F_{\widehat\lambda_\pi}-F_\lambda-\mathsf Q_p(S_\pi)}_1\\
 &\geq
 \left(\gamma-\frac{K_4Lh}{120}
                   -\frac{K_7ML^2h^2}{1152}\right)C_2
 \geq\frac{\gamma C_2}{2}\ 
 \byeq{eq:calc-second-grid-cost}{\geq}\ \frac{\gamma\mathcal V_h}{2}.
\end{align*}
\end{proof}
\subsection{Proof of Theorem~\ref{thm:remeshing-oracle}.}\label{proof:thm:remeshing-oracle}
\begin{proof}
For $\|f\|_{\BL}\leq1$,
\[
 |\langle f,\widehat\mu-\mu\rangle|
 =\left|\int_{\R^2}[f(z)-f(y)]\pi(\dd y,\dd z)\right|
 \leq\Lip(f)\int_{\R^2}|z-y|\pi(\dd y,\dd z)\leq Mh.
\]
Thus $\|\widehat\mu-\mu\|_{\BL^*}\leq Mh$. The affine remainder \eqref{eq:response-bound}, with zero mean change, and its $L^1$ representation give \eqref{eq:remeshing-remainder}. Let $\mathcal D$ be the generator of $\widehat\lambda_\pi$ minus that of $\lambda$. In the balanced case, for $f\in C_b^8(\R)$, it satisfies
\[
 \norm{\mathcal Df}_\infty\leq \frac{C_2\norm{f^{(4)}}_\infty}{24},\qquad
 \norm{\mathcal D^2f}_\infty\leq \frac{C_2^2\norm{f^{(8)}}_\infty}{576},
\]
where
\[
 C_2\coloneqq\int_{\R^2}|z-y|^2\pi(\dd y,\dd z)\leq Mh^2.
\]
The second interpolation identity bounds the error of $H_\pi$ by $\tfrac{K_7C_2^2}{1152}$, proving \eqref{eq:balanced-oracle-remainder}. For any feasible $\pi$, the remainder bound at $\pi_*$, its minimizing property, and the remainder bound at $\pi$ give:
\[
 W_1(\lambda,\widehat\lambda_{\pi_*})
 \byeq{eq:remeshing-remainder}{\leq}\mathcal J_p(\pi_*)+R_h
 \leq\mathcal J_p(\pi)+R_h
 \byeq{eq:remeshing-remainder}{\leq}
 W_1(\lambda,\widehat\lambda_\pi)+2R_h.
\]
Taking the infimum proves \eqref{eq:oracle}. If $|\widehat{\mathcal J}(\pi)-\mathcal J_p(\pi)|\leq e_q$ and $\widehat{\mathcal J}(\widehat\pi) \leq\inf_{\pi\in\mathcal C_h}\widehat{\mathcal J}(\pi)+e_o$, the uniform objective error and the terminal remainder give
\begin{align*}
 W_1(\lambda,\widehat\lambda_{\widehat\pi})
 &\byeq{eq:remeshing-remainder}{\leq}
 \mathcal J_p(\widehat\pi)+R_h
 \leq
 \widehat{\mathcal J}(\widehat\pi)+e_q+R_h\\
 &\leq
 \inf_{\pi\in\mathcal C_h}\widehat{\mathcal J}(\pi)
      +e_o+e_q+R_h
 \leq
 \inf_{\pi\in\mathcal C_h}\mathcal J_p(\pi)+e_o+2e_q+R_h\\
 &\byeq{eq:remeshing-remainder}{\leq}
 \inf_{\pi\in\mathcal C_h}W_1(\lambda,\widehat\lambda_\pi)
       +2R_h+2e_q+e_o.
\end{align*}
\end{proof}
\Needspace{12\baselineskip}
\subsection{Consistency of quadrature}\label{app:quadrature}
\begin{proposition}[Consistency of the finite-grid program]
\label{prop:LP-consistency}
Fix a finite source and grid, and a nonempty feasible polytope \eqref{eq:LP-constraints}. Use the exact response functions and midpoint quadrature on $[-L_n,L_n]$ of step $\Delta_n$, with $L_n\xrightarrow[n\to\infty]{}\infty$ and $\Delta_n\xrightarrow[n\to\infty]{}0$. Then the discrete objectives converge uniformly to $\mathcal J_p$ on this polytope. Their minimum values converge, and every cluster point of minimizers minimizes the exact objective. The same holds for feasible approximate minimizers whose tolerance tends to zero.
\end{proposition}
\begin{proof}
Write
\[
 \mathcal P_{\mathrm{LP}}
 \coloneqq\{\pi\in[0,\infty)^{I\times J}:\pi\text{ satisfies }\eqref{eq:LP-constraints}\}
\]
for the feasible polytope. For $1\leq i\leq I$ and $1\leq j\leq J$, put $B_{ij}\coloneqq A_p(z_j,\cdot)-A_p(y_i,\cdot)\in W^{1,1}(\R)$ and $H_\pi\coloneqq\sum_{i=1}^I\sum_{j=1}^J\pi_{ij}B_{ij}$. For the $n$th mesh, let $q_n\in\N$ be the number of cells. Write these cells as $I_1,\ldots,I_{q_n}$ and let $x_\ell$ be the midpoint of $I_\ell$. Use their lengths $\omega_\ell\coloneqq|I_\ell|=\Delta_n$ as quadrature weights. The resulting objective and its uniform error over the feasible plans are
\[
 \mathcal J_n(\pi)\coloneqq\sum_{\ell=1}^{q_n}\omega_\ell|H_\pi(x_\ell)|
\]
and
\[
 e_n\coloneqq\sup_{\pi\in\mathcal P_{\mathrm{LP}}}
                   |\mathcal J_n(\pi)-\mathcal J_p(\pi)|.
\]
For $H\in W^{1,1}(\R)$, use its absolutely continuous representative. Midpoint quadrature gives
\begin{equation}\label{eq:objective-quadrature}
 \left|\sum_{\ell=1}^{q_n}\omega_\ell|H(x_\ell)|-\norm H_1\right|
 \leq\frac{\Delta_n}{2}\norm{H'}_1+\int_{\{x\in\R:|x|>L_n\}}|H(x)|\dd x.
\end{equation}
For a cell $I_\ell$ with midpoint $x_\ell$ and length $\Delta_n$, absolute continuity and Fubini's theorem give the following bound, with weights at most $\tfrac{\Delta_n}{2}$:
\begin{align*}
 \int_{I_\ell}\big||H(x)|-|H(x_\ell)|\big|\dd x
 &\leq
 \int_{I_\ell}\left|\int_{x_\ell}^{x}|H'(t)|\dd t\right|\dd x\\
 &=
 \int_{x_\ell-\frac{\Delta_n}{2}}^{x_\ell}
      \left(t-x_\ell+\frac{\Delta_n}{2}\right)|H'(t)|\dd t\\
 &\qquad\qquad+\int_{x_\ell}^{x_\ell+\frac{\Delta_n}{2}}
      \left(x_\ell+\frac{\Delta_n}{2}-t\right)|H'(t)|\dd t\\
 &\leq
 \frac{\Delta_n}{2}\int_{I_\ell}|H'(t)|\dd t.
\end{align*}
Summing over the cells and adding the tail proves \eqref{eq:objective-quadrature}. The row-sum constraints then give
\begin{equation}\label{eq:uniform-LP-error}
 \begin{aligned}
 e_n&\leq\frac{\Delta_n}{2}\sum_{i=1}^I M_i\max_{1\leq j\leq J}\norm{B_{ij}'}_1\\
 &\qquad\qquad+\sum_{i=1}^I M_i\max_{1\leq j\leq J}\int_{\{x\in\R:|x|>L_n\}}|B_{ij}(x)|\dd x
 \xrightarrow[n\to\infty]{}0.
 \end{aligned}
\end{equation}
The family of response functions is finite, so $e_n\xrightarrow[n\to\infty]{}0$. Hence
\[
 \left|\min_{\pi\in\mathcal P_{\mathrm{LP}}}\mathcal J_n(\pi)-\min_{\pi\in\mathcal P_{\mathrm{LP}}}\mathcal J_p(\pi)\right|\leq e_n.
\]
For feasible $\pi_n$ with $\mathcal J_n(\pi_n)\leq\min_{\pi\in\mathcal P_{\mathrm{LP}}}\mathcal J_n(\pi)+\delta_n$, where the tolerances $\delta_n\geq0$ tend to zero, the optimization tolerance and the uniform error imply
\[
 \begin{aligned}
 \mathcal J_p(\pi_n)
 \byeq{eq:uniform-LP-error}{\leq}\mathcal J_n(\pi_n)+e_n
 &\leq\min_{\pi\in\mathcal P_{\mathrm{LP}}}\mathcal J_n(\pi)+\delta_n+e_n\\
 &\byeq{eq:uniform-LP-error}{\leq}\min_{\pi\in\mathcal P_{\mathrm{LP}}}\mathcal J_p(\pi)+2e_n+\delta_n.
 \end{aligned}
\]
Along any convergent subsequence, continuity shows that its limit minimizes $\mathcal J_p$. Compactness ensures that such subsequences exist.
\end{proof}
For balanced plans of radius $h$, let $K_4\coloneqq\|g_v^{(4)}\|_1$ and $M\coloneqq\mu(\R)$. Taylor expansion of $D_xA_p$ and cancellation give
\[
 \norm{H_\pi'}_1
 \leq\frac{K_4}{24}\int_{\R^2}|z-y|^2\pi(\dd y,\dd z)
 \leq\frac{K_4Mh^2}{24}.
\]
The derivative contribution is at most $\tfrac{\Delta_nK_4Mh^2}{48}$. Under joint refinement, the optimization gap remains fourth order if $\Delta_n=o(h^2)$ and the errors are of order $o(h^4)$ when truncating the tails, evaluating the response functions, and solving the linear program.
\subsection{Approximation guarantees}\label{sup:sec:approximation}
\subsubsection{A computable certificate for jump quantization}
\label{sup:sec:quantization}
Let $(m,v,N)$ be a finite-variance terminal triplet with mean $m\in\R$ and Gaussian variance $v>0$. Write its weighted jump measure as $\mu(\dd y)\coloneqq y^2N(\dd y)\in\cM_f^+(\R)$. For a measurable map $q:\R\to\R$, define the pushforward $\widehat\mu\coloneqq q_\#\mu$ by $\widehat\mu(A)\coloneqq\mu(q^{-1}(A))$ for Borel $A\subset\R$, and set
\begingroup
\begin{equation*}
 \widehat v\coloneqq v+\widehat\mu(\{0\}),
\qquad
 \widehat N(\dd z)\coloneqq\one_{\{z\ne0\}}z^{-2}\widehat\mu(\dd z),
\end{equation*}
\endgroup
and $ \widehat m\coloneqq m$. This defines a finite-variance triplet because
\[
 \int_{\R} z^2\widehat N(\dd z)=\widehat\mu(\R\setminus\{0\})<\infty,
\qquad
 \int_{\R}(1\wedge z^2)\widehat N(\dd z)
 \leq\widehat\mu(\R\setminus\{0\})<\infty.
\]
Moreover, $\widehat m=m$ and
\[
\widehat v+\int_{\R} z^2\widehat N(\dd z)
 =v+\widehat\mu(\R)=v+\mu(\R).
\]
Thus mean and total variance are preserved, with weighted mass sent to zero absorbed into the Gaussian variance.
\begin{proposition}[Weighted quantization certificate]
\label{sup:prop:quantization}
Let $\lambda,\widehat\lambda\in\cP_1(\R)$ be the laws of these triplets. Then
\begin{equation}\label{sup:eq:quantization-certificate}
 W_1(\lambda,\widehat\lambda)
 \leq\left(\frac{\sqrt{\tfrac{2}{\pi}}}{2\sqrt v}+\frac{\sqrt{\tfrac{8}{\pi e}}}{6v}\right)
       \int_\R\frac{2|q(y)-y|}{2+|q(y)-y|}\mu(\dd y).
\end{equation}
If $R>0$, $\eta\geq0$, and $|q(y)-y|\leq\eta$ on $[-R,R]$, the integral is at most
\begin{equation}\label{sup:eq:quantization-tail}
 \eta\mu([-R,R])+2\mu(\{|y|>R\}).
\end{equation}
\end{proposition}
\begin{proof}
The joint Gaussian and jump measures are
\[
\Gamma\coloneqq v\delta_0+\mu
\]
and
\[
\widehat\Gamma\coloneqq\widehat v\delta_0+z^2\widehat N
                =v\delta_0+q_\#\mu.
\]
For $f\in\BL(\R)$ with $\|f\|_{\BL}\leq1$, set $d\coloneqq|q(y)-y|$. Lemma~\ref{lem:atoms} gives $|f(q(y))-f(y)|\leq\tfrac{2d}{2+d}$. Indeed, write $A\coloneqq\|f\|_\infty$ and $L\coloneqq\Lip(f)$, so $A+L\leq1$. Using both bounds on the increments of $f$ and maximizing the resulting piecewise linear expression gives
\[
 \begin{aligned}
 |f(q(y))-f(y)|
 &\leq\min\{2A,Ld\} \leq\min\{2(1-L),Ld\}\\
 &\leq\max_{\ell\in[0,1]}\min\{2(1-\ell),\ell d\} =\frac{2d}{2+d}.
 \end{aligned}
\]
Consequently, duality and the triangle inequality yield
\begin{align*}
 d_{\FM}(\Gamma,\widehat\Gamma)
 &\byeq{eq:metrics}{=}
    \sup_{\substack{f\in\BL(\R)\\\norm f_{\BL}\leq1}}
       \left|\int_\R[f(q(y))-f(y)]\mu(\dd y)\right|\\
 &\leq\sup_{\substack{f\in\BL(\R)\\\norm f_{\BL}\leq1}}
       \int_\R|f(q(y))-f(y)|\mu(\dd y) \\&\leq
       \int_\R\frac{2|q(y)-y|}{2+|q(y)-y|}\mu(\dd y).
\end{align*}
The Gaussian variance stays at least $v$ throughout the interpolation. Corollaries~\ref{cor:joint} and~\ref{cor:Gaussian} therefore give \eqref{sup:eq:quantization-certificate}. Splitting the integral over $[-R,R]$ and its complement and using $\tfrac{2d}{2+d}\leq\min\{d,2\}$ gives \eqref{sup:eq:quantization-tail}.
\end{proof}
Sending the tail to zero and the remaining mass to a finite grid on $[-R,R]$ produces a Gaussian plus a compound-Poisson law. At each nonzero grid point $z$, the jump intensity is $\tfrac{\widehat\mu(\{z\})}{z^2}$. The certificate vanishes as $\eta\downarrow0$ and $R\uparrow\infty$. It requires no balance and is first order in the grid error. Theorem~\ref{thm:barycentric} gives a quadratic bound for balanced plans. Proposition~\ref{ex:replacement} gives third- and fourth-moment certificates for Gaussian replacement.
\subsubsection{Observables and limitations}\label{sup:sec:observables}
For real random variables $U,V$ with finite first moments and $f:\R\to\R$ with $\Lip(f)<\infty$,
\[
 |\E\left[f(U)\right]-\E\left[f(V)\right]|\leq\Lip(f)W_1(\Law(U),\Law(V)).
\]
Let $L_0,L_1$ be real random variables with finite first moments. Suppose $W_1(\Law(L_0),\Law(L_1))\leq\Delta$ and $\max_{i\in\{0,1\}}\E\left[e^{pL_i}\right]\leq K<\infty$, 
for some $\Delta\geq0$, $K>0$, and $p>1$. Then, for $S_0>0$,
\begin{equation*}
 W_1\bigl(\Law(S_0e^{L_0}),\Law(S_0e^{L_1})\bigr)
       \leq3S_0K^{\frac{1}{p}}\Delta^{1-\frac{1}{p}}.
\end{equation*}
Choose an optimal coupling $(U,V)$ of the log laws. The map $x\mapsto e^{x\wedge R}$ is $e^R$-Lipschitz. Truncating at $R$, then using this Lipschitz bound and the exponential moment assumption, gives, for every $R\in\R$,
\begin{align*}
 \E\left[|e^U-e^V|\right]
 &\leq\E\left[|e^{U\wedge R}-e^{V\wedge R}|\right]
       +\E\left[e^U\one_{\{U>R\}}+e^V\one_{\{V>R\}}\right]\\
 &\leq e^R\E\left[|U-V|\right]
       +e^{-(p-1)R}\bigl(\E\left[e^{pU}\right]+\E\left[e^{pV}\right]\bigr)\\
 &\leq e^R\Delta+2Ke^{-(p-1)R}.
\end{align*}
For $\Delta>0$, choose $e^{pR}=\tfrac{K}{\Delta}$. Substituting this choice gives
\[
 \E\left[|e^U-e^V|\right]
 \leq\Delta\left(\frac{K}{\Delta}\right)^{\frac{1}{p}}
              +2K\left(\frac{K}{\Delta}\right)^{-\frac{p-1}{p}} =3K^{\frac{1}{p}}\Delta^{1-\frac{1}{p}}.
\]
For $\Delta=0$, $U=V$ almost surely. Multiply by $S_0$ and take the infimum over couplings. Exponentiation thus changes the error from $\Delta$ to $\Delta^{1-\frac{1}{p}}$ in this bound. The additive bounds compare terminal laws. Reordering characteristics while preserving their integrals leaves the terminal triplet unchanged. Path-space sensitivity \citep{GairingEtAl2015,GairingEtAl2018} and $W_p$ contraction under fixed noise \citep{Wang2016} concern different comparisons.
\section{Sharpness and vanishing Gaussian smoothing}\label{app:sharp}
Write $\operatorname{Poisson}(\lambda)$, $\lambda>0$, for the law on $\N_0$ with mass $\tfrac{e^{-\lambda}\lambda^j}{j!}$ at $j\in\N_0$. Let $N\sim\operatorname{Poisson}(1)$ and $M\sim\operatorname{Poisson}\left(\tfrac{1}{4}\right)$, and put
\begin{subequations}\label{eq:YZ}
\begin{equation}
 Y\coloneqq N-1
\end{equation}
and
\begin{equation}
 Z\coloneqq2\left(M-\frac{1}{4}\right).
\end{equation}
\end{subequations}
For $P_\lambda\sim\operatorname{Poisson}(\lambda)$, Taylor expansion as $t\to0$ gives
\[
 \log\E\left[e^{t(P_\lambda-\lambda)}\right]
 =\lambda(e^t-1-t)
 =\lambda\left(\frac{t^2}{2}+\frac{t^3}{6}+\frac{t^4}{24}\right)+O(t^5).
\]
Exponentiating gives the centered moments $0,\lambda,\lambda,\lambda+3\lambda^2$. Hence
\[
 \left(\E\left[Y\right],\E\left[Y^2\right],\E\left[Y^3\right],\E\left[Y^4\right]\right)=(0,1,1,4)
\]
and
\[
 \left(\E\left[Z\right],\E\left[Z^2\right],\E\left[Z^3\right],\E\left[Z^4\right]\right)=(0,1,2,7).
\]
For $a>0$, the laws of $aY$ and $aZ$ have full-compensation jump measures $\delta_a$ and $\tfrac14\delta_{2a}$, respectively. Their variance-weighted measures are $a^2\delta_a$ and $a^2\delta_{2a}$.
\begin{lemma}[Two-atom discrepancy]\label{lem:atoms}
For $m\geq0$ and $x,z\in\R$,
\begin{equation*}
 d_{\FM}(m\delta_x,m\delta_z)=\frac{2m|x-z|}{2+|x-z|}.
\end{equation*}
\end{lemma}
\begin{proof}
For a bounded Lipschitz test $f:\R\to\R$, let $d\coloneqq|x-z|$ and $A\coloneqq\|f\|_\infty$ and $B\coloneqq\Lip(f)$. If $A+B\leq1$, then $|f(x)-f(z)|\leq\min\{2A,Bd\}\leq\tfrac{2d}{2+d}$. Equality is attained by the clipped affine function
\[
 f(y)\coloneqq\max\left\{-\frac d{2+d},
       \min\left\{\frac{2y-x-z}{2+d},\frac d{2+d}\right\}\right\}.
\]
Multiplying by $m$ gives the result.
\end{proof}
\begin{proposition}[Degenerate cubic-root lower scale]\label{prop:cubic}
Define
\[
 c_0\coloneqq W_1(\Law(Y),\Law(Z)).
\]
Then $c_0>0$ and
\begingroup
\begin{equation*}
 \theta(a)\coloneqq d_{\FM}(a^2\delta_a,a^2\delta_{2a})=\frac{2a^3}{2+a}
\end{equation*}
whereas
\begin{equation*}
 W_1(\Law(aY),\Law(aZ))=c_0a.
\end{equation*}
\endgroup
No exponent greater than $\tfrac{1}{3}$ can hold uniformly over this family for $0<a\leq\tfrac{1}{2}$, with constants depending only on common upper bounds on its moments and stochastic-flow derivatives.
\end{proposition}
\begin{proof}
Lemma~\ref{lem:atoms} gives $\theta(a)$, and Wasserstein scaling gives $ac_0$. The one-Lipschitz test $h:\R\to\left[0,\tfrac{1}{2}\right]$ defined by $h(x)\coloneqq\inf_{k\in\mathbb Z}|x-k|$ has $h(Y)=0$ and $h(Z)=\tfrac{1}{2}$ almost surely, so $c_0\geq\tfrac{1}{2}$. If a uniform estimate held with exponent $\gamma>\tfrac{1}{3}$, it would imply $c_0\leq C a^{3\gamma-1}\left(1+\tfrac{a}{2}\right)^{-\gamma}\xrightarrow[a\downarrow0]{}0$, a contradiction. The intensities are fixed and the jumps bounded for $a\leq\tfrac{1}{2}$. The additive flows have first derivative one and all higher derivatives zero, while their moments are uniformly bounded. The regularized estimate \eqref{eq:regularized-terminal} gives the matching upper exponent for this class.
\end{proof}
\begin{lemma}[Moment matching after Gaussian convolution]\label{lem:moments}
Let $k\in\N$, $k\geq1$, and let $U,V\in L^{k+1}(\Omega,\mathcal F,\mathbb P)$ satisfy $\E\left[U^j\right]=\E\left[V^j\right]$ for $1\leq j\leq k-1$. Let $G_\sigma\sim N(0,\sigma^2)$ be independent of them, with density $\varphi_\sigma$, where $\sigma>0$. Then, for $a>0$,
\begin{equation*}
 \begin{aligned}
 &\left|W_1\bigl(\Law(G_\sigma+aU),\Law(G_\sigma+aV)\bigr)
 -\frac{a^k}{k!}\left|\E\left[U^k\right]-\E\left[V^k\right]\right|\norm{\varphi_\sigma^{(k-1)}}_1\right|\\
 &\qquad\leq \frac{a^{k+1}}{(k+1)!}
 \left(\E\left[|U|^{k+1}\right]+\E\left[|V|^{k+1}\right]\right)\norm{\varphi_\sigma^{(k)}}_1.
 \end{aligned}
\end{equation*}
\end{lemma}
\begin{proof}
Let $H_\sigma(x)\coloneqq\int_{-\infty}^x\varphi_\sigma(t)\dd t$ for $x\in\R$. For $x,u\in\R$, Taylor expansion gives
\[
 H_\sigma(x-au)=\sum_{j=0}^k\frac{(-au)^j}{j!}H_\sigma^{(j)}(x)
 +\frac{(-au)^{k+1}}{k!}\int_0^1(1-r)^k
                         \varphi_\sigma^{(k)}(x-rau)\dd r.
\]
Let $R_u:\R\to\R$ denote the last term of this Taylor expansion. Fubini's theorem and translation invariance of the $L^1$ norm give
\begin{align*}
 \norm{R_u}_1
 &\leq\frac{a^{k+1}|u|^{k+1}}{k!}
  \int_0^1(1-r)^k\norm{\varphi_\sigma^{(k)}(\cdot-rau)}_1\dd r\\
 &=
 \frac{a^{k+1}|u|^{k+1}}{k!}
  \norm{\varphi_\sigma^{(k)}}_1\int_0^1(1-r)^k\dd r
 =\frac{a^{k+1}|u|^{k+1}}{(k+1)!}\norm{\varphi_\sigma^{(k)}}_1.
\end{align*}
Fubini's theorem applies because $\E\left[|U|^{k+1}\right]+\E\left[|V|^{k+1}\right]<\infty$. Subtracting the two expectations cancels the zeroth term and all matched moments. With $R\coloneqq\E[R_U]-\E[R_V]\in L^1(\R)$, this gives
\[
 F_{\Law(G_\sigma+aU)}-F_{\Law(G_\sigma+aV)}
 =\frac{(-a)^k}{k!}(\E\left[U^k\right]-\E\left[V^k\right])\varphi_\sigma^{(k-1)}+R
\]
and
\[
 \norm R_1\leq\E\left[\norm{R_U}_1\right]+\E\left[\norm{R_V}_1\right].
\]
CDF duality and the reverse triangle inequality prove the bound.
\end{proof}
\begin{theorem}[Fixed smoothing and a uniform bound across both regimes]\label{thm:crossover}
Let $G\sim N(0,1)$ be independent of $Y,Z$ in \eqref{eq:YZ}. For fixed $\sigma>0$,
\begin{equation}\label{eq:sharp-limit}
 W_1\bigl(\Law(\sigma G+aY),\Law(\sigma G+aZ)\bigr)
 =\frac{\sqrt{\tfrac{8}{\pi e}}}{6\sigma^2}a^3+O(a^4\sigma^{-3})
\end{equation}
as $\tfrac{a}{\sigma}\xrightarrow[a\downarrow0]{}0$. The remainder in \eqref{eq:sharp-limit} is bounded in absolute value by $\tfrac{11a^4\|\varphi'''\|_1}{(24\sigma^3)}$. There are numerical constants $0<c\leq C<\infty$, independent of $a,\sigma$, such that for every $a>0$ and $\sigma\geq0$,
\begin{equation}\label{eq:crossover}
 c\,a\min\left\{1,\frac{a^2}{\sigma^2}\right\}
 \leq W_1\bigl(\Law(\sigma G+aY),\Law(\sigma G+aZ)\bigr)
 \leq C\,a\min\left\{1,\frac{a^2}{\sigma^2}\right\},
\end{equation}
where the minimum is defined as one when $\sigma=0$.
\end{theorem}
\begin{proof}
Lemma~\ref{lem:moments} with $k=3$ and Gaussian scaling gives
\[
 \left|W_1(\Law(\sigma G+aY),\Law(\sigma G+aZ))-
                       \frac{\sqrt{\tfrac{8}{\pi e}}a^3}{6\sigma^2}\right|
       \leq\frac{11a^4}{24\sigma^3}\norm{\varphi'''}_1.
\]
Define
\[
 \mapdef{D}{[0,\infty)}{[0,\infty)}{r}{W_1(\Law(rG+Y),\Law(rG+Z))}.
\]
Scaling gives
\begin{align*}
 W_1(\Law(\sigma G+aY),\Law(\sigma G+aZ))
 &=aD\left(\frac{\sigma}{a}\right).
\end{align*}
Then the triangle inequality and a common coupling with $G$ yield, using $\E\left[|G|\right]=\sqrt{\tfrac{2}{\pi}}$,
\begin{align*}
 |D(r)-D(s)|
 &\leq
 W_1(\Law(rG+Y),\Law(sG+Y))+W_1(\Law(rG+Z),\Law(sG+Z))\\
 &\leq2\E\left[|G|\right]\,|r-s|
 =2\sqrt{\frac{2}{\pi}}|r-s|.
\end{align*}
If $D(r)=0$, cancelling the nonzero Gaussian characteristic function would give $\Law(Y)=\Law(Z)$, contradicting Proposition~\ref{prop:cubic}. Set $A\coloneqq\tfrac{\sqrt{\frac{8}{\pi e}}}{6}>0$ and $B\coloneqq\tfrac{11\|\varphi'''\|_1}{24}$. The first estimate with $a=1$ gives
\[
 |D(r)-Ar^{-2}|\leq Br^{-3}.
\]
Set $R\coloneqq\max\left\{1,\tfrac{2B}{A}\right\}$. For $r\geq R$, $\tfrac{A}{2r^2}\leq D(r)\leq\tfrac{3A}{2r^2}$. On $[0,R]$, continuity and positivity bound $\tfrac{D(r)}{\min\{1,r^{-2}\}}$ above and away from zero, with the denominator set to one at $r=0$. Scaling proves \eqref{eq:crossover}.
\end{proof}
Since $\tfrac{\theta(a)}{a^3}\xrightarrow[a\downarrow0]{}1$,
\[
 \lim_{a\downarrow0}
 \frac{W_1(\Law(\sigma G+aY),\Law(\sigma G+aZ))}{\theta(a)}
 =\frac{\sqrt{\tfrac{8}{\pi e}}}{6\sigma^2}>0.
\]
The linear exponent is sharp: its constant grows at least as $v^{-1}$ for $v=\sigma^2\downarrow0$, and $\sigma\lesssim a$ gives the degenerate scale within this family. The time-profile rates of Appendices~\ref{sup:sec:one-sided} and~\ref{sup:sec:stable} concern different classes.
\subsection{Proof of Theorem~\ref{thm:modulus}.}\label{proof:thm:modulus}
\begin{proof}
For $\varepsilon>0$, independent Gaussian coupling gives $W_1(\lambda_i,\lambda_i*g_\varepsilon)\leq \sqrt{\tfrac{2}{\pi}}\sqrt\varepsilon$. The triangle inequality and Corollary~\ref{cor:Gaussian} give
\begin{align}
 W_1(\lambda_0,\lambda_1)
 &\leq
 W_1(\lambda_0,\lambda_0*g_\varepsilon)
 +W_1(\lambda_0*g_\varepsilon,\lambda_1*g_\varepsilon)
 +W_1(\lambda_1*g_\varepsilon,\lambda_1)
 \nonumber\\
 &\leq
 2\sqrt{\frac{2}{\pi}}\sqrt\varepsilon
 +W_1(\lambda_0*g_\varepsilon,\lambda_1*g_\varepsilon)
 \nonumber\\ &\leq
 2\sqrt{\frac{2}{\pi}}\sqrt\varepsilon+
 \left(\frac{\sqrt{\tfrac{2}{\pi}}}{2\sqrt{v+\varepsilon}}+
                \frac{\sqrt{\tfrac{8}{\pi e}}}{6(v+\varepsilon)}\right)\delta.
 \label{eq:regularized-terminal}
\end{align}
If $v\geq\delta^{\frac{2}{3}}$, send $\varepsilon\downarrow0$ and use $v\leq1$ to obtain $\tfrac{C\delta}{v}$. Otherwise set $\varepsilon\coloneqq\delta^{\frac{2}{3}}$ to obtain $C\delta^{\frac{1}{3}}$. Since $\min\left\{\delta^{\frac{1}{3}},\tfrac{\delta}{v}\right\}$ lies between $\tfrac{\delta}{v+\delta^{\frac{2}{3}}}$ and twice this quantity, the upper bound follows. At $v=0$, we interpret this minimum as $\delta^{\frac{1}{3}}$. For the lower bound, take $a\coloneqq\delta^{\frac{1}{3}}$ and the pair $\sqrt vG+aY$, $\sqrt vG+aZ$ of Theorem~\ref{thm:crossover}. Both jump variances equal $a^2\leq1$, their discrepancy is $\tfrac{2a^3}{2+a}\leq\delta$, and hence this pair is admissible in \eqref{eq:modulus-definition}. Substituting $a=\delta^{\frac{1}{3}}$ in the bound covering both regimes and using $\max\{v,\delta^{\frac{2}{3}}\}\leq v+\delta^{\frac{2}{3}}$ therefore gives
\begin{align*}
 \mathfrak M(v,\delta)
 &\byeq{eq:modulus-definition}{\geq}
 W_1(\Law(\sqrt vG+aY),\Law(\sqrt vG+aZ))\\
 &\byeq{eq:crossover}{\geq}
 c\,a\min\left\{1,\frac{a^2}{v}\right\}
 =c\min\left\{\delta^{\frac{1}{3}},\frac{\delta}{v}\right\}
 \geq c\,\frac{\delta}{v+\delta^{\frac{2}{3}}}.
\end{align*}
At $v=0$, the minima have the conventions already specified. The total variances agree.
\end{proof}
\section{Multivariate response and stability}\label{app:multi}
\subsection{A multidimensional transport extension}\label{sec:multitangent}
For an integer $d\geq1$, let $I_d$ be the identity matrix, $\|A\|_{\mathrm{op}}\coloneqq\sup_{\substack{x\in\R^d\\|x|=1}}|Ax|$ for $A\in\R^{d\times d}$, and $A\succeq B$ mean that the symmetric matrix $A-B$ is positive semidefinite. We write $A\succ0$ when $A$ is positive definite. Let the reference law have full-compensation triplet $(m,V,N)$. Its mean is $m\in\R^d$ and its Gaussian covariance $V\in\R^{d\times d}$ is symmetric positive definite. The finite jump measure $N\in\cM_f^+(\R^d)$ is supported on $A_{r_0,R_0}\coloneqq\{y\in\R^d:r_0\leq|y|\leq R_0\}$, where $0<r_0\leq R_0<\infty$. Let $\Pi\in\cM_f^+(\R_y^d\times\R_z^d)$ have first marginal $N$ and satisfy $|z|\leq Z_0$ for $\Pi$-almost every $(y,z)$, with $Z_0\geq0$ finite. Let $\eta\in\cM_f(\R^d)$ be supported on the same annulus. Choose a mean direction $\alpha\in\R^d$ and a symmetric covariance direction $C\in\R^{d\times d}$. Fix $\ell>0$. For $\varepsilon\in\R$ near zero, define
\begin{subequations}\label{eq:multi-lift}
\begin{equation}
 m_\varepsilon\coloneqq m+\varepsilon\alpha,
\end{equation}
and
\begin{equation}
 V_\varepsilon\coloneqq V+\varepsilon C\succeq\ell I_d,
\end{equation}
and
\begin{equation}
 N_\varepsilon\coloneqq(y+\varepsilon z)_\#\Pi+\varepsilon\eta\geq0.
\end{equation}
\end{subequations}
Here admissibility means that the two displayed measure and matrix inequalities hold; for sufficiently small $|\varepsilon|$, the transported jumps stay away from zero. Let $\lambda_\varepsilon$ be the law of this triplet, let $p_\varepsilon$ be its density, and write $\lambda\coloneqq\lambda_0$ and $p\coloneqq p_0$. For Borel $A\subset\R^d$ and $1\leq j,k\leq d$, define
\[
 J_j(A)\coloneqq\int_{A\times\R^d}z_j\Pi(\dd y,\dd z)
\]
and
\[
 S_{jk}(A)\coloneqq\int_{A\times\R^d}z_jz_k\Pi(\dd y,\dd z).
\]
Write $J\coloneqq(J_1,\ldots,J_d)$ for the vector of finite signed measures and $S\coloneqq(S_{jk})_{j,k=1}^d$ for the matrix of finite signed measures. Then set $\tau\coloneqq\eta-\operatorname{div}_yJ\in\mathcal D'(\R^d)$. For every Borel set $A$, $S(A)$ is symmetric and positive semidefinite. The distributional divergence is $\operatorname{div}_yJ\coloneqq\sum_{j=1}^d\partial_{y_j}J_j$; its action is specified by $$\langle\phi,\tau\rangle=\int_{\R^d}\phi(y)\eta(\dd y)+\int_{\R^d}\nabla\phi(y)\cdot J(\dd y)$$ for $\phi\in C_c^\infty(\R^d)$. For $f\in C_b^2(\R^d)$, set $R_f(x,y)\coloneqq f(x+y)-f(x)-y\cdot\nabla f(x)$, $(x,y)\in\R^d\times\R^d$.
\par\Needspace{12\baselineskip}
The first two generator coefficients $\mathcal A,\mathcal B:C_b^2(\R^d)\to C_b(\R^d)$ are
\begin{equation*}
 \begin{aligned}
 \mathcal Af(x)\coloneqq\alpha\cdot\nabla f(x)&+\tfrac12 C:\nabla^2f(x)
     +\int_{\R^d} R_f(x,y)\eta(\dd y)\\
 &+\int_{\R^d}[\nabla f(x+y)-\nabla f(x)]\cdot J(\dd y)
 \end{aligned}
\end{equation*}
and
\begin{equation*}
 \mathcal Bf(x)\coloneqq\tfrac12\sum_{j,k=1}^d\int_{\R^d}\partial_{jk}f(x+y)S_{jk}(\dd y).
\end{equation*}
Their spatial adjoints map $W^{2,1}(\R^d)$ into $L^1(\R^d)$ and act as
\[
 \begin{aligned}
 \mathcal A^*u(x)
 =-\alpha\cdot\nabla u(x)+\tfrac12 C:\nabla^2u(x)
 +&\int_{\R^d}[u(x-y)-u(x)+y\cdot\nabla u(x)]\eta(\dd y)\\
 &-\sum_{j=1}^d\int_{\R^d}[\partial_j u(x-y)-\partial_j u(x)]J_j(\dd y)
 \end{aligned}
\]
and
\[
 \mathcal B^*u(x)
 =\tfrac12\sum_{j,k=1}^d\int_{\R^d}\partial_{jk}u(x-y)S_{jk}(\dd y).
\]
Gaussian smoothing ensures that the reference density $p$ belongs to $W^{k,1}(\R^d)$ for every integer $k\geq0$.
\begin{theorem}[First and second responses in $\R^d$]\label{thm:multi-tangent}
Under \eqref{eq:multi-lift}, put $r_1\coloneqq\mathcal A^*p$ and $r_2\coloneqq\mathcal B^*p+\tfrac12\mathcal A^*r_1$.  Then $r_1,r_2\in L^1(\R^d)$ have zero integrals and finite first absolute moments, and
\begingroup
\begin{equation*}
 \big\|p_\varepsilon-p-\varepsilon r_1-\varepsilon^2r_2\big\|_{\mathrm{KR}}\leq C_0|\varepsilon|^3,
\end{equation*}
so
\begin{equation*}
 W_1(\lambda_\varepsilon,\lambda)=\norm{\varepsilon r_1+\varepsilon^2r_2}_{\mathrm{KR}}+O(|\varepsilon|^3).
\end{equation*}
\endgroup
The constant depends only on $d,\ell$, the support and velocity bounds, $N(\R^d)$, $\|\eta\|_{\mathrm{TV}}$, $|\alpha|$, and $\|C\|_{\mathrm{op}}$. The first response vanishes exactly when $(\alpha,C,\tau)=0$, and $\|r_1\|_{\mathrm{KR}}$ defines a norm on the triples $(\alpha,C,\tau)$ just described, is invariant under simultaneous rotations, and contracts under common independent infinitely divisible convolution. If $\alpha=C=\eta=0$ and $J=0$, then
\begingroup
\begin{equation*}
 W_1(\lambda_\varepsilon,\lambda)=\varepsilon^2\norm{\mathcal B^*p}_{\mathrm{KR}}+O(|\varepsilon|^3),
\end{equation*}
and
\begin{equation*}
 \int_{\R^{2d}}|z|^2\Pi(\dd y,\dd z)>0\Longrightarrow\norm{\mathcal B^*p}_{\mathrm{KR}}>0.
\end{equation*}
\endgroup
\end{theorem}
Appendix~\ref{sup:sec:multivariate-proof} proves the expansion in the KR norm. For the matrix characteristic $\boldsymbol\Gamma_\varepsilon\coloneqq V_\varepsilon\delta_0+yy^\top N_\varepsilon$, the derivative $\dot{\boldsymbol\Gamma}\coloneqq\left.\partial_\varepsilon\boldsymbol\Gamma_\varepsilon\right|_{\varepsilon=0}$ in distributions is
\begin{equation}\label{eq:orientation-correction}
 \dot{\boldsymbol\Gamma}=C\delta_0+yy^\top\eta+yJ^\top+Jy^\top
                  -\sum_{\ell=1}^d\partial_{y_\ell}(yy^\top J_\ell).
\end{equation}
For a matrix field $F:\R^d\to\R^{d\times d}$ with smooth compactly supported entries, written $F\in C_c^\infty(\R^d;\R^{d\times d})$, we have
\[
 \frac{\dd}{\dd\varepsilon}\bigg|_{\varepsilon=0}
 [(y+\varepsilon z)(y+\varepsilon z)^\top:F(y+\varepsilon z)]
 =(yz^\top+zy^\top):F(y)+yy^\top:DF(y)[z].
\]
Integrating and using the distributional product rule gives \eqref{eq:orientation-correction}. The coefficient of $\varepsilon^2$ is
\[
 \int_{\R^{2d}}\!\left[zz^\top:F(y)+(yz^\top+zy^\top):DF(y)[z]
               +\tfrac12yy^\top:D^2F(y)[z,z]\right]\Pi(\dd y,\dd z),
\]
where $DF(y)[z]=\sum_{\ell=1}^d z_\ell\partial_\ell F(y)$ and $D^2F(y)[z,z]=\sum_{j,k=1}^d z_jz_k\partial_{jk}F(y)$. For raw balanced redistribution, $J=0$ and $C=\eta=0$. Hence
\[
 \operatorname{Cov}(\lambda_\varepsilon)
 -\operatorname{Cov}(\lambda)
 =\int_{\R^{2d}}[(y+\varepsilon z)(y+\varepsilon z)^\top-yy^\top]\Pi(\dd y,\dd z)
 =\varepsilon^2S(\R^d).
\]
Variance-weighted one-dimensional remeshing preserves total variance. The matrix characteristic above retains the required form $V\delta_0+yy^\top N$.
\subsection{Proof of the multivariate second-order response}
\label{sup:sec:multivariate-proof}
\begin{proof}
Keep the notation of Theorem~\ref{thm:multi-tangent} and first take $f\in C_b^\infty(\R^d)$. Let $\mathcal D_\varepsilon$ be the difference between the terminal L\'evy generators. Then
\begin{align*}
 \mathcal D_\varepsilon f(x) &=\varepsilon\alpha\cdot\nabla f(x) +\frac\varepsilon2 C:\nabla^2f(x) +\varepsilon\int_{\R^d} R_f(x,y)\eta(\dd y)\\
 &\quad+\int_{\R^{2d}}\bigl[R_f(x,y+\varepsilon z)-R_f(x,y)\bigr]\Pi(\dd y,\dd z).
\end{align*}
Taylor's formula in $y$ gives
\begin{align*}
 R_f(x,y+\varepsilon z)-R_f(x,y)
 &=\varepsilon z\cdot[\nabla f(x+y)-\nabla f(x)]
   +\frac{\varepsilon^2}{2}z^\top\nabla^2f(x+y)z\\
 &\quad+\frac{\varepsilon^3}{2}\int_0^1(1-r)^2
       D^3f(x+y+r\varepsilon z)[z,z,z]\dd r.
\end{align*}
Thus $\mathcal D_\varepsilon=\varepsilon\mathcal A+ \varepsilon^2\mathcal B+\mathcal E_\varepsilon$, where
\[
 \begin{aligned}
 \mathcal Af(x)=\alpha\cdot\nabla f(x)&+\tfrac12C:\nabla^2f(x)
  +\int_{\R^d} R_f(x,y)\eta(\dd y)\\
 &+\int_{\R^d}[\nabla f(x+y)-\nabla f(x)]\cdot J(\dd y)
 \end{aligned}
\]
and
\[
 \mathcal Bf(x)=\tfrac12\sum_{j,k=1}^d\int_{\R^d}
                 \partial_{jk}f(x+y)S_{jk}(\dd y).
\]
The second-order remainder satisfies
\[
 \norm{\mathcal E_\varepsilon f}_\infty
 \leq C|\varepsilon|^3\max_{\substack{\beta\in\N_0^d\\|\beta|=3}}\norm{D^\beta f}_\infty.
\]
The support and velocity bounds also imply
\[
 \norm{\mathcal D_\varepsilon f}_\infty
  +|\varepsilon|\norm{\mathcal Af}_\infty
 \leq C|\varepsilon|\sum_{j=1}^2\norm{D^jf}_\infty.
\]
The first-order remainder satisfies
\[
 \norm{(\mathcal D_\varepsilon-\varepsilon\mathcal A)f}_\infty
 \leq C\varepsilon^2\norm{D^2f}_\infty.
\]
These operators commute with translations and spatial derivatives, so
\[
 \mathcal D_\varepsilon^2-\varepsilon^2\mathcal A^2
 =(\mathcal D_\varepsilon-\varepsilon\mathcal A)
   (\mathcal D_\varepsilon+\varepsilon\mathcal A).
\]
The preceding bounds therefore give
\[
 \norm{(\mathcal D_\varepsilon^2-\varepsilon^2\mathcal A^2)f}_\infty
 \leq C|\varepsilon|^3\sum_{j=3}^4\norm{D^jf}_\infty.
\]
The third iterate satisfies
\[
 \norm{\mathcal D_\varepsilon^3f}_\infty
 \leq C|\varepsilon|^3\sum_{j=3}^6\norm{D^jf}_\infty.
\]
Let $\Psi,\Psi_\varepsilon:\R^d\to\mathbb C$ be the characteristic exponents of $\lambda,\lambda_\varepsilon$. For $u\in[0,1]$, let $\pi_{u,\varepsilon}$ be the probability law with exponent $(1-u)\Psi+u\Psi_\varepsilon$, and put $Q_{u,\varepsilon}h(x)\coloneqq\int_{\R^d} h(x+y)\pi_{u,\varepsilon}(\dd y)$, with $Q_0\coloneqq Q_{0,\varepsilon}$. For $h\in\cS(\R^d)$, Fourier inversion yields
\[
 \partial_u^jQ_{u,\varepsilon}h
       =\mathcal D_\varepsilon^jQ_{u,\varepsilon}h,
 \qquad j=1,2,3.
\]
Indeed, the Fourier integrands are bounded by $C(1+|\xi|)^6|\widehat h(\xi)|$, since the exponents have nonpositive real parts and at most quadratic growth. Taylor's formula in $u$ gives
\[
 (Q_{1,\varepsilon}-Q_0)h
 =\mathcal D_\varepsilon Q_0h
  +\tfrac12\mathcal D_\varepsilon^2Q_0h
  +\tfrac12\int_0^1(1-u)^2
               \mathcal D_\varepsilon^3Q_{u,\varepsilon}h\dd u.
\]
The Gaussian covariance of $\pi_{u,\varepsilon}$ is $V+u\varepsilon C\succeq\ell I_d$. Factor its density as $g_{\ell I_d}*\zeta_{u,\varepsilon}$ with $\zeta_{u,\varepsilon}$ a probability law. For a multi-index $\beta$ with $1\leq|\beta|\leq6$, choose $j$ with $\beta_j>0$. For $\Lip(h)\leq1$, Young's inequality and the Lipschitz bound give
\[
 \begin{aligned}
 \norm{D^\beta Q_{u,\varepsilon}h}_\infty
 &=\norm{(\partial_jh)*(D^{\beta-e_j}g_{\ell I_d})* \check\zeta_{u,\varepsilon}}_\infty\\
 &\leq\norm{\partial_jh}_\infty \norm{D^{\beta-e_j}g_{\ell I_d}}_1
 \leq \norm{D^{\beta-e_j}g_{\ell I_d}}_1,
 \end{aligned}
\]
where $e_j$ is the $j$th coordinate multi-index. Substituting these bounds gives
\[
 \sup_{\substack{h\in\cS(\R^d)\\\Lip(h)\leq1}}\left|
 \int_{\R^d} h(x)(\lambda_\varepsilon-\lambda)(\dd x)
 -\varepsilon\mathcal A Q_0h(0)
 -\varepsilon^2\left(\mathcal B+\tfrac12\mathcal A^2\right)Q_0h(0)
 \right|\leq C_0|\varepsilon|^3.
\]
The spatial adjoints are
\[
 \begin{aligned}
 \mathcal A^*u
 =-\alpha\cdot\nabla u+\tfrac12 C:\nabla^2u
   &+\int_{\R^d}[u(\cdot-y)-u+y\cdot\nabla u]\eta(\dd y)\\
 &+\int_{\R^d}[\nabla u-\nabla u(\cdot-y)]\cdot J(\dd y)
 \end{aligned}
\]
and
\[
 \mathcal B^*u=\tfrac12\sum_{j,k=1}^d
                         \int_{\R^d}\partial_{jk}u(\cdot-y)S_{jk}(\dd y).
\]
Writing $p=g_V*\rho$ with $\rho\in\cP_1(\R^d)$, Fubini's theorem, a change of variables, and the triangle inequality give, for every multi-index $\beta$,
\begin{align*}
 \int_{\R^d}(1+|x|)|D^\beta p(x)|\dd x
 &\leq\int_{\R^d}\rho(\dd y)
                  \int_{\R^d}(1+|x|)|D^\beta g_V(x-y)|\dd x\\
 &=\int_{\R^d}\rho(\dd y)
                  \int_{\R^d}(1+|z+y|)|D^\beta g_V(z)|\dd z\\
 &\leq\left(1+\int_{\R^d}|y|\rho(\dd y)\right)
             \int_{\R^d}(1+|z|)|D^\beta g_V(z)|\dd z<\infty.
\end{align*}
Compact translations preserve this property. Each adjoint term integrates to zero, so $r_1=\mathcal A^*p$ and $r_2=\mathcal B^*p+\tfrac12(\mathcal A^*)^2p$ are zero-mass densities with finite first absolute moments. Integration by parts gives $\mathcal A Q_0h(0)=\int_{\R^d} h(x)r_1(x)\dd x$ and $\left(\mathcal B+\tfrac{\mathcal A^2}{2}\right)Q_0h(0)=\int_{\R^d} h(x)r_2(x)\dd x$. Lemma~\ref{lem:approximation} and these weighted integrability bounds extend the estimate to all Lipschitz tests. Hence
\[
 \norm{p_\varepsilon-p-\varepsilon r_1-\varepsilon^2r_2}_{\mathrm{KR}}
 \leq C_0|\varepsilon|^3,
\]
which also gives
\[
 \left|W_1(\lambda_\varepsilon,\lambda)
       -\norm{\varepsilon r_1+\varepsilon^2r_2}_{\mathrm{KR}}\right|
 \leq C_0|\varepsilon|^3.
\]
To prove injectivity, define the first exponent variation $\dot\Psi(\xi)\coloneqq\left.\partial_\varepsilon\Psi_\varepsilon(\xi)\right|_{\varepsilon=0}$, $\xi\in\R^d$. Then
\[
\widehat r_1(\xi)=\dot\Psi(-\xi)\widehat p(\xi)
\]
where
\[
\dot\Psi(\xi)=i\alpha\cdot\xi-\tfrac12\xi^\top C\xi
  +\langle e^{i\xi\cdot y}-1-i\xi\cdot y,\tau\rangle.
\]
Since $\widehat p=e^{\Psi(-\cdot)}$ never vanishes, $r_1=0$ implies $\dot\Psi=0$. The inverse Fourier transform of the exponential term is $\tau$ up to reflection, whereas the polynomial terms have inverse transforms supported at $0$. Restricting to $\R^d\setminus\{0\}$ gives $\tau=0$, because $\operatorname{supp}\tau\subset A_{r_0,R_0}$. The identity $i\alpha\cdot\xi-\tfrac{\xi^\top C\xi}{2}=0$ gives $\alpha=C=0$ by its first and second derivatives at $\xi=0$. The converse follows from the same multiplier identity. Thus the linear map $(\alpha,C,\tau)\mapsto r_1$ is injective and its KR norm is a norm. For an orthogonal matrix $U$, $h\mapsto h\circ U$ preserves the test class, so simultaneous rotations preserve this norm. If $\zeta\in\cP_1(\R^d)$ is an independent infinitely divisible convolution factor, then the response becomes $r_1\mapsto r_1*\zeta$. Since $\Lip(h*\check\zeta)\leq\Lip(h)$, Fubini's theorem and Kantorovich duality give
\begin{align*}
 \norm{r_1*\zeta}_{\mathrm{KR}}
 &=\sup_{\substack{h:\R^d\to\R\\\Lip(h)\leq1}}\left|\int_{\R^d} h(x)(r_1*\zeta)(x)\dd x\right|\\
 &=\sup_{\substack{h:\R^d\to\R\\\Lip(h)\leq1}}\left|\int_{\R^d}(h*\check\zeta)(x)r_1(x)\dd x\right| \leq\norm{r_1}_{\mathrm{KR}},
\end{align*}
which proves contraction. If $\alpha=C=\eta=0$ and $J=0$, then $\mathcal A=0$ and $r_2=\mathcal B^*p$. Moreover,
\[
 \widehat{\mathcal B^*p}(\xi)
 =-\tfrac12\widehat p(\xi)\sum_{j,k=1}^d
                                     \xi_j\xi_k\widehat S_{jk}(\xi).
\]
The matrix $M\coloneqq S(\R^d)=\int_{\R^{2d}} zz^\top\Pi(\dd y,\dd z)$ is positive semidefinite and $\operatorname{tr}M=\int_{\R^{2d}}|z|^2\Pi(\dd y,\dd z)$. If this trace is positive, choose a unit vector $e$ with $e^\top Me>0$. Continuity of the Fourier transforms of the finite measures $S_{jk}$ yields
\[
 \lim_{t\to0}\frac{\widehat{\mathcal B^*p}(te)}{t^2}
 =-\tfrac12e^\top Me<0.
\]
Thus $\mathcal B^*p\ne0$ and $\|\mathcal B^*p\|_{\mathrm{KR}}>0$.
\end{proof}
\subsection{Rotation-invariant multivariate comparison}\label{sup:sec:matrix}
Let $\lambda_i\in\cP_1(\R^d)$, $i\in\{0,1\}$, be additive terminal laws with characteristic exponents
\[
\Psi_i(\xi)=i\xi\cdot m_i-\tfrac12\xi^\top V_i\xi
     +\int_{\R^d}(e^{i\xi\cdot y}-1-i\xi\cdot y)N_i(\dd y),
\]
with
\[
\int_{\R^d}|y|^2N_i(\dd y)<\infty.
\]
Here $m_i\in\R^d$, $V_i$ are symmetric positive semidefinite matrices, and $N_i$ are L\'evy measures on $\R^d\setminus\{0\}$. Define the symmetric matrix-valued finite measures $\mu_i(\dd y)\coloneqq yy^\top N_i(\dd y)$ and then $\Gamma_i\coloneqq V_i\delta_0+\mu_i.$ Their entries are finite signed measures because $|y_jy_k|\leq|y|^2$, and $\operatorname{tr}\mu_i(\R^d)=\int_{\R^d}|y|^2N_i(\dd y)$. Let $\mathrm{Sym}_d$ be the vector space of real symmetric $d\times d$ matrices. For a matrix field $F:\R^d\to\mathrm{Sym}_d$, write
\[
 \norm F_{\infty,\mathrm{op}}\coloneqq\sup_{y\in\R^d}\norm{F(y)}_{\mathrm{op}}
\]
and
\[
 \Lip_{\mathrm{op}}(F)\coloneqq\sup_{\substack{y,z\in\R^d\\y\ne z}}\frac{\norm{F(y)-F(z)}_{\mathrm{op}}}{|y-z|}.
\]
For bounded Lipschitz $F$, set
\begin{subequations}\label{sup:eq:matrix-discrepancy}
\begin{equation}
 \norm{F}_{\BL,\mathrm{op}}
   \coloneqq\sup_{y\in\R^d}\norm{F(y)}_{\mathrm{op}}
       +\sup_{\substack{y,z\in\R^d\\y\ne z}}\frac{\norm{F(y)-F(z)}_{\mathrm{op}}}{|y-z|}
\end{equation}
and then
\begin{equation}
 d_{\mathrm{mat}}(M_0,M_1)
   \coloneqq\sup_{\substack{F:\R^d\to\mathrm{Sym}_d\\\norm{F}_{\BL,\mathrm{op}}\leq1}}
       \left|\int_{\R^d}\tr\bigl(F(y)(M_0-M_1)(\dd y)\bigr)\right|.
\end{equation}
\end{subequations}
Here $M_0,M_1$ are symmetric matrix-valued finite measures and
\[
 \int_{\R^d}\operatorname{tr}(F(y)M(\dd y))
 =\sum_{j,k=1}^d\int_{\R^d} F_{jk}(y)M_{kj}(\dd y).
\]
The discrepancy is finite, and it separates measures: fields $F=fA$, with scalar bounded Lipschitz $f$ and constant $A\in\mathrm{Sym}_d$, separate all symmetric components. For $U\in\R^{d\times d}$ with $U^\top U=I_d$, define $M^U(B)\coloneqq UM(U^\top B)U^\top$ for Borel $B\subset\R^d$, and $F^U(y)\coloneqq UF(U^\top y)U^\top$ for $y\in\R^d$. Then $\|F^U\|_{\BL,\mathrm{op}}=\|F\|_{\BL,\mathrm{op}}$, $d_{\mathrm{mat}}(M_0^U,M_1^U)=d_{\mathrm{mat}}(M_0,M_1)$ and
\[
 \int_{\R^d}\operatorname{tr}(F^U(y)M^U(\dd y))=\int_{\R^d}\operatorname{tr}(F(y)M(\dd y)).
\]
Let $\pi_u$ have exponent $\psi_u\coloneqq(1-u)\Psi_1+u\Psi_0$, $u\in[0,1]$, as in \eqref{eq:interpolation}. Assume that, for almost every $u$, $\pi_u$ has a density $\rho_u\in W^{2,1}(\R^d)$ and
\begingroup
\begin{equation*}
 M_{1,d}\coloneqq\int_0^1\sup_{\substack{e\in\R^d\\|e|=1}}\norm{D_e\rho_u}_1\dd u<\infty
\end{equation*}
and
\begin{equation*}
 M_{2,d}\coloneqq\int_0^1\sup_{\substack{e\in\R^d\\|e|=1}}\norm{D_e^2\rho_u}_1\dd u<\infty.
\end{equation*}
\endgroup
We assume that the derivative norms are measurable. Their directional suprema equal the suprema over a countable dense subset of the unit sphere, since $e\mapsto D_e\rho_u$ and $e\mapsto D_e^2\rho_u$ are continuous in $L^1$. For a real matrix $A$, let $\|A\|_*\coloneqq\operatorname{tr}\sqrt{A^\top A}$ be its nuclear norm.
\begin{theorem}[Matrix-valued terminal stability]\label{sup:thm:matrix}
Under these assumptions,
\begin{equation}\label{sup:eq:matrix-stability}
 W_1(\lambda_0,\lambda_1) \leq|m_0-m_1|+\frac{M_{1,d}}2\norm{V_0-V_1}_* +\left(\frac{M_{1,d}}2+\frac{M_{2,d}}6\right) d_{\mathrm{mat}}(\mu_0,\mu_1).
\end{equation}
We also have the joint bound
\begin{equation*}
 W_1(\lambda_0,\lambda_1)
 \leq|m_0-m_1|+\left(\frac{M_{1,d}}2+\frac{M_{2,d}}6\right)
                    d_{\mathrm{mat}}(\Gamma_0,\Gamma_1).
\end{equation*}
If $V_0,V_1$ are positive definite, put $V_u\coloneqq(1-u)V_1+uV_0$. Then
\begin{subequations}\label{sup:eq:matrix-gaussian}
\begin{equation}
 M_{1,d}\leq \sqrt{\frac{2}{\pi}}\int_0^1\lambda_{\min}(V_u)^{-\frac{1}{2}}\dd u
\end{equation}
and
\begin{equation}
 M_{2,d}\leq \sqrt{\frac{8}{\pi e}}\int_0^1\lambda_{\min}(V_u)^{-1}\dd u,
\end{equation}
\end{subequations}
where $\lambda_{\min}$ denotes the smallest eigenvalue. A common lower bound $V_0,V_1\succeq\ell I_d$, $\ell>0$, therefore gives $M_{1,d}\leq \sqrt{\tfrac{2}{\pi}}\ell^{-\frac{1}{2}}$ and $M_{2,d}\leq \sqrt{\tfrac{8}{\pi e}}\ell^{-1}$.
\end{theorem}
\begin{remark*}
These bounds have no explicit dimension factor.
\end{remark*}
\begin{proof}
For $h\in\cS(\R^d)$ with $\Lip(h)\leq1$, set $q_u\coloneqq h*\check\rho_u$ and $A_j(u)\coloneqq\sup_{\substack{e\in\R^d\\|e|=1}}\|D_e^j\rho_u\|_1$, $j=1,2$. For unit vectors $e,z$, differentiation gives
\[
D_e^2q_u=(D_eh)*D_e\check\rho_u, 
\qquad
D_zD_e^2q_u=(D_zh)*D_e^2\check\rho_u.
\]
Translation is continuous for the $L^1$ factors, so these distributional derivatives are continuous and hence classical. Since $\|D_eh\|_\infty\leq1$ and $\|A\|_{\mathrm{op}}=\sup_{\substack{e\in\R^d\\|e|=1}}|e^\top Ae|$ for $A\in\mathrm{Sym}_d$, $\|\nabla q_u\|_\infty\leq1$, $\|\nabla^2q_u\|_{\infty,\mathrm{op}}\leq A_1(u)$ and
\[
 |e^\top[\nabla^2q_u(x)-\nabla^2q_u(x')]e|
 \leq|x-x'|\sup_{\substack{z\in\R^d\\|z|=1}}\norm{D_zD_e^2q_u}_\infty
 \leq A_2(u)|x-x'|.
\]
Taking the supremum in $e$ gives $\Lip_{\mathrm{op}}(\nabla^2q_u)\leq A_2(u)$. For $u\in[0,1]$ and $x\in\R^d$, the matrix field $F_{u,x}:\R^d\to\mathrm{Sym}_d$ defined by
\[
 F_{u,x}(y)\coloneqq\int_0^1(1-r)\nabla^2q_u(x+ry)\dd r
\]
satisfies
\[
 \sup_{y\in\R^d}\norm{F_{u,x}(y)}_{\mathrm{op}}\leq\tfrac12A_1(u).
\]
Its Lipschitz bound is
\[
 \begin{aligned}
 \norm{F_{u,x}(y)-F_{u,x}(y')}_{\mathrm{op}}
 \leq A_2(u)|y-y'|\int_0^1r(1-r)\dd r
 =\tfrac16A_2(u)|y-y'|.
 \end{aligned}
\]
Taylor's formula also gives
\[
 q_u(x+y)-q_u(x)-y\cdot\nabla q_u(x)
 =y^\top F_{u,x}(y)y
\]
and $F_{u,x}(0)=\tfrac12\nabla^2q_u(x)$. Thus $\|F_{u,x}\|_{\BL,\mathrm{op}} \leq \tfrac{A_1(u)}{2}+\tfrac{A_2(u)}{6}$. Let $\mathcal D$ be the difference of the terminal generators and write $\delta m\coloneqq m_0-m_1$ and $\delta V\coloneqq V_0-V_1$ and $\delta\mu\coloneqq\mu_0-\mu_1$. Fourier differentiation as in Lemma~\ref{lem:interpolation} gives
\[
 \int_{\R^d} h(x)(\lambda_0-\lambda_1)(\dd x)=\int_0^1\mathcal Dq_u(0)\dd u.
\]
Indeed, $|e^{ia}-1-ia|\leq \tfrac{a^2}{2}$ gives
\[
 |\Psi_i(\xi)|\leq|m_i||\xi|+\tfrac12\left(\norm{V_i}_{\mathrm{op}}
 +\int_{\R^d}|y|^2N_i(\dd y)\right)|\xi|^2,
\]
so $C(|\xi|+|\xi|^2)|\widehat h(\xi)|$ is an integrable majorant. Using $y^\top Fy=\operatorname{tr}(Fyy^\top)$, we have
\[
 \mathcal Dq_u(0)=\delta m\cdot\nabla q_u(0)
 +\tfrac12\operatorname{tr}(\delta V\nabla^2q_u(0))
 +\int_{\R^d}\operatorname{tr}(F_{u,0}(y)\delta\mu(\dd y)).
\]
The inequality $|\operatorname{tr}(AB)|\leq\|A\|_*\|B\|_{\mathrm{op}}$, together with the matrix discrepancy and the derivative bounds above, therefore gives
\begin{align*}
 |\mathcal Dq_u(0)|
 &\byeq{sup:eq:matrix-discrepancy}{\leq}
   |\delta m|\norm{\nabla q_u}_\infty
   +\tfrac12\norm{\delta V}_*\norm{\nabla^2q_u}_{\infty,\mathrm{op}} +\norm{F_{u,0}}_{\BL,\mathrm{op}}
                                  d_{\mathrm{mat}}(\mu_0,\mu_1)\\
 &\leq
   |\delta m|+\tfrac12A_1(u)\norm{\delta V}_* +\left(\tfrac12A_1(u)+\tfrac16A_2(u)\right)
                                  d_{\mathrm{mat}}(\mu_0,\mu_1).
\end{align*}
Integrating in $u$ proves \eqref{sup:eq:matrix-stability}. For the joint bound, write instead
\[
 \mathcal Dq_u(0)=\delta m\cdot\nabla q_u(0)
             +\int_{\R^d}\operatorname{tr}\bigl(F_{u,0}(y)(\Gamma_0-\Gamma_1)(\dd y)\bigr),
\]
using $F_{u,0}(0)=\tfrac{\nabla^2q_u(0)}{2}$. Lemma~\ref{lem:approximation} extends the two estimates from Schwartz tests to all Lipschitz tests. The first moments justify dominated convergence. For $V\succ0$ and $e\in\R^d$ with $|e|=1$, set $a\coloneqq e^\top V^{-1}e$. Differentiation gives
\[
D_eg_V(x)=-(e^\top V^{-1}x)g_V(x)
\]
and
\[
D_e^2g_V(x)=\bigl[(e^\top V^{-1}x)^2-a\bigr]g_V(x).
\]
If $Z\sim N(0,V)$, then $e^\top V^{-1}Z\overset{\mathrm{law}}=\sqrt aG$ for $G\sim N(0,1)$. Hence
\[
 \norm{D_eg_V}_1
 =\E\left[|e^\top V^{-1}Z|\right] =\sqrt a\,\E\left[|G|\right] =\sqrt{\frac{2}{\pi}}\sqrt a
\]
and
\[
 \norm{D_e^2g_V}_1
 =\E\left[\left|(e^\top V^{-1}Z)^2-a\right|\right] =a\E\left[|G^2-1|\right] =4a\varphi(1)=\sqrt{\frac{8}{\pi e}}a.
\]
The second identity follows from $\int_1^\infty(x^2-1)\varphi(x)\dd x= [-x\varphi(x)]_1^\infty=\varphi(1)$ and $\E\left[G^2-1\right]=0$. Since $\rho_u=g_{V_u}*\zeta_u$ for a probability law $\zeta_u$, Young's inequality gives $\|D_e^j\rho_u\|_1\leq\|D_e^jg_{V_u}\|_1$, $j=1,2$. Finally, $\sup_{\substack{e\in\R^d\\|e|=1}}e^\top V_u^{-1}e=\lambda_{\min}(V_u)^{-1}$. Taking suprema and then integrating proves \eqref{sup:eq:matrix-gaussian}.
\end{proof}
For $d=1$, $d_{\mathrm{mat}}=d_{\FM}$. Writing $\mu_i^{jk}(\dd y)\coloneqq y_jy_kN_i(\dd y)$ gives the coordinate bounds
\begingroup
\begin{equation*}
 d_{\mathrm{mat}}(\mu_0,\mu_1)\leq\sum_{j,k=1}^d d_{\FM}^{(d)}(\mu_0^{jk},\mu_1^{jk})
\end{equation*}
and
\begin{equation*}
 \norm{V_0-V_1}_*\leq\sum_{j,k=1}^d|V_0^{jk}-V_1^{jk}|.
\end{equation*}
\endgroup
Indeed, $|F_{jk}(y)|\leq\|F(y)\|_{\mathrm{op}}$ and $|F_{jk}(y)-F_{jk}(z)|\leq\|F(y)-F(z)\|_{\mathrm{op}}$, so every entry of an admissible $F$ has scalar BL norm at most one. For the nuclear norm, let $e_1,\ldots,e_d$ be the coordinate vectors of $\R^d$, write $A=\sum_{j,k=1}^d A_{jk}e_je_k^\top$, and use $\|e_je_k^\top\|_*=1$. The matrix formulation is invariant under a change of orthonormal basis. The characteristics and their total trace can grow with $d$.
\section{State-dependent response and comparison}\label{app:state}
\subsection{A state-dependent second-order response}\label{sup:sec:state-response}
For state-dependent generators, the second response keeps track of time order. Let $c\in C_b^3(\R)$ be nonnegative and let $\Pi\in\cP(\R^2)$ have compact support. Its first marginal is the probability measure $F\in\cP(\R)$, and $F_\varepsilon\coloneqq (y+\varepsilon z)_\#\Pi$. For $f\in C_b^2(\R)$, consider the generator
\begin{equation*}
 \cL_\varepsilon f(x)\coloneqq \tfrac12 f''(x)
       +c(x)\int_{\R}[f(x+y)-f(x)]F_\varepsilon(\dd y).
\end{equation*}
This gives a linear map $\cL_\varepsilon:C_b^2(\R)\to C_b^0(\R)$. Bounded-rate thinning constructs the real-valued Markov process $X^\varepsilon$. We use uncompensated jumps, whose full-compensation drift is $c(x)\int_{\R^2}(y+\varepsilon z)\Pi(\dd y,\dd z)$. For $k\in\{0,1,2,3\}$, define the weighted Sobolev space
\[
X_k\coloneqq \left\{u\in W^{k,1}(\R):
             \sum_{j=0}^k\int_{\R}(1+|x|)|D^ju(x)|\dd x<\infty\right\}.
\]
Equip this space with the norm
\[
\norm u_{X_k}\coloneqq \sum_{j=0}^k\norm{(1+|x|)D^ju}_1.
\]
For $t\geq0$, let $P_t^{\varepsilon,*}:X_k\to X_k$ be the forward semigroup: for an initial density $u$, $P_t^{\varepsilon,*}u$ is its density at time $t$; the action extends linearly to signed $u\in X_k$. Write $P_t^*\coloneqq P_t^{0,*}$. Fix $T>0$ and let every $X^\varepsilon$ start with the same probability density $p_0\in X_3$. The first forward generator coefficient is
\begin{subequations}\label{sup:eq:state-operators}
\begin{equation}
\mathcal A^*u\coloneqq -\int_{\R^2} zD(cu)(\cdot-y)\Pi(\dd y,\dd z).
\end{equation}
The second coefficient is
\begin{equation}
\mathcal B^*u\coloneqq \tfrac12\int_{\R^2} z^2D^2(cu)(\cdot-y)\Pi(\dd y,\dd z).
\end{equation}
\end{subequations}
In particular, $\mathcal A^*:X_{k+1}\to X_k$ for $k\leq2$ and $\mathcal B^*:X_{k+2}\to X_k$ for $k\leq1$ are bounded linear maps.
\begin{theorem}[Time-ordered second response with bounded jump activity]
\label{sup:thm:state-response}
Let $p_T^\varepsilon\coloneqq P_T^{\varepsilon,*}p_0$. Define the Bochner integrals
\begin{equation}\label{sup:eq:state-q1}
q_1(T)\coloneqq \int_0^T P_{T-t}^*\mathcal A^*P_t^*p_0\dd t,
\end{equation}
and
\begin{equation}\label{sup:eq:state-q2}
q_2(T)\coloneqq \int_0^T P_{T-t}^*\mathcal B^*P_t^*p_0\dd t
 +\int_{0<s<t<T}P_{T-t}^*\mathcal A^*P_{t-s}^*
                         \mathcal A^*P_s^*p_0\dd s\dd t.
\end{equation}
Then $q_1(T)\in X_2$ and $q_2(T)\in X_1$. Moreover, for $|\varepsilon|\leq1$,
\begin{equation}\label{sup:eq:state-remainder}
 \norm{p_T^\varepsilon-p_T^0-\varepsilon q_1(T)-\varepsilon^2q_2(T)}_{X_0}
                  \leq C_T|\varepsilon|^3.
\end{equation}
The constant $C_T$ depends only on $T$, the support bound of $\Pi$, $\|c\|_{C_b^3}$ and $\|p_0\|_{X_3}$. Both coefficients have zero integral. For fixed $T$, define $H_j(T):\R\to\R$ by $H_j(T)(x)\coloneqq \int_{-\infty}^xq_j(T,z)\dd z$, $j=1,2$. These primitives belong to $L^1(\R)$ and
\begin{equation}\label{sup:eq:state-W1}
 \left|W_1(\Law(X_T^\varepsilon),\Law(X_T^0))
            -\norm{\varepsilon H_1(T)+\varepsilon^2H_2(T)}_1\right|
                  \leq C_T|\varepsilon|^3.
\end{equation}
If $\int_{\R^2} zf(y)\Pi(\dd y,\dd z)=0$ for every bounded measurable $f:\R\to\R$, then $\mathcal A^*=0$, so $H_1(T)=0$ and the time-ordered double integral vanishes. Consequently, $W_1(\Law(X_T^\varepsilon),\Law(X_T^0)) =\varepsilon^2\|H_2(T)\|_1+O(|\varepsilon|^3)$.
\end{theorem}
\begin{proof}
For $a\in\R$, let $T_a:X_k\to X_k$ be translation, defined by $T_au\coloneqq u(\cdot-a)$. A change of variables and the inequality $1+|x+a|\leq(1+|a|)(1+|x|)$ give
\[
\begin{aligned}
 \norm{T_au}_{X_k}
 &=
      \sum_{j=0}^k\int_{\R}(1+|x+a|)|D^ju(x)|\dd x\\
 &\leq
      (1+|a|)\sum_{j=0}^k\int_{\R}(1+|x|)|D^ju(x)|\dd x
 =(1+|a|)\norm u_{X_k}.
\end{aligned}
\]
For the heat semigroup $H_t:X_k\to X_k$, defined by $H_tu\coloneqq g_t*u$ for $t>0$ and $H_0u\coloneqq u$, apply this estimate under the convolution integral. Then Tonelli's theorem and the first absolute Gaussian moment give, for $t>0$,
\[
\begin{aligned}
 \norm{H_tu}_{X_k} 
 &\leq\int_{\R} g_t(a)\norm{T_au}_{X_k}\dd a\\
 &\leq
        \norm u_{X_k}\int_{\R}(1+|a|)g_t(a)\dd a
 =\left(1+\sqrt{\frac{2}{\pi}}\sqrt t\right)\norm u_{X_k}.
\end{aligned}
\]
The same bound holds at $t=0$. Multiplication by $c$ is bounded on $X_k$ for $k\leq3$, by $D^j(cu)=\sum_{i=0}^j\binom ji(D^{j-i}c)D^iu$. Hence the jump adjoint
\[
 \mathcal J_\varepsilon^*u
 \coloneqq \int_{\R^2} T_{y+\varepsilon z}(cu)\Pi(\dd y,\dd z)-cu
\]
is bounded on each $X_k$, uniformly over $|\varepsilon|\leq1$. The perturbation series
\[
 P_t^{\varepsilon,*}
 =H_t+\sum_{n=1}^{\infty}\int_{0<t_1<\cdots<t_n<t}
  H_{t-t_n}\mathcal J_\varepsilon^*
  H_{t_n-t_{n-1}}\cdots
  \mathcal J_\varepsilon^*H_{t_1}\dd t_1\cdots\dd t_n
\]
converges in operator norm on $X_k$: its $n$th summand has norm at most $\tfrac{M_T^{n+1}B_k^nt^n}{n!}$, where $M_T\coloneqq 1+\sqrt{\tfrac{2}{\pi}}\sqrt T$ and $B_k\coloneqq \sup_{\varepsilon\in[-1,1]} \|\mathcal J_\varepsilon^*\|_{X_k\to X_k}<\infty$. Thus
\[
 \sup_{t\in[0,T],\,\varepsilon\in[-1,1]}
 \norm{P_t^{\varepsilon,*}}_{X_k\to X_k}\leq C_T,
 \qquad k=0,1,2,3.
\]
These strongly continuous semigroups agree with the bounded-rate process and preserve positivity and mass. Define $\Delta_\varepsilon:X_k\to X_k$ by $\Delta_\varepsilon\coloneqq \mathcal J_\varepsilon^*-\mathcal J_0^*$. For $v\coloneqq cu$ and weak derivatives through order three,
\[
T_{y+\varepsilon z}v-T_yv
 =-\varepsilon z\int_0^1T_{y+r\varepsilon z}Dv\dd r.
\]
The first-order Taylor remainder is
\[
T_{y+\varepsilon z}v-T_yv+\varepsilon zT_yDv
 =\varepsilon^2z^2\int_0^1(1-r)
                     T_{y+r\varepsilon z}D^2v\dd r.
\]
The second-order remainder is
\[
T_{y+\varepsilon z}v-T_yv+\varepsilon zT_yDv
                  -\tfrac12\varepsilon^2z^2T_yD^2v
 =-\tfrac12\varepsilon^3z^3\int_0^1(1-r)^2
                     T_{y+r\varepsilon z}D^3v\dd r.
\]
The shifts are uniformly bounded. Integrating against $\Pi$ therefore gives
\[
\norm{\Delta_\varepsilon u}_{X_k}
       \leq C|\varepsilon|\norm u_{X_{k+1}} \qquad (k\leq2).
\]
For the first-order remainder, this gives
\[
\norm{(\Delta_\varepsilon-\varepsilon\mathcal A^*)u}_{X_k}
       \leq C\varepsilon^2\norm u_{X_{k+2}} \qquad (k\leq1).
\]
At second order, we obtain
\[
\norm{(\Delta_\varepsilon-\varepsilon\mathcal A^*
                       -\varepsilon^2\mathcal B^*)u}_{X_0}
       \leq C|\varepsilon|^3\norm u_{X_3}.
\]
Iterating Duhamel's identity
\[
 P_T^{\varepsilon,*}=P_T^*
       +\int_0^T P_{T-t}^*\Delta_\varepsilon
                                  P_t^{\varepsilon,*}\dd t
\]
gives, on $X_3$,
\[
\begin{aligned}
P_T^{\varepsilon,*}-P_T^*
 =\int_0^TP_{T-t}^*\Delta_\varepsilon P_t^*\dd t
 +\int_{0<s<t<T}P_{T-t}^*\Delta_\varepsilon P_{t-s}^*
                                      \Delta_\varepsilon P_s^*\dd s\dd t
       +R_{3,\varepsilon}.
\end{aligned}
\]
Here the remainder is
\[
R_{3,\varepsilon}
 \coloneqq \int_{0<r<s<t<T}P_{T-t}^*\Delta_\varepsilon P_{t-s}^*
       \Delta_\varepsilon P_{s-r}^*\Delta_\varepsilon
                         P_r^{\varepsilon,*}\dd r\dd s\dd t.
\]
Three applications of the generator difference successively map $X_3$ to $X_2$, $X_1$, and $X_0$, so
\[
 \norm{R_{3,\varepsilon}}_{X_3\to X_0}
 \leq
       C_T|\varepsilon|^3\int_0^T\int_0^t\int_0^s\dd r\dd s\dd t =C_T|\varepsilon|^3\frac{T^3}{6}.
\]
For the double integral, use the identity
\[
 \Delta_\varepsilon P_{t-s}^*\Delta_\varepsilon -\varepsilon^2\mathcal A^*P_{t-s}^*\mathcal A^* =(\Delta_\varepsilon-\varepsilon\mathcal A^*) P_{t-s}^*\Delta_\varepsilon +\varepsilon\mathcal A^*P_{t-s}^* (\Delta_\varepsilon-\varepsilon\mathcal A^*).
\]
Both terms have norm at most $C_T|\varepsilon|^3$ from $X_3$ to $X_0$: the derivative losses are respectively $1+2$ and $2+1$. In the single integral, the generator difference expands as $\varepsilon\mathcal A^*+\varepsilon^2\mathcal B^*$ with the same cubic error. Integrating over the time simplexes proves \eqref{sup:eq:state-remainder} and \eqref{sup:eq:state-q1}--\eqref{sup:eq:state-q2}. For $v\in W^{1,1}(\R)$, $\int_{\R} Dv(x)\dd x=0$. Applying this also to $Dv$ gives $\int_{\R} D^2v(x)\dd x=0$ when $v\in W^{2,1}$. Thus both operators in \eqref{sup:eq:state-operators} have zero mass. Since $P_t^*$ preserves mass, we obtain $\int_{\R} q_j(T,x)\dd x=0$. For every zero-mass $u\in X_0$, define its primitive $U:\R\to\R$ by
\[
 U(x)\coloneqq \int_{-\infty}^xu(z)\dd z =
   \int_{\R}\bigl[\one_{\{z\leq x\}}-\one_{\{0\leq x\}}\bigr]u(z)\dd z.
\]
Tonelli's theorem gives
\[
 \norm U_1
 \leq\int_{\R}|u(z)|\int_{\R}
       |\one_{\{z\leq x\}}-\one_{\{0\leq x\}}|\dd x\dd z =\int_{\R}|z||u(z)|\dd z \leq\norm u_{X_0}.
\]
The one-dimensional identity $W_1(\lambda,\widetilde\lambda) =\|F_\lambda-F_{\widetilde\lambda}\|_1$, where $F_\lambda$ is the CDF, and the reverse triangle inequality now give \eqref{sup:eq:state-W1}. Under conditional balance, the signed measure $B\mapsto\int_{\R^2}\one_B(y)z\Pi(\dd y,\dd z)$ vanishes, so \eqref{sup:eq:state-operators} gives $\mathcal A^*=0$.
\end{proof}
A Gaussian initial density belongs to $X_3$. The next example gives a positive quadratic coefficient for a nonconstant intensity.
\subsubsection{A nonzero state-dependent quadratic coefficient}
\label{sup:sec:state-positive}
Take $c(x)\coloneqq 1+\tfrac12\sin x$ and $p_0\coloneqq g_1$, and choose $\Pi\coloneqq \tfrac12\delta_{(a,-1)}+\tfrac12\delta_{(a,1)}$ with $a\in\R\setminus\{0\}$. Then $F=\delta_a$ and $F_\varepsilon=\tfrac{\delta_{a-\varepsilon}+ \delta_{a+\varepsilon}}{2}$, and $\mathcal A^*=0$. The second generator coefficient is
\[
\mathcal B^*p=\tfrac12D^2(cp)(\cdot-a).
\]
The second response therefore reduces to
\[
q_2(T)=\int_0^TP_{T-t}^*\mathcal B^*P_t^*p_0\dd t.
\]
Strong continuity on $X_2$ and $X_0$, together with $\mathcal B^*:X_2\to X_0$, implies
\[
\begin{aligned}
 \norm{\frac{q_2(T)}{T}-\mathcal B^*p_0}_{X_0} 
 \leq\sup_{0\leq t\leq T} \norm{P_{T-t}^*\mathcal B^*(P_t^*p_0-p_0)}_{X_0}
 +\sup_{0\leq t\leq T} \norm{(P_{T-t}^*-I)\mathcal B^*p_0}_{X_0}
 \xrightarrow[T\downarrow0]{}0.
\end{aligned}
\]
Applying the primitive bound gives
\begingroup
\begin{equation*}
\frac{H_2(T)}{T}\xrightarrow[T\downarrow0]{L^1}\tfrac12D(cp_0)(\cdot-a).
\end{equation*}

Consequently,
\begin{equation*}
\frac{\norm{H_2(T)}_1}{T}\xrightarrow[T\downarrow0]{}\tfrac12\norm{D(cp_0)}_1>0.
\end{equation*}
\endgroup
Here $H_2(T)$ is the primitive of $q_2(T)$. The derivative in the limit above is
\[
D(cp_0)(x)=\left[\tfrac12\cos x-x\left(1+\tfrac12\sin x\right)\right]
                \varphi(x).
\]

It follows that $D(cp_0)(0)=\tfrac12\varphi(0)>0.$ Thus $\|H_2(T)\|_1>0$ for every sufficiently small fixed $T>0$.
\subsection{One-sided smoothing for non-additive comparisons}
\label{sup:sec:one-sided}
Let $X^i$, $i\in\{0,1\}$, be real-valued Markov semimartingales with initial laws in $\cP_1(\R)$. For $f\in C^2(\R)$ with bounded $f',f''$, their full-compensation generators act by
\[
\begin{aligned}
 \cL_i(t)f(x)&\coloneqq b_i(t,x)f'(x)+\tfrac12a_i(t,x)f''(x)\\
 &\quad +\int_{\R}\bigl[f(x+y)-f(x)-yf'(x)\bigr]\nu_i(t,x,\dd y).
\end{aligned}
\]
Here $b_i:[0,T]\times\R\to\R$ and $a_i:[0,T]\times\R\to[0,\infty)$ are measurable, and $\nu_i$ are measurable L\'evy kernels. Assume
\begin{equation}\label{sup:eq:semimartingale-integrability}
 \E\left[\int_0^T\left[|b_i(t,X^i_{t-})|+a_i(t,X^i_{t-})
       +\int_{\R} y^2\nu_i(t,X^i_{t-},\dd y)\right]\dd t\right]<\infty.
\end{equation}
Writing $X_t^i=X_0^i+\int_0^tb_i(s,X^i_{s-})\dd s+M_t^i$, where $M^i$ is the square-integrable martingale part starting at zero, the martingale isometry and Doob's inequality give
\begin{align*}
 \E\left[\sup_{t\in[0,T]}|X^i_t|\right]
 &\leq\E\left[|X^i_0|\right]+\E\left[\int_0^T|b_i(t,X^i_{t-})|\dd t\right]
                    +2(\E\left[|M_T^i|^2\right])^{\frac{1}{2}}\\
 &=\E\left[|X^i_0|\right]+\E\left[\int_0^T|b_i(t,X^i_{t-})|\dd t\right]\\
 &\quad+2\left(\E\left[\int_0^T\left[a_i(t,X^i_{t-})
       +\int_{\R} y^2\nu_i(t,X^i_{t-},\dd y)\right]\dd t\right]\right)^{\frac{1}{2}}<\infty.
\end{align*}
Define the measurable discrepancy profiles $r_b,r_a,r_\nu:[0,T]\to[0,\infty]$ along $X^0$:
\begingroup
\begin{equation*}
r_b(t)\coloneqq \E\left[|b_0(t,X^0_{t-})-b_1(t,X^0_{t-})|\right],
\qquad 
r_a(t)\coloneqq \E\left[|a_0(t,X^0_{t-})-a_1(t,X^0_{t-})|\right],
\end{equation*}
and
\begin{equation*}
r_\nu(t)\coloneqq \E\left[d_{\FM}\bigl(y^2\nu_0(t,X^0_{t-},\dd y),
                              y^2\nu_1(t,X^0_{t-},\dd y)\bigr)\right].
\end{equation*}
\endgroup
Assume $r_b,r_a,r_\nu\in L^1(0,T)$ and set $\theta_0\coloneqq W_1(\Law(X^0_0),\Law(X^1_0))$, $\theta_b\coloneqq \|r_b\|_{L^1(0,T)}$, $\theta_a\coloneqq \|r_a\|_{L^1(0,T)}$, and $J_q\coloneqq \|r_\nu\|_{L^q(0,T)}$. We use $J_q$ when it is finite. These bounds also control the coefficients of model~$1$ along $X^0$: the test $f=1$ gives
\[
 \E\left[\int_{\R} y^2\nu_1(t,X^0_{t-},\dd y)\right]
 \leq\E\left[\int_{\R} y^2\nu_0(t,X^0_{t-},\dd y)\right]+r_\nu(t),
\]
with analogous bounds for $|b_1|$ and $a_1$ using $r_b$ and $r_a$.
\begin{lemma}[Generator comparison]\label{sup:lem:comparison}
For a Borel terminal test $f:\R\to\R$ whose expectations below exist, let $P^1_{t,T}f:\R\to\R$ be the transition operator 
\begingroup
$$P^1_{t,T}f(x)\coloneqq \E[f(X_T^1)\mid X_t^1=x].$$ 
\endgroup
Let $v:[0,T]\times\R\to\R$ be defined by $v(t,x)\coloneqq P^1_{t,T}f(x)$ and solve $\partial_tv+\cL_1(t)v=0$ with terminal condition $v(T)=f$. Suppose localized It\^o's formula applies to $v(t,X^0_t)$, with stopped martingales of zero expectation and uniformly integrable stopped terminal values. If 
\begingroup
$$\E\left[\int_0^T|((\cL_0-\cL_1)v)(t,X^0_{t-})|\dd t\right]<\infty,$$ 
\endgroup
then
\begin{equation}\label{sup:eq:comparison}
\begin{aligned}
 \E\left[f(X^0_T)\right]-\E\left[f(X^1_T)\right] 
 &=\E\left[P^1_{0,T}f(X^0_0)-P^1_{0,T}f(X^1_0)\right]\\
 &\quad +\int_0^T\E\left[((\cL_0(t)-\cL_1(t))v(t))(X^0_{t-})\right]\dd t.
\end{aligned}
\end{equation}
\end{lemma}
\begin{proof}
Choose increasing localizing times $\sigma_n$ with $\sigma_n\xrightarrow[n\to\infty]{}\infty$, and set $\tau_n\coloneqq \sigma_n\wedge T$. It\^o's formula and the backward equation give
\begin{align*}
 \E\left[v(\tau_n,X^0_{\tau_n})\right]-\E\left[v(0,X^0_0)\right]
 &=\E\left[\int_0^{\tau_n}(\partial_t+\cL_0(t))v(t,X^0_{t-})\dd t\right]\\
 &=\E\left[\int_0^T\one_{\{t\leq\tau_n\}}
                   ((\cL_0-\cL_1)v)(t,X^0_{t-})\dd t\right].
\end{align*}
By uniform integrability, the left-hand side converges to $\E\left[f(X_T^0)\right]-\E\left[P_{0,T}^1f(X_0^0)\right]$. Dominated convergence gives the integral without stopping on the right. Subtract $\E\left[f(X_T^1)\right]=\E\left[P_{0,T}^1f(X_0^1)\right]$ to obtain the identity.
\end{proof}
For $v\in C^3(\R)$ with bounded first three derivatives and fixed $x\in\R$, Taylor's formula gives the field
\[
 \mapdef{\Phi_{v,x}}{\R}{\R}{y}
      {\displaystyle\int_0^1(1-r)v''(x+ry)\dd r}.
\]
It satisfies
\[
 v(x+y)-v(x)-yv'(x)=y^2\Phi_{v,x}(y).
\]
Since $\int_0^1(1-r)\dd r=\tfrac{1}{2}$ and $\int_0^1r(1-r)\dd r=\tfrac{1}{6}$,
\begin{align*}
 \norm{\Phi_{v,x}}_{\BL}
 &=\norm{\Phi_{v,x}}_\infty+\Lip(\Phi_{v,x})\\
 &\leq
    \int_0^1(1-r)\norm{v''}_\infty\dd r
    +\int_0^1r(1-r)\norm{v'''}_\infty\dd r\\
 &=
       \tfrac12\norm{v''}_\infty+\tfrac16\norm{v'''}_\infty.
\end{align*}
Lemma~\ref{lem:Taylor} then gives
\[
\begin{aligned}
 \E\left[|((\cL_0-\cL_1)v)(t,X^0_{t-})|\right] 
 \leq r_b(t)\norm{v'}_\infty+\tfrac12r_a(t)\norm{v''}_\infty
 +r_\nu(t)\left(\tfrac12\norm{v''}_\infty
                              +\tfrac16\norm{v'''}_\infty\right).
\end{aligned}
\]
\subsubsection{External regularization and its baseline rate}
\label{sup:sec:external-regularization}
For $\Lip(h)\leq1$ and $\varepsilon>0$, set $h_\varepsilon\coloneqq h*g_\varepsilon$. With $G\sim N(0,1)$, the Lipschitz property and the first absolute Gaussian moment give
\[
\begin{aligned}
 |h_\varepsilon(x)-h(x)| 
 =\left|\E\left[h(x+\sqrt\varepsilon G)-h(x)\right]\right|
 &\leq\E\left[|h(x+\sqrt\varepsilon G)-h(x)|\right]
 \\ &\leq\sqrt\varepsilon\,\E\left[|G|\right]
 =\sqrt{\frac{2}{\pi}}\sqrt\varepsilon.
\end{aligned}
\]
Since $Dh\in L^\infty$ and $\|Dh\|_\infty\leq1$, $D^jh_\varepsilon=(Dh)*D^{j-1}g_\varepsilon$, so
\begin{subequations}\label{sup:eq:external-gaussian}
\begin{equation}
\norm{h-h_\varepsilon}_\infty\leq \sqrt{\frac{2}{\pi}}\sqrt\varepsilon.
\end{equation}

The first derivative satisfies
\begin{equation}
\norm{h'_\varepsilon}_\infty\leq1.
\end{equation}

For the second derivative,
\begin{equation}
\norm{h''_\varepsilon}_\infty\leq \sqrt{\frac{2}{\pi}}\varepsilon^{-\frac{1}{2}}.
\end{equation}

For the third derivative,
\begin{equation}
\norm{h'''_\varepsilon}_\infty\leq \sqrt{\frac{8}{\pi e}}\varepsilon^{-1}.
\end{equation}
\end{subequations}
\begin{proposition}[External regularization]
\label{sup:prop:external-regularization}
Assume \eqref{sup:eq:comparison} for $h_\varepsilon$ and, for $k=1,2,3$,
\begin{equation}\label{sup:eq:derivative-transfer}
 \norm{D^kP^1_{t,T}f}_\infty
 \leq C_F\sum_{j=1}^k\norm{D^jf}_\infty,
 \qquad 0\leq t\leq T,
\end{equation}
for smooth $f$ of at most linear growth with bounded first three derivatives. If $\theta_a,J_1\leq1$, then
\begin{equation*}
 W_1(\Law(X^0_T),\Law(X^1_T))
 \leq CC_F\left(\theta_0+\theta_b+\theta_a^{\frac{1}{2}}+J_1^{\frac{1}{3}}\right).
\end{equation*}
\end{proposition}
\begin{proof}
Put $v_\varepsilon(t)\coloneqq P_{t,T}^1h_\varepsilon$. For $0<\varepsilon\leq1$, \eqref{sup:eq:derivative-transfer} and \eqref{sup:eq:external-gaussian} give $\|v_\varepsilon'\|_\infty\leq C_F$, $\|v_\varepsilon''\|_\infty\leq CC_F\varepsilon^{-\frac{1}{2}}$, and $\|v_\varepsilon'''\|_\infty\leq CC_F\varepsilon^{-1}$. The initial discrepancy is at most $C_F\theta_0$. Substituting the bounds into \eqref{sup:eq:comparison}, adding $2\sqrt{\tfrac{2}{\pi}}\sqrt\varepsilon$ for terminal regularization, and taking the supremum over all $h:\R\to\R$ with $\Lip(h)\leq1$ gives
\[
 W_1(\Law(X_T^0),\Law(X_T^1))\byeq{sup:eq:comparison}{\leq}
 CC_F\left(\theta_0+\theta_b+\sqrt\varepsilon
                     +\theta_a\varepsilon^{-\frac{1}{2}}+J_1\varepsilon^{-1}\right).
\]
If $\varepsilon=\max\left\{\theta_a,J_1^{\frac{2}{3}}\right\}>0$, then $\sqrt\varepsilon\leq\theta_a^{\frac{1}{2}}+J_1^{\frac{1}{3}}$. The Gaussian term satisfies $\theta_a\varepsilon^{-\frac{1}{2}}\leq\theta_a^{\frac{1}{2}}$, and the jump term satisfies $J_1\varepsilon^{-1}\leq J_1^{\frac{1}{3}}$. If both discrepancies vanish, let $\varepsilon\downarrow0$.
\end{proof}
For an additive reference, derivatives commute with convolution, so \eqref{sup:eq:derivative-transfer} holds with $C_F=1$. Section~\ref{sup:sec:flow} gives a flow criterion for the same bound.
\subsubsection{Time-resolved Gaussian smoothing}
\label{sup:sec:time-smoothing}
\begin{assumption}[One-sided derivative bounds]\label{sup:ass:one-sided}
For every $\Lip(h)\leq1$ and $0<\varepsilon\leq1$, the comparison identity holds for $v_\varepsilon(t)\coloneqq P^1_{t,T}h_\varepsilon$. There is $C_S\geq1$, independent of $h,\varepsilon,t$, such that
\begin{subequations}\label{sup:eq:one-sided-derivatives}
\begin{equation}
\norm{Dv_\varepsilon(t)}_\infty\leq C_S.
\end{equation}

The second derivative is bounded by
\begin{equation}
\norm{D^2v_\varepsilon(t)}_\infty\leq C_S(\varepsilon+T-t)^{-\frac{1}{2}}.
\end{equation}

The third derivative satisfies
\begin{equation}
\norm{D^3v_\varepsilon(t)}_\infty\leq C_S(\varepsilon+T-t)^{-1}.
\end{equation}
\end{subequations}
\end{assumption}
\begin{proposition}[An additive reference verifying the assumption]
\label{sup:prop:additive-comparison}
Suppose model~$1$ has deterministic characteristics satisfying \eqref{eq:add-int} and $a_1(t)\geq\lambda>0$ almost everywhere. Then Assumption~\ref{sup:ass:one-sided} holds for the comparisons above, with $C_S\coloneqq \max\left\{1,\sqrt{\tfrac{2}{\pi}}\kappa^{-\frac{1}{2}},\sqrt{\tfrac{8}{\pi e}}\kappa^{-1}\right\}$ and $\kappa\coloneqq \min\{1,\lambda\}$.
\end{proposition}
\begin{remark*}
This includes discrete jump laws.
\end{remark*}
\begin{proof}
Let $V_{t,T}\coloneqq \int_t^Ta_1(s)\dd s$ and let $\eta_{t,T}\in\cP_1(\R)$ be the drift-and-jump increment law. Independence of the Gaussian and jump parts gives
\[
v_\varepsilon(t)=h*g_{\varepsilon+V_{t,T}}*\check\eta_{t,T}.
\]

The variance satisfies
\[
\varepsilon+V_{t,T}\geq\kappa(\varepsilon+T-t).
\]
Convolution with a probability law contracts the supremum norm. Apply \eqref{eq:gauss-L1} with variance $\varepsilon+V_{t,T}$ to obtain the three derivative bounds. For fixed $\varepsilon>0$, the increment Dynkin identity and translation invariance give
\[
 v_\varepsilon(t)=h_\varepsilon+
                 \int_t^T\cL_1(s)v_\varepsilon(s)\dd s.
\]
Set $M(s)\coloneqq a_1(s)+\int_{\R} y^2\nu_1(s,\dd y)\in L^1(0,T)$. For $k=0,1,2$, Taylor's bound for the compensated jump term gives
\begin{align*}
 \norm{\cL_1(s)D^kv_\varepsilon(s)}_\infty
 &\leq|b_1(s)|\norm{D^{k+1}v_\varepsilon(s)}_\infty
       +\tfrac12M(s)\norm{D^{k+2}v_\varepsilon(s)}_\infty\\
 &\leq|b_1(s)|\norm{D^{k+1}h_\varepsilon}_\infty
       +\tfrac12M(s)\norm{D^{k+2}h_\varepsilon}_\infty.
\end{align*}
All derivatives of $h_\varepsilon$ of positive order are bounded. The right-hand sides are deterministic and integrable. They justify absolute continuity in time through two spatial derivatives. Mollifying in time and then using dominated convergence gives the localized It\^o formula. Equation~\eqref{sup:eq:semimartingale-integrability} bounds the model~$0$ terms. The comparison integrand is integrable by the profile bounds. Finally, increment first moments are uniformly bounded, so
\[
 |v_\varepsilon(\tau_n,X^0_{\tau_n})|
 \leq C_{h,\varepsilon}\bigl(1+\sup_{t\in[0,T]}|X_t^0|\bigr).
\]
This integrable bound gives uniform integrability. Lemma~\ref{sup:lem:comparison} now applies.
\end{proof}
\begin{theorem}[One-sided master estimate]\label{sup:thm:master}
Under Assumption~\ref{sup:ass:one-sided}, for every $0<\varepsilon\leq1$,
\begin{equation}\label{sup:eq:master}
\begin{aligned}
 W_1(\Law(X^0_T),\Law(X^1_T)) 
 \leq\ &C\biggl[\theta_0+\theta_b+\sqrt\varepsilon
       +\int_0^T\frac{r_a(t)}{\sqrt{\varepsilon+T-t}}\dd t\\
 &+\int_0^Tr_\nu(t)
       \left((\varepsilon+T-t)^{-\frac{1}{2}}+(\varepsilon+T-t)^{-1}\right)\dd t
       \biggr],
\end{aligned}
\end{equation}
where $C$ depends only on $C_S$.
\end{theorem}
\begin{proof}
Let $s\coloneqq T-t$. Weighted Taylor testing and \eqref{sup:eq:one-sided-derivatives} give
\[
\begin{aligned}
 \E\left[|((\cL_0-\cL_1)v_\varepsilon)(t,X^0_{t-})|\right]
 \ \ \byeq{sup:eq:one-sided-derivatives}{\leq}\ \ 
 C_S\biggl[r_b(t)+\frac{r_a(t)}{2\sqrt{\varepsilon+s}}
 +r_\nu(t)\left(\frac1{2\sqrt{\varepsilon+s}}
                           +\frac1{6(\varepsilon+s)}\right)\biggr].
\end{aligned}
\]
In \eqref{sup:eq:comparison}, the initial term is at most $C_S\theta_0$. Integrate this estimate, add the terminal smoothing error $2\sqrt{\tfrac{2}{\pi}}\sqrt\varepsilon$, and take the supremum over all $h:\R\to\R$ with $\Lip(h)\leq1$.
\end{proof}
\begin{corollary}[Finite $q$ and the endpoint]\label{sup:cor:profile-rates}
Suppose Assumption~\ref{sup:ass:one-sided} holds and $\theta_a\leq1$. For $1\leq q<\infty$ and $J_q\leq1$,
\begin{equation*}
 W_1(\Law(X_T^0),\Law(X_T^1))\leq C_{q,T,C_S}
   \left(\theta_0+\theta_b+\theta_a^{\frac{1}{2}}+J_q^{\frac{q}{q+2}}\right).
\end{equation*}
For $J_\infty\leq e^{-1}$,
\begin{equation}\label{sup:eq:gaussian-profile-log}
 W_1(\Law(X_T^0),\Law(X_T^1))\leq C_{T,C_S}
   \left(\theta_0+\theta_b+\theta_a^{\frac{1}{2}}
                +J_\infty(1+|\log J_\infty|)\right),
\end{equation}
where the last term is zero at $J_\infty=0$.
\end{corollary}
\begin{proof}
For $1<q<\infty$, let $q'\coloneqq \tfrac{q}{q-1}$. Integration gives
\[
 \norm{(\varepsilon+\cdot)^{-1}}_{L^{q'}(0,T)}^{q'}
 =\int_0^T(\varepsilon+s)^{-q'}\dd s =
    \frac{\varepsilon^{1-q'}-(\varepsilon+T)^{1-q'}}{q'-1} \leq\frac{\varepsilon^{1-q'}}{q'-1}.
\]
Taking the $q'$th root and using $\tfrac{1}{q}+\tfrac{1}{q'}=1$ gives
\[
 \norm{(\varepsilon+\cdot)^{-1}}_{L^{q'}(0,T)}
 \leq(q'-1)^{-\frac{1}{q'}}\varepsilon^{-\frac{1}{q}}.
\]
For $q=1$, the $L^\infty$ norm equals $\varepsilon^{-1}$. H\"older's inequality, $u^{-\frac{1}{2}}\leq1+u^{-1}$, and $\varepsilon\leq1$ therefore reduce \eqref{sup:eq:master} to
\[
 W_1(\Law(X_T^0),\Law(X_T^1))\ \ \byeq{sup:eq:master}{\leq}\ \ 
 C_{q,T,C_S}\left(\theta_0+\theta_b+\sqrt\varepsilon
       +\theta_a\varepsilon^{-\frac{1}{2}}+J_q\varepsilon^{-\frac{1}{q}}\right).
\]
Take $\varepsilon\coloneqq \max\left\{\theta_a,J_q^{\frac{2q}{q+2}}\right\}$. Then $\sqrt\varepsilon\leq\theta_a^{\frac{1}{2}}+J_q^{\frac{q}{q+2}}$. The Gaussian term satisfies $\theta_a\varepsilon^{-\frac{1}{2}}\leq\theta_a^{\frac{1}{2}}$, and the jump term satisfies $J_q\varepsilon^{-\frac{1}{q}}\leq J_q^{\frac{q}{q+2}}$. For $q=\infty$, compute instead
\[
\int_0^T\frac{\dd s}{\varepsilon+s}
    =\log\frac{\varepsilon+T}{\varepsilon}.
\]

For the square-root kernel,
\[
\int_0^T\frac{\dd s}{\sqrt{\varepsilon+s}}
    =2(\sqrt{\varepsilon+T}-\sqrt\varepsilon)\leq2\sqrt T.
\]
With $\varepsilon\coloneqq \max\{\theta_a,J_\infty^2\}$, the first two error terms are bounded by $\theta_a^{\frac{1}{2}}+J_\infty$, and, when $J_\infty>0$,
\[
 J_\infty\log\frac{\varepsilon+T}{\varepsilon}
 \leq J_\infty\bigl[\log(1+T)+2|\log J_\infty|\bigr].
\]
If the chosen maximum vanishes, let $\varepsilon\downarrow0$.
\end{proof}
To express the endpoint bound in terms of $J_1$, assume $r_\nu(t)\leq\Lambda J_1$ almost everywhere. Then $J_\infty\leq\Lambda J_1$, which gives an error of order $J_1(1+|\log J_1|)$ as $J_1\downarrow0$ for fixed $\Lambda$.
\begin{remark}[Operator singularity and terminal-law lower bounds]
\label{sup:rem:operator-singularity}
For $H_sh(x)=\E\left[h(x+\sqrt sG)\right]$, $s>0$, $G\sim N(0,1)$, and $h(x)\coloneqq |x|$, $D^2|x|=2\delta_0$,
\[
(H_sh)''(x)=2s^{-\frac{1}{2}}\varphi\left(\frac{x}{\sqrt s}\right)
\]
and
\[
(H_sh)'''(x)=-2xs^{-\frac{3}{2}}\varphi\left(\frac{x}{\sqrt s}\right).
\]
Since $\sup_{r\in[0,\infty)}r\varphi(r)=\varphi(1)$, $\|(H_sh)'''\|_\infty=2\varphi(1)s^{-1}$. Thus the $s^{-1}$ singularity in \eqref{sup:eq:master} is sharp for Lipschitz tests. Sharpness of the logarithmic and finite-$q$ rates for terminal laws remains open.
\end{remark}
For pure jump perturbations with equal initial laws, meaning $\theta_0=\theta_b\coloneqq \theta_a\coloneqq 0$, the condition
\begin{equation*}
 \int_0^T\frac{r_\nu(t)}{T-t}\dd t\leq C_0J_1
\end{equation*}
gives a linear bound: since $(T-t)^{-\frac{1}{2}}\leq\tfrac{\sqrt T}{T-t}$, letting $\varepsilon\downarrow0$ in \eqref{sup:eq:master} gives
$$W_1(\Law(X_T^0),\Law(X_T^1))\leq CC_0(1+\sqrt T)J_1.$$
Support in $[0,T-\delta]$, $\delta>0$, gives $C_0=\delta^{-1}$.
\begin{proposition}[A stronger-discrepancy comparison]
\label{sup:prop:TV-comparison}
Suppose the comparison identity and the first two bounds in \eqref{sup:eq:one-sided-derivatives} hold. For a pure jump perturbation with equal initial laws, let $r_{\mathrm{TV}}(t) \coloneqq \E\left[\|y^2(\nu_0-\nu_1)(t,X^0_{t-},\dd y)\|_{\mathrm{TV}}\right]\leq K$ for a finite constant $K\geq0$; thus $r_{\mathrm{TV}}:[0,T]\to[0,K]$ is a measurable profile. Set $\theta_{\mathrm{TV}}\coloneqq \int_0^Tr_{\mathrm{TV}}(t)\dd t$. Then $W_1(\Law(X_T^0),\Law(X_T^1))\leq C\sqrt{K\theta_{\mathrm{TV}}}$.
\end{proposition}
\begin{proof}
For the total variation pairing, use $\|\Phi_{v,x}\|_\infty\leq\tfrac{\|v''\|_\infty}{2}$. Thus
\[
 W_1(\Law(X_T^0),\Law(X_T^1))\leq2\sqrt{\frac{2}{\pi}}\sqrt\varepsilon+
             \tfrac12C_S\int_0^T
                 \frac{r_{\mathrm{TV}}(t)}{\sqrt{\varepsilon+T-t}}\dd t.
\]
Define $r:[0,T]\to[0,K]$ by $r(s)\coloneqq r_{\mathrm{TV}}(T-s)$ and $m:[0,K]\to[0,T]$ by $m(u)\coloneqq |\{s\in(0,T):r(s)>u\}|$, where $|E|$ denotes Lebesgue measure. For any measurable $E\subset(0,T)$, decreasing rearrangement of $s^{-\frac{1}{2}}$ gives $\int_Es^{-\frac{1}{2}}\dd s\leq\int_0^{|E|}s^{-\frac{1}{2}}\dd s=2\sqrt{|E|}$. Writing $r(s)$ as an integral over its level sets and applying Tonelli's theorem, we obtain
\begin{align*}
 \int_0^T r(s)s^{-\frac{1}{2}}\dd s
 &=\int_0^K\int_{\{s\in(0,T):r(s)>u\}}s^{-\frac{1}{2}}\dd s\dd u\\
 &\leq 2\int_0^K\sqrt{m(u)}\dd u
 \leq 2\sqrt{K\int_0^K m(u)\dd u}
 =2\sqrt{K\theta_{\mathrm{TV}}}.
\end{align*}
This is finite, so letting $\varepsilon\downarrow0$ proves the bound.
\end{proof}
\subsection{Factorizable state-dependent kernels}\label{sup:sec:factorized}
Consider the processes of section~\ref{sup:sec:one-sided}, with the same drift and diffusion coefficients, under \eqref{sup:eq:semimartingale-integrability}. To transfer derivatives to the forward density, assume that, for almost every $t\in[0,T]$ and every $x\in\R$,
\begin{equation}\label{sup:eq:factor-decomposition}
 K_t(x,\dd y)\coloneqq y^2(\nu_0-\nu_1)(t,x,\dd y)
   =\sum_{j=1}^m r_j(t,x)\kappa_{j,t}(\dd y)
   \quad\text{in }\cM_f(\R).
\end{equation}
For each $t$, $K_t$ maps $\R$ to $\cM_f(\R)$. Here $m\in\N$. The factors $r_j:[0,T]\times\R\to\R$ and $\kappa_{j,t}\in\cM_f(\R)$ are jointly measurable, and each $r_j(t,\cdot)$ belongs to $W^{2,\infty}(\R)$. With weak spatial derivatives and $\|r\|_{W^{2,\infty}}\coloneqq \|r\|_\infty+\|r'\|_\infty+ \|r''\|_\infty$, assume
\begin{equation}\label{sup:eq:factor-cost}
 \mathcal D\coloneqq \int_0^T\sum_{j=1}^m
       \norm{r_j(t)}_{W^{2,\infty}}
       \norm{\kappa_{j,t}}_{\BL^*}\dd t<\infty.
\end{equation}
\begin{assumption}[Forward--backward regularity]
\label{sup:ass:factor-regularity}
For every $h\in C_b^\infty(\R)$ with $\Lip(h)\leq1$, the comparison identity \eqref{sup:eq:comparison} holds for $v(t,\cdot)\coloneqq P^1_{t,T}h$, with an integrable scalar integrand. For $t<T$, $v(t,\cdot)\in C_b^3(\R)$, and
\begin{equation}\label{sup:eq:factor-backward}
 \sup_{0\leq t\leq \frac{T}{2}}
   \bigl(\norm{v''(t)}_\infty+\norm{v'''(t)}_\infty\bigr)
 +\sup_{0\leq t<T}\norm{v'(t)}_\infty\leq B.
\end{equation}
For $\tfrac{T}{2}\leq t\leq T$, $\Law(X_t^0)$ has a density $p_t\in W^{2,1}(\R)$ satisfying
\begin{equation}\label{sup:eq:factor-forward}
 \sup_{\frac{T}{2}\leq t\leq T}
   \bigl(\norm{p_t'}_1+\norm{p_t''}_1\bigr)\leq F.
\end{equation}
The constants $B\geq1$ and $F\geq0$ are finite and independent of $h$.
\end{assumption}
Theorem~\ref{sup:thm:factor-intensity} verifies this assumption for a state-dependent intensity model.
\begin{lemma}[Transfer to a forward factor]\label{sup:lem:factor-transfer}
If $q\in W^{2,1}(\R)$, $w\in L^\infty(\R)$, and $r\in[0,1]$, then the map
\[
 \mapdef{A_r}{\R}{\R}{y}
       {-\displaystyle\int_\R q'(z-ry)w(z)\dd z}
\]
satisfies
\begin{subequations}\label{sup:eq:factor-transfer}
\begin{equation}
\norm{A_r}_\infty\leq\norm{q'}_1\norm w_\infty.
\end{equation}
Its Lipschitz constant satisfies
\begin{equation}
\Lip(A_r)\leq r\norm{q''}_1\norm w_\infty.
\end{equation}
\end{subequations}
If $w=v'$ with bounded $v',v''$, then $A_r(y)=\int_\R q(x)v''(x+ry)\dd x$.
\end{lemma}
\begin{proof}
For $f\in W^{1,1}(\R)$ and $a\in\R$, the absolutely continuous representative satisfies
\[
f(z-a)-f(z)=-a\int_0^1f'(z-\theta a)\dd\theta.
\]

Integrating absolute values gives
\[
\norm{f(\cdot-a)-f}_1\leq|a|\norm{f'}_1.
\]
H\"older's inequality and translation invariance give, for $y,\widetilde y\in\R$,
\[
 |A_r(y)|
 \leq\norm{q'(\cdot-ry)}_1\norm w_\infty =\norm{q'}_1\norm w_\infty.
\]
Then H\"older's inequality and the translation estimate above give
\[
 |A_r(y)-A_r(\widetilde y)|
 \leq\norm w_\infty
       \norm{q'(\cdot-ry)-q'(\cdot-r\widetilde y)}_1 \leq
       r|y-\widetilde y|\norm{q''}_1\norm w_\infty.
\]
For $w=v'$, weak integration by parts and $z=x+ry$ yield
\[
\begin{aligned}
 \int_\R q(x)v''(x+ry)\dd x
 &=-\int_\R q'(x)v'(x+ry)\dd x\\
 &=-\int_\R q'(z-ry)v'(z)\dd z =A_r(y).
\end{aligned}
\]
A compactly supported cutoff justifies the integration by parts against $v'$; all products are integrable.
\end{proof}
\begin{theorem}[Linear bound for a factorizable difference]
\label{sup:thm:factor-bound}
Under \eqref{sup:eq:semimartingale-integrability}, \eqref{sup:eq:factor-decomposition}--\eqref{sup:eq:factor-cost}, and assumption~\ref{sup:ass:factor-regularity},
\begingroup
\begin{equation*}
W_1\bigl(\Law(X_T^0),\Law(X_T^1)\bigr)
 \leq B\theta_0+C_{B,F}\mathcal D.
\end{equation*}

Here
\begin{equation*}
C_{B,F}\coloneqq B\max\left\{\frac12,\frac{2F}{3},
                       \frac12+\frac F3,\frac16\right\}.
\end{equation*}
\endgroup
\end{theorem}
\begin{proof}
Fix $h\in C_b^\infty(\R)$ with $\Lip(h)\leq1$ and put $v(t)\coloneqq P^1_{t,T}h$. Taylor's integral remainder gives
\[
v(t,x+y)-v(t,x)-yv'(t,x)
 =y^2\Phi_{v(t),x}(y).
\]

Here
\[
\Phi_{v(t),x}(y)\coloneqq \int_0^1(1-r)v''(t,x+ry)\dd r.
\]
Define $\Psi_{j,t}:\R\to\R$ by the formula below. The jump term in \eqref{sup:eq:comparison} is therefore
\begingroup
\begin{equation*}
\int_0^T\sum_{j=1}^{m}\langle\Psi_{j,t},\kappa_{j,t}\rangle\dd t.
\end{equation*}

Here
\begin{equation*}
\Psi_{j,t}(y)\coloneqq \E\left[
 r_j(t,X_t^0)\int_0^1(1-r)v''(t,X_t^0+ry)\dd r\right].
\end{equation*}
\endgroup
For $t\leq \tfrac{T}{2}$, Lemma~\ref{lem:Taylor} and $\int_0^1(1-r)\dd r=\tfrac{1}{2}$ and $\int_0^1r(1-r)\dd r=\tfrac{1}{6}$ give
\begin{align*}
 \norm{\Psi_{j,t}}_{\BL}
 \ \ &\byeq{eq:Phi-norm}{\leq}\ \ \norm{r_j(t)}_\infty
       \left(\tfrac12\norm{v''(t)}_\infty
                    +\tfrac16\norm{v'''(t)}_\infty\right)\\
 &\leq\frac{\norm{r_j(t)}_\infty}{2}
               \bigl(\norm{v''(t)}_\infty+\norm{v'''(t)}_\infty\bigr)\ \ \byeq{sup:eq:factor-backward}{\leq}\ \ \frac B2\norm{r_j(t)}_\infty.
\end{align*}
For $\tfrac{T}{2}\leq t<T$, abbreviate $r_j\coloneqq r_j(t,\cdot)$ and set $q\coloneqq p_tr_j\in W^{2,1}(\R)$. Since $\|p_t\|_1=1$,
\[
q'=p_t'r_j+p_tr_j',
\qquad 
\norm{q'}_1\leq F\norm{r_j}_\infty+\norm{r_j'}_\infty,
\]
and
\[
q''=p_t''r_j+2p_t'r_j'+p_tr_j'',
\qquad 
\norm{q''}_1\leq F\norm{r_j}_\infty
                  +2F\norm{r_j'}_\infty+\norm{r_j''}_\infty.
\]
Apply lemma~\ref{sup:lem:factor-transfer} with $w\coloneqq v'(t)$ and integrate in $r$:
\[
\begin{aligned}
 \norm{\Psi_{j,t}}_{\BL} 
 \ \ &\byeq{sup:eq:factor-transfer}{\leq}\ \ \norm{v'(t)}_\infty
        \int_0^1(1-r)\bigl(\norm{q'}_1+r\norm{q''}_1\bigr)\dd r\\
 \ \ &\byeq{sup:eq:factor-backward}{\leq}\ \ 
 B\left[\frac{2F}{3}\norm{r_j}_\infty
       +\left(\frac12+\frac F3\right)\norm{r_j'}_\infty
       +\tfrac16\norm{r_j''}_\infty\right].
\end{aligned}
\]
On both time intervals, $\|\Psi_{j,t}\|_{\BL}\leq C_{B,F}\|r_j(t)\|_{W^{2,\infty}}$. Hence
\[
 \left|\int_0^T\sum_{j=1}^{m}\langle\Psi_{j,t},\kappa_{j,t}\rangle\dd t\right|
 \leq C_{B,F}\int_0^T\sum_{j=1}^{m}
      \norm{r_j(t)}_{W^{2,\infty}}
      \norm{\kappa_{j,t}}_{\BL^*}\dd t=C_{B,F}\mathcal D.
\]
The initial term is bounded by $\Lip(v(0))W_1(\Law(X_0^0),\Law(X_0^1))\leq B\theta_0$. Lemma~\ref{lem:approximation} and the finite terminal first moments extend the estimate to all one-Lipschitz tests. Kantorovich duality completes the proof.
\end{proof}
\begin{corollary}[Countable factorizations]\label{sup:cor:factor-countable}
The same bound holds with $m=\infty$ if the series in \eqref{sup:eq:factor-decomposition} converges in $\BL^*$ for almost every $t$ and every $x$, and the series version of \eqref{sup:eq:factor-cost} is integrable.
\end{corollary}
\begin{proof}
For almost every fixed $t<T$, define the maps $K_{n,t}:\R\to\cM_f(\R)$ $K_{n,t}(x)\coloneqq \sum_{j=1}^nr_j(t,x)\kappa_{j,t}$ and the measurable majorant $S:[0,T]\to[0,\infty]$,
\[
S(t)\coloneqq \sum_{j=1}^{\infty}\norm{r_j(t)}_{W^{2,\infty}}
                       \norm{\kappa_{j,t}}_{\BL^*}.
\]
The majorant satisfies $S\in L^1(0,T)$. Since $\Phi_{v(t),x}\in\BL(\R)$, convergence in $\BL^*$ gives $\langle\Phi_{v(t),x},K_{n,t}(x)\rangle \xrightarrow[n\to\infty]{}\langle\Phi_{v(t),x},K_t(x)\rangle$. Moreover,
\[
 |\langle\Phi_{v(t),x},K_{n,t}(x)\rangle|
 \leq\left(\tfrac12\norm{v''(t)}_\infty
           +\tfrac16\norm{v'''(t)}_\infty\right)S(t),
\]
so dominated convergence applies under $\Law(X_t^0)$. Writing $I_n(t)\coloneqq \E[\langle\Phi_{v(t),X_t^0},K_{n,t}(X_t^0)\rangle]$ and $I(t)\coloneqq \E[\langle\Phi_{v(t),X_t^0},K_t(X_t^0)\rangle]$, we have $I_n(t)\xrightarrow[n\to\infty]{}I(t)$, and theorem~\ref{sup:thm:factor-bound} gives $|I_n(t)|\leq C_{B,F}S(t)$. Applying dominated convergence once more gives
\[
\int_0^TI(t)\dd t=\lim_{n\to\infty}\int_0^TI_n(t)\dd t.
\]

The limit satisfies
\[
\left|\int_0^TI(t)\dd t\right|\leq C_{B,F}\mathcal D.
\]
The initial term is unchanged.
\end{proof}
The infimum of $\mathcal D$ over all finite or countable jointly measurable factorizations satisfying \eqref{sup:eq:factor-decomposition} and \eqref{sup:eq:factor-cost} also bounds the jump contribution. For such a factorization, with $m\in\N\cup\{\infty\}$, pointwise duality gives
\[
\norm{K_t(x)}_{\BL^*}
 \leq\sum_{j=1}^{m}|r_j(t,x)|\norm{\kappa_{j,t}}_{\BL^*}.
\]

After integration, we obtain
\[
\int_0^T\E\left[\norm{K_t(X_t^0)}_{\BL^*}\right]\dd t\leq\mathcal D.
\]
The cost $\mathcal D$ controls both the weak discrepancy and spatial variation of the kernels.
\subsubsection{A state-dependent intensity model}
\label{sup:sec:factor-intensity}
For $i\in\{0,1\}$, let $c_i\in C_b^3(\R)$ be nonnegative and $F_i\in\cP(\R\setminus\{0\})$ satisfy
\begingroup
\begin{equation*}
\int_\R yF_i(\dd y)=0.
\end{equation*}

Assume also that the second moment is finite, and denote it by
\begin{equation*}
s_i\coloneqq \int_\R y^2F_i(\dd y)<\infty.
\end{equation*}
\endgroup
On $C_b^2(\R)$, the generator
\begin{equation}\label{sup:eq:factor-generator}
 \cL_i f(x)\coloneqq \tfrac12f''(x)
          +c_i(x)\int_\R[f(x+y)-f(x)]F_i(\dd y)
\end{equation}
is realized by Brownian motion and bounded-rate thinning: candidate jumps arrive at rate $\|c_i\|_\infty$ and are accepted at state $x$ with probability $\tfrac{c_i(x)}{\|c_i\|_\infty}$ (no jumps if $c_i=0$). There are finitely many jumps on bounded time intervals. Centering makes the full-compensation drift zero. Write $\mu_i(\dd y)\coloneqq y^2F_i(\dd y)\in\cM_f^+(\R)$ .
\begin{theorem}[Weak jump-law stability with state-dependent intensities]
\label{sup:thm:factor-intensity}
Let $X^i$ have generator \eqref{sup:eq:factor-generator} and initial law in $\cP_1(\R)$. For every $T>0$, a finite constant $C_T$, depending only on $T$, the $C_b^3$ bounds of $c_i$, and $\int_{\R}|y|F_i(\dd y)$, satisfies
\begin{equation*}
\begin{aligned}
 &W_1\bigl(\Law(X_T^0),\Law(X_T^1)\bigr)
 \leq C_T\bigl[\theta_0 +T\norm{c_0-c_1}_{W^{2,\infty}}s_0\\
 &\qquad +T\norm{c_1}_{W^{2,\infty}}d_{\FM}(\mu_0,\mu_1)\bigr].
\end{aligned}
\end{equation*}
\end{theorem}
\begin{remark*}
The bound applies to discrete jump laws and uses their weighted weak discrepancy.
\end{remark*}
\begin{proof}
We verify assumption~\ref{sup:ass:factor-regularity}. Omit the index $i$, let $\mathrm{BUC}(\R)$ be the Banach space of bounded uniformly continuous real functions on $\R$, with the supremum norm. Define bounded operators $R,B:\mathrm{BUC}(\R)\to\mathrm{BUC}(\R)$, using the heat semigroup $H_t$ on this space, by $H_tf\coloneqq f*g_t$, and
\[
Rf\coloneqq \int_\R[f(\cdot+y)-f]F(\dd y),
\]
and $Bf\coloneqq cRf$. All unqualified operator norms in this paragraph are from $\mathrm{BUC}(\R)$ to itself; in particular $\|B\|\leq2\|c\|_\infty$. The absolutely convergent heat perturbation series is
\[
P_t\coloneqq \sum_{n=0}^\infty P_t^{(n)}.
\]
Here $P_t^{(0)}\coloneqq H_t$ and
\[
P_t^{(n+1)}\coloneqq \int_0^tH_{t-s}BP_s^{(n)}\dd s.
\]

We have
\[
\norm{P_t^{(n)}}\leq\frac{(2t\norm c_\infty)^n}{n!}.
\]
It solves $P_t=H_t+\int_0^tH_{t-s}BP_s\dd s$ and gives the semigroup of the bounded-rate process. For $h\in C_b^\infty(\R)$, define $w_s:\R\to\R$ by $w_s\coloneqq DP_sh$. Set $m_1\coloneqq \int_\R|y|F(\dd y)$ and define
\[
Jw\coloneqq \int_\R\int_0^yw(\cdot+z)\dd zF(\dd y).
\]
This operator satisfies $\|Jw\|_\infty\leq m_1\norm w_\infty.$ Here $J:\mathrm{BUC}(\R)\to\mathrm{BUC}(\R)$ is a bounded linear operator. The inner integral is oriented when $y<0$. Since $Rf=J(Df)$, $D(Rf)=R(Df)$, and $D(Jw)=J(Dw)$, differentiating the backward equation gives
\begingroup
\begin{equation*}
\partial_sw_s=\tfrac12D^2w_s+\mathcal Cw_s.
\end{equation*}

Here
\begin{equation*}
\mathcal Cw\coloneqq cRw+c'Jw.
\end{equation*}
\endgroup
The operator $\mathcal C:\mathrm{BUC}(\R)\to\mathrm{BUC}(\R)$ is bounded. For $w\in C_b^1(\R)$,
\[
 D\mathcal Cw=c'Rw+cR(Dw)+c''Jw+c'J(Dw).
\]
Consequently,
\begin{subequations}\label{sup:eq:factor-gradient-bounds}
\begin{equation}
\norm{\mathcal Cw}_\infty\leq K\norm w_\infty.
\end{equation}
For the derivative, we have
\begin{equation}
\norm{D\mathcal Cw}_\infty
 \leq K'\bigl(\norm w_\infty+\norm{Dw}_\infty\bigr).
\end{equation}
\end{subequations}
where
\[
K\coloneqq 2\norm c_\infty+m_1\norm{c'}_\infty,
\]
and
\[
K'\coloneqq \max\{K,2\norm{c'}_\infty+m_1\norm{c''}_\infty\}.
\]
Apply heat contraction to the mild form $w_s=H_sh'+\int_0^sH_{s-r}\mathcal Cw_r\dd r$. Then the operator bound and Gronwall's inequality give
\[
 \norm{w_s}_\infty
 \leq
     \norm{h'}_\infty+\int_0^s\norm{\mathcal Cw_r}_\infty\dd r \byeq{sup:eq:factor-gradient-bounds}{\leq}
     \norm{h'}_\infty+K\int_0^s\norm{w_r}_\infty\dd r \leq e^{Ks}\norm{h'}_\infty.
\]
Using $\|Dg_s\|_1=\sqrt{\tfrac{2}{\pi}}s^{-\frac{1}{2}}$ and $\|D^2g_s\|_1=\sqrt{\tfrac{8}{\pi e}}s^{-1}$, distribute the derivatives as
\[
\begin{aligned}
 Dw_s&=(Dg_s)*h'+\int_0^s(Dg_{s-r})*\mathcal Cw_r\dd r,\\
 D^2w_s&=(D^2g_s)*h'
             +\int_0^s(Dg_{s-r})*D\mathcal Cw_r\dd r.
\end{aligned}
\]
The bounds are
\[
\begin{aligned}
\norm{Dw_s}_\infty
 \leq \sqrt{\frac{2}{\pi}}s^{-\frac{1}{2}}\norm{h'}_\infty
 +\sqrt{\frac{2}{\pi}}K\int_0^s(s-r)^{-\frac{1}{2}}\norm{w_r}_\infty\dd r.
\end{aligned}
\]

Moreover,
\[
\begin{aligned}
\norm{D^2w_s}_\infty
 \leq \sqrt{\frac{8}{\pi e}}s^{-1}\norm{h'}_\infty
+\sqrt{\frac{2}{\pi}}K'\int_0^s(s-r)^{-\frac{1}{2}}
                  (\norm{w_r}_\infty+\norm{Dw_r}_\infty)\dd r.
\end{aligned}
\]
For $0<s\leq T$, the first inequality gives
\[
\begin{aligned}
\norm{Dw_s}_\infty
 \leq \sqrt{\frac{2}{\pi}}\left(s^{-\frac{1}{2}}+2Ke^{KT}\sqrt s\right)\norm{h'}_\infty
 \leq A_Ts^{-\frac{1}{2}}\norm{h'}_\infty.
\end{aligned}
\]
Here
\[
A_T\coloneqq \sqrt{\frac{2}{\pi}}(1+2KTe^{KT}).
\]
Substituting into the second inequality gives
\[
 \norm{D^2w_s}_\infty
 \leq\left[\sqrt{\frac{8}{\pi e}}s^{-1}
       +\sqrt{\frac{2}{\pi}}K'(2e^{KT}\sqrt s+\pi A_T)\right]\norm{h'}_\infty,
\]
because
\[
\int_0^s(s-r)^{-\frac{1}{2}}\dd r=2\sqrt s.
\]

For the second integral,
\[
\int_0^s(s-r)^{-\frac{1}{2}}r^{-\frac{1}{2}}\dd r=\pi.
\]
Thus, with constants depending only on the stated parameters,
$\|DP_sh\|_\infty\leq C_T\Lip(h).$ For the second derivative, $\|D^2P_sh\|_\infty\leq C_Ts^{-\frac{1}{2}}\Lip(h).$ For the third derivative, $\|D^3P_sh\|_\infty\leq C_Ts^{-1}\Lip(h).$
The bounds justify termwise differentiation in distributions. Heat convolution makes the derivatives continuous. For $v(t)\coloneqq P_{T-t}h$, $T-t\geq \tfrac{T}{2}$ on the first half of the interval, which verifies \eqref{sup:eq:factor-backward}. For an initial law $\zeta\in\cP(\R)$, the forward series has a density $p_t\geq0$ for $t>0$, with $\|p_t\|_1=1$ and
\begin{subequations}\label{sup:eq:factor-forward-mild}
\begin{equation}
p_t=H_t\zeta+\int_0^tH_{t-s}B^*p_s\dd s.
\end{equation}

Here
\begin{equation}
B^*p\coloneqq (cp)*F-cp.
\end{equation}
\end{subequations}
Here $B^*:L^1(\R)\to L^1(\R)$ is the forward jump operator and $H_t\zeta(x)=\int_\R g_t(x-z)\zeta(\dd z)$. For $p\in W^{1,1}(\R)$,
\[
 D(B^*p)=(c'p+cp')*F-(c'p+cp'),
\]
so Young's inequality gives
\begin{subequations}\label{sup:eq:factor-adjoint-bounds}
\begin{equation}
\norm{B^*p}_1\leq2\norm c_\infty\norm p_1.
\end{equation}
The weak derivative satisfies
\begin{equation}
\norm{D(B^*p)}_1
 \leq2\norm{c'}_\infty\norm p_1+2\norm c_\infty\norm{p'}_1.
\end{equation}
\end{subequations}
Differentiate \eqref{sup:eq:factor-forward-mild} as above. Young's inequality, followed by \eqref{sup:eq:factor-adjoint-bounds}, yields
\[
\begin{aligned}
\norm{p_t'}_1
 &\leq\norm{Dg_t}_1
       +\int_0^t\norm{Dg_{t-s}}_1\norm{B^*p_s}_1\dd s\\
 &\byeq{sup:eq:factor-adjoint-bounds}{\leq}
       \sqrt{\frac{2}{\pi}}t^{-\frac{1}{2}}
       +2\sqrt{\frac{2}{\pi}}\norm c_\infty\int_0^t(t-s)^{-\frac{1}{2}}\dd s\\
 &=\sqrt{\frac{2}{\pi}}t^{-\frac{1}{2}}+4\sqrt{\frac{2}{\pi}}\norm c_\infty\sqrt t \leq A_T^*t^{-\frac{1}{2}}.
\end{aligned}
\]
Here $A_T^*\coloneqq \sqrt{\tfrac{2}{\pi}}(1+4T\|c\|_\infty).$ Using this estimate for the second derivative gives
\begin{align*}
 \norm{p_t''}_1
 &\leq\norm{D^2g_t}_1
       +\int_0^t\norm{Dg_{t-s}}_1\norm{D(B^*p_s)}_1\dd s\\
 &\byeq{sup:eq:factor-adjoint-bounds}{\leq}
       \sqrt{\frac{8}{\pi e}}t^{-1}
       +2\sqrt{\frac{2}{\pi}}\int_0^t(t-s)^{-\frac{1}{2}}
       \left(\norm{c'}_\infty+\norm c_\infty A_T^*s^{-\frac{1}{2}}\right)\dd s\\
 &=\sqrt{\frac{8}{\pi e}}t^{-1}
       +4\sqrt{\frac{2}{\pi}}\norm{c'}_\infty\sqrt t
       +2\pi\sqrt{\frac{2}{\pi}}\norm c_\infty A_T^* \leq C_Tt^{-1}.
\end{align*}
The integrable bounds justify the weak derivatives and imply $p_t\in W^{2,1}(\R)$, even for singular $F$. On $\left[\tfrac{T}{2},T\right]$ these bounds give \eqref{sup:eq:factor-forward}. Bounded intensities and $s_i<\infty$ give \eqref{sup:eq:semimartingale-integrability}. In particular,
\[
 \E\left[\int_0^T\int_{\R} y^2\nu_i(X_t^i,\dd y)\dd t\right]
 \leq T\norm{c_i}_\infty s_i<\infty.
\]
The backward equation and localized It\^o formula give the comparison identity for smooth bounded tests. The path-supremum bound gives uniform integrability of the stopped variables. Finally,
\[
c_0\mu_0-c_1\mu_1=(c_0-c_1)\mu_0+c_1(\mu_0-\mu_1).
\]

The weighted measure also satisfies
\[
\norm{\mu_0}_{\BL^*}=\mu_0(\R)=s_0.
\]
and hence
\[
 \mathcal D\leq T\norm{c_0-c_1}_{W^{2,\infty}}s_0
       +T\norm{c_1}_{W^{2,\infty}}d_{\FM}(\mu_0,\mu_1).
\]
Apply theorem~\ref{sup:thm:factor-bound}.
\end{proof}
Without centering, the full-compensation drift is $b_i(x)=c_i(x)\int_{\R} yF_i(\dd y)$ and contributes a drift discrepancy.
\subsubsection{A sharp non-additive moving-atom example}
\label{sup:sec:factor-moving-atom}
Let $c(x)\coloneqq 1+\tfrac12\sin x$ and, for $0<a\leq\tfrac12$, define $F_0^a\coloneqq \tfrac12\delta_{-a}+\tfrac12\delta_a$ and $F_1^a\coloneqq \tfrac45\delta_{-\frac{a}{2}}+\tfrac15\delta_{2a}$. Both laws are centered, with $\int_{\R} yF_i^a(\dd y)=0$, and their second moments agree, with $\int_{\R} y^2F_i^a(\dd y)=a^2$. Their third moments are $\int_{\R} y^3F_0^a(\dd y)=0$ and $\int_{\R} y^3F_1^a(\dd y)=\tfrac32a^3$. The weighted measures $\mu_0^a\coloneqq \tfrac{a^2}{2}(\delta_{-a}+\delta_a)$ and $\mu_1^a\coloneqq \tfrac{a^2}{5}\delta_{-\frac{a}{2}}+\tfrac{4a^2}{5}\delta_{2a}$ have disjoint supports and equal mass $a^2$. Therefore
\begin{subequations}\label{sup:eq:factor-atom-discrepancies}
\begin{equation}
\norm{\mu_0^a-\mu_1^a}_{\mathrm{TV}}=2a^2.
\end{equation}
For the weak discrepancy,
\begin{equation}
\tfrac34a^3\leq d_{\FM}(\mu_0^a,\mu_1^a)\leq3a^3.
\end{equation}
\end{subequations}
Indeed, $f(y)\coloneqq \tfrac12\max\{-1,\min\{y,1\}\}$ has $\|f\|_{\BL}=1$ and $|\langle f,\mu_0^a-\mu_1^a\rangle|=\tfrac{3a^3}{4}$. Conversely, for $\|f\|_{\BL}\leq1$, equal masses give
\[
 |\langle f,\mu_0^a-\mu_1^a\rangle|
 =|\langle f-f(0),\mu_0^a-\mu_1^a\rangle|
 \leq\int_{\R}|y|^3(F_0^a+F_1^a)(\dd y)=\tfrac{27}{10}a^3\leq3a^3.
\]
\begin{proposition}[Linear order with a nonconstant intensity]
\label{sup:prop:factor-atom}
Let $X^{i,a}$ start at zero and have generator $\cL_i^af\coloneqq \tfrac12f''+c\int_{\R}[f(\cdot+y)-f]F_i^a(\dd y)$. For every $T>0$,
\begin{equation*}
 \E\left[\sin X_T^{0,a}\right]-\E\left[\sin X_T^{1,a}\right]
    =\frac T4e^{-\frac{T}{2}}a^3+O_T(a^4).
\end{equation*}
Consequently, there are $a_T>0$ and $0<c_T\leq C_T<\infty$ such that
\begin{equation*}
 c_Ta^3\leq W_1\bigl(\Law(X_T^{0,a}),\Law(X_T^{1,a})\bigr)
       \leq C_Ta^3,\qquad 0<a\leq a_T.
\end{equation*}
\end{proposition}
\begin{remark*}
The linear power of the weak jump-law discrepancy is therefore sharp for this state-dependent family.
\end{remark*}
\begin{proof}
The upper bound follows uniformly in $a$ from theorem~\ref{sup:thm:factor-intensity} and \eqref{sup:eq:factor-atom-discrepancies}. Define $B_i^a:C_b^{k+2}(\R)\to C_b^k(\R)$, for $k=0,2$, by $B_i^af\coloneqq c\int_{\R}[f(\cdot+y)-f]F_i^a(\dd y)$. Centering gives
\[
 B_i^af(x)=c(x)\int_{\R} y^2\int_0^1(1-r)f''(x+ry)\dd rF_i^a(\dd y).
\]
Leibniz's rule and $\int_{\R} y^2F_i^a(\dd y)=a^2$ imply
\begin{equation}\label{sup:eq:factor-atom-operator}
 \norm{B_i^af}_{C_b^k}\leq Ca^2\norm f_{C_b^{k+2}},
 \qquad k=0,2.
\end{equation}
On $\mathrm{BUC}^4(\R)$, the space of functions with bounded uniformly continuous derivatives through order four, equipped with the $C_b^4$ norm, $\|B_i^a\|_{\mathrm{BUC}^4\to\mathrm{BUC}^4}\leq C$ uniformly in $i,a$. Heat contraction and the perturbation series give $\sup_{\substack{s\in[0,T],\ i\in\{0,1\}\\0<a\leq\frac{1}{2}}}\|P_s^{i,a}\sin\|_{C_b^4}\leq C_T$. Here $P_s^{i,a}$ is the transition semigroup of $X^{i,a}$ on $\mathrm{BUC}^4(\R)$. Applying Duhamel's formula twice gives
\[
P_T^{i,a}h
 =H_Th+\int_0^TH_tB_i^aH_{T-t}h\dd t+\mathcal E_i^a(h).
\]

Here
\[
\mathcal E_i^a(h)
 \coloneqq \int_0^T\int_0^{T-t}
       H_tB_i^aH_sB_i^aP_{T-t-s}^{i,a}h\dd s\dd t.
\]
Here $\mathcal E_i^a:\mathrm{BUC}^4(\R)\to\mathrm{BUC}(\R)$ is the remainder operator. For $h=\sin$, \eqref{sup:eq:factor-atom-operator} gives
\[
 \norm{H_tB_i^aH_sB_i^aP_{T-t-s}^{i,a}\sin}_\infty
 \byeq{sup:eq:factor-atom-operator}{\leq}
 Ca^2\norm{B_i^aP_{T-t-s}^{i,a}\sin}_{C_b^2}
 \leq C_Ta^4,
\]
so $\|\mathcal E_i^a(\sin)\|_\infty\leq C_Ta^4$. The third-order Taylor formula defines a remainder operator $\mathcal R_i^a:C_b^4(\R)\to C_b^0(\R)$ by
\[
B_i^af=\frac{a^2}{2}cf''
        +\tfrac16\left(\int_{\R} y^3F_i^a(\dd y)\right)cf'''
        +\mathcal R_i^a(f).
\]

The remainder satisfies
\[
\norm{\mathcal R_i^a(f)}_\infty
 \leq\frac{\norm c_\infty}{24}\norm{f^{(4)}}_\infty
                       \int_{\R} y^4F_i^a(\dd y).
\]
Since the fourth moments are $a^4$ and $\tfrac{13a^4}{4}$, respectively,
\[
(B_0^a-B_1^a)f=-\tfrac14a^3cf'''+R_a(f).
\]

The remainder is bounded by
\[
\norm{R_a(f)}_\infty\leq Ca^4\norm{f^{(4)}}_\infty.
\]
Here $R_a\coloneqq \mathcal R_0^a-\mathcal R_1^a:C_b^4(\R)\to C_b^0(\R)$. Now $H_{T-t}\sin=e^{-\frac{T-t}{2}}\sin$ and $D^3\sin=-\cos$. Also,
\[
 H_t(c\cos)(0)=H_t\cos(0)+\tfrac12H_t(\sin\cos)(0)=e^{-\frac{t}{2}},
\]
because $g_t$ is even and $\sin\cos$ is odd. Subtracting the Duhamel expansions and evaluating at zero gives
\[
\begin{aligned}
 \E\left[\sin X_T^{0,a}\right]-\E\left[\sin X_T^{1,a}\right] 
 &=\frac{a^3}{4}\int_0^Te^{-\frac{T-t}{2}}e^{-\frac{t}{2}}\dd t+O_T(a^4)\\
 &=\frac T4e^{-\frac{T}{2}}a^3+O_T(a^4).
\end{aligned}
\]
Since $\Lip(\sin)=1$ and $\tfrac{Te^{-\frac{T}{2}}}{4}>0$, duality gives the lower bound for sufficiently small $a$.
\end{proof}
\section{Infinite variance and stable smoothing}\label{app:stable}\label{sup:sec:stable}
\subsection{Canonical truncation and a hybrid discrepancy}
\label{sup:sec:finite-canonical}
For $i\in\{0,1\}$, let $b_i:[0,T]\to\R$, $a_i:[0,T]\to[0,\infty)$, and the L\'evy kernels $\nu_i$ be deterministic and measurable, with
\begin{equation}\label{sup:eq:finite-integrability}
\begin{aligned}
 \int_0^T\biggl[|b_i(t)|+a_i(t)
      +\int_{\{y\in\R:|y|\leq1\}}y^2\nu_i(t,\dd y)
\biggr.\biggl.\ + \int_{\{y\in\R:|y|>1\}}|y|\nu_i(t,\dd y)\biggr]\dd t<\infty.
\end{aligned}
\end{equation}
For $f\in C^2(\R)$ with bounded $f',f''$, the truncation $\chi:\R\to\R$ given by $\chi(y)\coloneqq y\one_{\{|y|\leq1\}}$ defines the canonical generator
\begin{equation*}
 \cL_i(t)f(x)\coloneqq b_i(t)f'(x)+\frac{a_i(t)}2f''(x)
   +\int_\R[f(x+y)-f(x)-\chi(y)f'(x)]\nu_i(t,\dd y).
\end{equation*}
The canonical and full-compensation drifts are related by
\begin{equation*}
 b_i^{\mathrm{can}}(t)=b_i^{\mathrm{full}}(t)
                -\int_{\{y\in\R:|y|>1\}}y\nu_i(t,\dd y).
\end{equation*}
Let $X^i$ be the additive process with these canonical characteristics, starting from $x_i\in\R$. Define the integrated L\'evy measure $N_i(A)\coloneqq \int_0^T\nu_i(t,A)\dd t$ for Borel $A\subset\R\setminus\{0\}$ and put
\[
\ell_i\coloneqq x_i+\int_0^Tb_i(t)\dd t\in\R,
\qquad 
v_i\coloneqq \int_0^Ta_i(t)\dd t\geq0.
\]

Consequently,
\[
\E\left[X_T^i\right]=\ell_i+\int_{\{y\in\R:|y|>1\}}yN_i(\dd y).
\]
The small- and large-jump weighted measures are
\begin{subequations}\label{sup:eq:finite-split-measures}
\begin{equation}
\mu_i^S(\dd y)\coloneqq y^2\one_{\{|y|\leq1\}}N_i(\dd y),
\end{equation}
and
\begin{equation}
\mu_i^L(\dd y)\coloneqq |y|\one_{\{|y|>1\}}N_i(\dd y).
\end{equation}
\end{subequations}
Both belong to $\cM_f^+(\R)$. Set $D_S\coloneqq d_{\FM}(\mu_0^S,\mu_1^S)$ and $D_L\coloneqq d_{\FM}(\mu_0^L,\mu_1^L)$.
\begin{theorem}[Linear terminal bound with infinite variance]
\label{sup:thm:finite-split}
For $i\in\{0,1\}$, the canonical terminal exponent $\Psi_i:\R\to\mathbb C$ is
\[
 \Psi_i(\xi)\coloneqq i\ell_i\xi-\tfrac12v_i\xi^2
           +\int_\R(e^{i\xi y}-1-i\xi\chi(y))N_i(\dd y),
 \qquad \xi\in\R.
\]
Interpolate with exponent $(1-u)\Psi_1+u\Psi_0$, $u\in[0,1]$, and let $\rho_u$ be its density. Set $M_k\coloneqq \int_0^1\|D^k\rho_u\|_1\dd u$, $k=1,2$. Under \eqref{sup:eq:finite-integrability} and Assumption~\ref{ass:interpolation} for these interpolation laws,
\begin{equation}\label{sup:eq:finite-split-bound}
 W_1\bigl(\Law(X_T^0),\Law(X_T^1)\bigr)
 \leq|\ell_0-\ell_1|+\frac{M_1}{2}|v_0-v_1|
         +\left(\frac{M_1}{2}+\frac{M_2}{6}\right)D_S+3D_L.
\end{equation}
\end{theorem}
\begin{proof}
The inequalities $|e^{iz}-1|\leq|z|$ and $|e^{iz}-1-iz|\leq \tfrac{z^2}{2}$ give
\[
\begin{aligned}
 |\Psi_i(\xi)|
 \leq\left(|\ell_i|+\int_{\{y\in\R:|y|>1\}}|y|N_i(\dd y)\right)|\xi|
 +\tfrac12\left(v_i+\int_{\{y\in\R:|y|\leq1\}}y^2N_i(\dd y)\right)|\xi|^2.
\end{aligned}
\]
Lemma~\ref{lem:interpolation} therefore applies to $\Psi_u\coloneqq (1-u)\Psi_1+u\Psi_0$, $0\leq u\leq1$. Let $\rho_u\in W^{2,1}(\R)$ be the interpolation density, $Q_uh(x)\coloneqq \int_{\R} h(x+z)\rho_u(z)\dd z$, and $M_k\coloneqq \int_0^1\|D^k\rho_u\|_1\dd u$, $k=1,2$. For $\Lip(h)\leq1$ and $q\coloneqq Q_uh$, derivative transfer gives $\|q'\|_\infty\leq1$, $\|q''\|_\infty\leq\|\rho_u'\|_1$, and $\|q'''\|_\infty\leq\|\rho_u''\|_1$. The drift, Gaussian, and small-jump terms are therefore bounded as in Theorem~\ref{thm:terminal} by
\[
 |\ell_0-\ell_1|+\tfrac12\norm{\rho_u'}_1|v_0-v_1|
 +\left(\tfrac12\norm{\rho_u'}_1
              +\tfrac16\norm{\rho_u''}_1\right)D_S.
\]
For the large-jump term and fixed $x\in\R$, define $\Xi:\{y\in\R:|y|>1\}\to\R$ by $\Xi(y)\coloneqq \tfrac{q(x+y)-q(x)}{|y|}$ on $|y|>1$. Then $|\Xi|\leq1$. For $y,z$ in the same outer interval,
\[
 |\Xi(y)-\Xi(z)|
 \leq\frac{|y-z|+\bigl||z|-|y|\bigr|}{\min\{|y|,|z|\}}
 \leq2|y-z|.
\]
For opposite signs, $|\Xi(y)-\Xi(z)|\leq2\leq|y-z|$. Extend $\Xi$ linearly from its boundary values at $-1$ and $1$. The slope on $[-1,1]$ has absolute value at most one. The resulting continuous extension satisfies $\|\Xi\|_\infty\leq1$ and $\Lip(\Xi)\leq2$, so
\[
 \left|\int_{\{y\in\R:|y|>1\}}[q(x+y)-q(x)](N_0-N_1)(\dd y)\right|
 =|\langle\Xi,\mu_0^L-\mu_1^L\rangle|\leq3D_L.
\]
Integrate in $u$. Finite terminal first moments allow us to approximate all Lipschitz tests.
\end{proof}
\subsection{A continuous weight without a hard cutoff}
\label{sup:sec:finite-continuous}
An atom crossing $|y|=1$ changes the two measures in \eqref{sup:eq:finite-split-measures} discontinuously. Under \eqref{sup:eq:finite-integrability}, use instead the mean and continuous weight $w:\R\to[0,\infty)$:
\begingroup
\begin{equation*}
m_i\coloneqq \ell_i+\int_{\{y\in\R:|y|>1\}}yN_i(\dd y),
\end{equation*}
\endgroup
and $w(y)\coloneqq \tfrac{y^2}{1+|y|}$, $\mu_i^w\coloneqq wN_i$, and $\Gamma_i^w\coloneqq v_i\delta_0+\mu_i^w$. The measures are finite because $w(y)\leq y^2$ for $|y|\leq1$ and $w(y)\leq|y|$ for $|y|>1$. For $f\in C^2(\R)$ with bounded $f',f''$,
\[
 |f(x+y)-f(x)-yf'(x)|
 \leq\begin{cases}
       \tfrac12\norm{f''}_\infty y^2,& |y|\leq1,\\
       2\norm{f'}_\infty|y|,& |y|>1.
      \end{cases}
\]
Hence full compensation defines the terminal operator
\[
 \cA_if\coloneqq m_if'+\tfrac12v_if''
        +\int_\R[f(\cdot+y)-f-yf']N_i(\dd y)
\]
without a second moment at infinity. It maps such tests to bounded continuous functions of $x\in\R$.
\begin{theorem}[Cutoff-free finite-first-moment comparison]
\label{sup:thm:finite-continuous}
Assume \eqref{sup:eq:finite-integrability} and Assumption~\ref{ass:interpolation} for the terminal laws $\lambda_i\coloneqq \Law(X_T^i)\in\cP_1(\R)$. With $C_w\coloneqq 10+\tfrac32M_1+\tfrac13M_2$,
\begin{equation}\label{sup:eq:finite-continuous-bound}
W_1(\lambda_0,\lambda_1)
 \leq|m_0-m_1|+\tfrac12M_1|v_0-v_1|
        +C_wd_{\FM}(\mu_0^w,\mu_1^w).
\end{equation}

Moreover,
\begin{equation}\label{sup:eq:finite-continuous-joint}
W_1(\lambda_0,\lambda_1)
 \leq|m_0-m_1|+C_wd_{\FM}(\Gamma_0^w,\Gamma_1^w).
\end{equation}
\end{theorem}
\begin{remark*}
Both bounds use the mean, so they are independent of the drift truncation convention.
\end{remark*}
\begin{proof}
For $f\in C^3(\R)$ with bounded first three derivatives and $x\in\R$, define
\[
 \mapdef{\Theta_{f,x}}{\R}{\R}{y}
 {\begin{cases}
   \dfrac{(1+|y|)[f(x+y)-f(x)-yf'(x)]}{y^2},&y\ne0,\\[4pt]
   \tfrac12f''(x),&y=0.
  \end{cases}}
\]
Put $L\coloneqq \|f'\|_\infty$, $H_2\coloneqq \|f''\|_\infty$, and $H_3\coloneqq \|f'''\|_\infty$. For $|y|\leq1$,
\[
\Theta_{f,x}(y)=(1+|y|)\Phi_{f,x}(y).
\]

Here
\[
\Phi_{f,x}(y)\coloneqq \int_0^1(1-r)f''(x+ry)\dd r.
\]
Thus,
\[
\norm{\Theta_{f,x}}_{L^\infty([-1,1])}\leq H_2.
\]

Its Lipschitz constant satisfies
\[
\Lip(\Theta_{f,x}|_{[-1,1]})
 \leq\norm{\Phi_{f,x}}_\infty+2\Lip(\Phi_{f,x})\leq\tfrac12H_2+\tfrac13H_3.
\]
For $|y|\geq1$, define $A:\{y\in\R:|y|\geq1\}\to\R$ by $A(y)\coloneqq \tfrac{f(x+y)-f(x)}{y}-f'(x)$. Then $|A(y)|\leq2L.$ Its derivative is
\[
A'(y)=\frac{yf'(x+y)-[f(x+y)-f(x)]}{y^2}.
\]
Consequently, $|A'(y)|\leq\tfrac{2L}{|y|}.$ Since $\Theta_{f,x}=\left(\tfrac{1}{y}+\operatorname{sgn}y\right)A$, we have $|\Theta_{f,x}(y)|\leq4L.$ Its derivative satisfies
\[
|\Theta_{f,x}'(y)|\leq\frac{2L}{y^2}
                +\left(1+\frac1{|y|}\right)\frac{2L}{|y|}
                \leq6L.
\]
The formulas agree at $\pm1$ and are continuous at zero. Splitting any interval at $-1,0,1$ therefore gives
\begin{equation}\label{sup:eq:finite-continuous-test}
 \norm{\Theta_{f,x}}_{\BL}
 \leq\max\{H_2,4L\}+\max\left\{\frac{H_2}{2}+\frac{H_3}{3},6L\right\}
 \leq10L+\tfrac32H_2+\tfrac13H_3.
\end{equation}
The full-compensation symbols satisfy $|\Psi_i(\xi)|\leq C_i(|\xi|+|\xi|^2)$: use $|e^{iz}-1-iz|\leq \tfrac{z^2}{2}$ for small jumps and $|e^{iz}-1-iz|\leq2|z|$ for large jumps. Thus the proof of Lemma~\ref{lem:interpolation} applies. For $f=Q_uh$, Lemma~\ref{lem:convolution} gives $L\leq1$, $H_2\leq\|\rho_u'\|_1$, and $H_3\leq\|\rho_u''\|_1$. The operator difference pairs $\Theta_{Q_uh,x}$ with $\mu_0^w-\mu_1^w$. Equation~\eqref{sup:eq:finite-continuous-test}, weak duality, and integration in $u$ give \eqref{sup:eq:finite-continuous-bound}. Because $\Theta_{f,x}(0)=\tfrac{f''(x)}{2}$, the identity
\[
 \tfrac12(v_0-v_1)f''(x)
  +\langle\Theta_{f,x},\mu_0^w-\mu_1^w\rangle
 =\langle\Theta_{f,x},\Gamma_0^w-\Gamma_1^w\rangle
\]
gives \eqref{sup:eq:finite-continuous-joint}. Finite terminal first moments justify approximating Lipschitz tests in the last step.
\end{proof}
For equal means, fixed $v>0$, and $N_0\coloneqq \delta_{1-\varepsilon}$, $N_1\coloneqq \delta_{1+\varepsilon}$, $0<\varepsilon<\tfrac{1}{2}$, the split discrepancies are $D_S=(1-\varepsilon)^2$ and $D_L=1+\varepsilon$. In contrast, $w$ is one-Lipschitz and Lemma~\ref{lem:atoms} gives
\[
\begin{aligned}
 d_{\FM}\bigl(w(1-\varepsilon)\delta_{1-\varepsilon},
              w(1+\varepsilon)\delta_{1+\varepsilon}\bigr)
 \leq|w(1-\varepsilon)-w(1+\varepsilon)|
          +w(1-\varepsilon)\frac{4\varepsilon}{2+2\varepsilon}
 =O(\varepsilon).
\end{aligned}
\]
Thus \eqref{sup:eq:finite-continuous-bound} gives $W_1(\lambda_0,\lambda_1)=O(\varepsilon)$ across the cutoff. \subsection{Stable convolution factors}\label{sup:sec:stable-factors} Fix $1<\beta<2$. For $s>0$, let $p_s^{(\beta)}:\R\to[0,\infty)$ be the symmetric stable probability density with Fourier transform $e^{-s|\xi|^\beta}$. Put $B_{k,\beta}\coloneqq \|D^kp_1^{(\beta)}\|_1<\infty$, $k=1,2$. These norms are finite by Section~\ref{sup:sec:stable-derivatives}. The stable density scales as
\[
p_s^{(\beta)}(x)=s^{-\frac{1}{\beta}}p_1^{(\beta)}\left(s^{-\frac{1}{\beta}}x\right).
\]

Differentiating gives
\[
D^kp_s^{(\beta)}(x)=s^{-\frac{k+1}{\beta}}
                    (D^kp_1^{(\beta)})\left(s^{-\frac{1}{\beta}}x\right).
\]
Consequently,
\begin{equation}\label{sup:eq:stable-scaling}
 \norm{(p_s^{(\beta)})^{(k)}}_1=B_{k,\beta}s^{-\frac{k}{\beta}}.
\end{equation}
\begin{corollary}[Terminal stable components with positive parameters]
\label{sup:cor:stable-terminal}
Suppose the $i$th terminal exponent has an independent symmetric $\beta$-stable factor with parameter $s_i>0$: subtracting $-s_i|\xi|^\beta$ leaves a L\'evy--Khintchine exponent satisfying the finite-first-moment condition. Then theorems~\ref{sup:thm:finite-split} and~\ref{sup:thm:finite-continuous} hold with
\begin{equation*}
 M_k\leq B_{k,\beta}\int_0^1
       \bigl((1-u)s_1+us_0\bigr)^{-\frac{k}{\beta}}\dd u
 \leq B_{k,\beta}\min(s_0,s_1)^{-\frac{k}{\beta}},\qquad k=1,2.
\end{equation*}
\end{corollary}
\begin{remark*}
If $c_i:[0,T]\to[0,\infty)$ is a measurable stable-parameter rate and $s_i=\int_0^Tc_i(t)\dd t$, positivity of the total parameter is enough.
\end{remark*}
\begin{proof}
Set $s_u\coloneqq (1-u)s_1+us_0$. The interpolated residual exponent defines $\eta_u\in\cP_1(\R)$. The convolution representation $\rho_u=p_{s_u}^{(\beta)}*\eta_u$ implies $D^k\rho_u=(D^kp_{s_u}^{(\beta)})*\eta_u$. Young's inequality and the stable scaling identity then give
\[
\begin{aligned}
 M_k=\int_0^1\norm{(D^kp_{s_u}^{(\beta)})*\eta_u}_1\dd u
 &\leq\int_0^1\norm{D^kp_{s_u}^{(\beta)}}_1\dd u\\
 \ \ &\byeq{sup:eq:stable-scaling}{=}\ \ B_{k,\beta}\int_0^1s_u^{-\frac{k}{\beta}}\dd u\leq B_{k,\beta}\min(s_0,s_1)^{-\frac{k}{\beta}},
\end{aligned}
\]
because $s_u\geq\min(s_0,s_1)>0$. For $\alpha\coloneqq \tfrac{k}{\beta}\ne1$,
\[
 \int_0^1s_u^{-\alpha}\dd u=
 \begin{cases}
 \displaystyle\frac{s_0^{1-\alpha}-s_1^{1-\alpha}}
                    {(1-\alpha)(s_0-s_1)},&s_0\ne s_1,\\[6pt]
 s_0^{-\alpha},&s_0=s_1.
 \end{cases}
\]
\end{proof}
The condition $\beta>1$ ensures a finite first moment for the stable factor \citep{Sato1999}.
\subsection{One-sided stable upper rates}\label{sup:sec:stable-onesided}
Return to the additive canonical setting of Section~\ref{sup:sec:finite-canonical}. Here only the reference process supplies smoothing. For almost every $t\in[0,T]$, define
\[
\mu_i^S(t,\dd y)\coloneqq y^2\one_{\{|y|\leq1\}}\nu_i(t,\dd y),
\qquad 
\mu_i^L(t,\dd y)\coloneqq |y|\one_{\{|y|>1\}}\nu_i(t,\dd y).
\]
For $1\leq q\leq\infty$,
\begingroup
\begin{equation*}
r_S(t)\coloneqq d_{\FM}(\mu_0^S(t),\mu_1^S(t)),
\qquad  
J_{S,q}\coloneqq \norm{r_S}_{L^q(0,T)},
\end{equation*}
and
\begin{equation*}
\Theta_L\coloneqq \int_0^Td_{\FM}(\mu_0^L(t),\mu_1^L(t))\dd t.
\end{equation*}
\endgroup
These quantities retain the time profile, whereas $D_S,D_L$ in \eqref{sup:eq:finite-split-bound} compare the integrated measures. For every one-Lipschitz $h:\R\to\R$, assume spatial mollification at scale $r\in(0,1]$ yields $h_r$ with $\Lip(h_r)\leq1$, $\|h-h_r\|_\infty\leq Cr$, and
\begin{subequations}\label{sup:eq:stable-onesided-smoothing}
\begin{equation}
\Lip(P^1_{t,T}h_r)\leq1.
\end{equation}

Moreover,
\begin{equation}
\norm{D^2P^1_{t,T}h_r}_\infty
   \leq C\min\left\{r^{-1},(T-t)^{-\frac{1}{\beta}}\right\}.
\end{equation}

Moreover,
\begin{equation}
\norm{D^3P^1_{t,T}h_r}_\infty
   \leq C\min\left\{r^{-2},(T-t)^{-\frac{2}{\beta}}\right\}.
\end{equation}
\end{subequations}
A sufficient condition is an independent reference stable factor with parameter $S_{t,T}\geq c_*(T-t)$, $c_*>0$, with residual law $\eta_{t,T}\in\cP_1(\R)$. Then
\[
 P^1_{t,T}h=h*p_{S_{t,T}}^{(\beta)}*\check\eta_{t,T}.
\]
Choose $\zeta\in C_c^\infty(\R)$ such that $\zeta\geq0$ and $\int_{\R}\zeta(z)\dd z=1$. Then set $\zeta_r(x)\coloneqq r^{-1}\zeta\left(\tfrac{x}{r}\right)$ and $h_r\coloneqq h*\zeta_r$. Since the weak derivative $h'\in L^\infty(\R)$ has norm at most one, we can place the derivatives either on the mollifier or on the stable density:
\[
 D^jP^1_{t,T}h_r
 =h'*(D^{j-1}\zeta_r)*p_{S_{t,T}}^{(\beta)}*\check\eta_{t,T} =h'*\zeta_r*(D^{j-1}p_{S_{t,T}}^{(\beta)})*\check\eta_{t,T},
 \qquad j=2,3.
\]
Hence
\[
\norm{D^jP^1_{t,T}h_r}_\infty
 \leq\min\left\{r^{1-j}\norm{\zeta^{(j-1)}}_1,
              B_{j-1,\beta}S_{t,T}^{-\frac{j-1}{\beta}}\right\}.
\]

The mollification error satisfies
\[
\norm{h-h_r}_\infty\leq r\int_{\R}|z|\zeta(z)\dd z.
\]
This proves \eqref{sup:eq:stable-onesided-smoothing}. For $f\in C^2(\R)$ with bounded first two derivatives,
\[
\begin{aligned}
 \norm{\cL_i(t)f}_\infty
 &\leq\left(|b_i(t)|+\int_{\{y\in\R:|y|>1\}}|y|\nu_i(t,\dd y)\right)
                      \norm{f'}_\infty\\
 &\quad +\tfrac12\left(a_i(t)+\int_{\{y\in\R:|y|\leq1\}}y^2\nu_i(t,\dd y)\right)
                      \norm{f''}_\infty.
\end{aligned}
\]
For fixed $r$, \eqref{sup:eq:finite-integrability} makes this bound integrable in time and verifies the backward and comparison identities for mollified tests as in proposition~\ref{sup:prop:additive-comparison}.
\begin{theorem}[One-sided stable estimate]\label{sup:thm:stable-onesided}
Assume \eqref{sup:eq:finite-integrability}, \eqref{sup:eq:stable-onesided-smoothing}, and the corresponding comparison identity. Suppose the initial points, canonical drifts, and Gaussian coefficients agree. For $1\leq q\leq\infty$ and $J_{S,q}\leq1$,
\begin{equation*}
 W_1\bigl(\Law(X_T^0),\Law(X_T^1)\bigr)
 \leq C_{\beta,q,T}J_{S,q}^{\,\frac{1}{3-\beta+\frac{\beta}{q}}}+3\Theta_L,
\end{equation*}
where $\tfrac{\beta}{\infty}=0$. The constant also depends on the smoothing constant in \eqref{sup:eq:stable-onesided-smoothing}.
\end{theorem}
\begin{proof}
Let $q'\in[1,\infty]$ satisfy $\tfrac{1}{q}+\tfrac{1}{q'}=1$. Set
\[
\kappa\coloneqq 2-\frac{\beta}{q'}=2-\beta+\frac{\beta}{q}>0.
\]

For $r\in(0,1]$ and $s>0$, define the kernels $K_2,K_3:(0,1]\times(0,\infty)\to(0,\infty)$ by
\[
K_3(r,s)\coloneqq \min\left\{r^{-2},s^{-\frac{2}{\beta}}\right\},
\qquad 
K_2(r,s)\coloneqq \sqrt{K_3(r,s)}.
\]
For $q'<\infty$, extending the integration interval to $(0,\infty)$ and splitting at $s=r^\beta$ gives
\begin{align*}
 \int_0^TK_3(r,s)^{q'}\dd s
 &\leq\int_0^\infty\min\left\{r^{-2q'},s^{-\frac{2q'}{\beta}}\right\}\dd s\\
 &=\int_0^{r^\beta}r^{-2q'}\dd s
           +\int_{r^\beta}^\infty s^{-\frac{2q'}{\beta}}\dd s \\&=
       r^{\beta-2q'}\left(1+\frac1{\tfrac{2q'}{\beta}-1}\right).
\end{align*}
Here $\tfrac{2q'}{\beta}>1$. Thus $\|K_3(r,\cdot)\|_{L^{q'}(0,T)}\leq C_{\beta,q}r^{-\kappa}$. For $q'=\infty$, the norm is $r^{-2}$ and $\kappa=2$. Since $K_2\leq1+K_3$ and $r\leq1$, both kernels satisfy
\[
 \norm{K_j(r,\cdot)}_{L^{q'}(0,T)}
 \leq C_{\beta,q,T}r^{-\kappa},\qquad j=2,3.
\]
The weighted Taylor test for small jumps, H\"older's inequality, and $\|h-h_r\|_\infty\leq Cr$ yield
\begin{align*}
 |\E\left[h(X_T^0)\right]-\E\left[h(X_T^1)\right]| 
 &\leq2Cr+C\int_0^Tr_S(t)
                 [K_2(r,T-t)+K_3(r,T-t)]\dd t +3\Theta_L\\
 &\byeq{sup:eq:stable-onesided-smoothing}{\leq}
 C_{\beta,q,T}(r+J_{S,q}r^{-\kappa})+3\Theta_L.
\end{align*}
As in theorem~\ref{sup:thm:finite-split}, the large-jump bound uses only Lipschitz contraction. Take the supremum over all $h:\R\to\R$ with $\Lip(h)\leq1$. For $J_{S,q}>0$, choose $r\coloneqq J_{S,q}^{\frac{1}{1+\kappa}}\leq1$. Then $r=J_{S,q}r^{-\kappa}$ and $1+\kappa=3-\beta+\tfrac{\beta}{q}$. If $J_{S,q}=0$, let $r\downarrow0$.
\end{proof}
The exponents at the endpoints are $\tfrac{1}{3}$ for $q=1$ and $\tfrac{1}{3-\beta}$ for $q=\infty$. At $(\beta,q)=(2,\infty)$, the second integral becomes logarithmic, as in \eqref{sup:eq:gaussian-profile-log}.
\section{Analytic tools}\label{app:analytic-tools}
\subsection{Analytic details}\label{sup:sec:analytic}
\subsubsection{A criterion for derivative transfer}\label{sup:sec:flow}
Let $X_s^{1,t,x}$ denote the reference process started from $x\in\R$ at time $t$. Assume that $\Phi_{t,T}:\R\to\R$, defined by $\Phi_{t,T}(x)\coloneqq X_T^{1,t,x}$, is an almost surely $C^3$ stochastic flow. Using $D\coloneqq \partial_x$, write $J\coloneqq D\Phi$, $K\coloneqq D^2\Phi$ and $L\coloneqq D^3\Phi$. Suppose that, for some $p>3$,
\begin{equation}\label{sup:eq:flow-moments}
 \sup_{0\leq t\leq T,\,x\in\R}
       \E\left[|J|^p+|K|^p+|L|^p\right]<\infty.
\end{equation}
For smooth $f:\R\to\R$ with bounded first three derivatives and finite expectations, $P^1_{t,T}f(x)=\E\left[f(\Phi_{t,T}(x))\right]$ satisfies
\[
DP^1_{t,T}f=\E\left[f'(\Phi)J\right],
\qquad  
D^2P^1_{t,T}f=\E\left[f''(\Phi)J^2+f'(\Phi)K\right],
\]
and
\[
D^3P^1_{t,T}f=\E\left[f'''(\Phi)J^3+3f''(\Phi)JK+f'(\Phi)L\right].
\]
To justify each derivative, define the random function $G:\R\to\R$ by $G(x)\coloneqq f(\Phi_{t,T}(x))$. For $j\in\{1,2,3\}$ and $\varepsilon\ne0$,
\[
 \frac{D^{j-1}G(x+\varepsilon)-D^{j-1}G(x)}{\varepsilon}
 =\int_0^1D^jG(x+r\varepsilon)\dd r.
\]
The chain rule gives the three integrands displayed above. Under \eqref{sup:eq:flow-moments}, set $M\coloneqq \sup_{t\in[0,T],\,x\in\R}\E\left[|J|^p+|K|^p+|L|^p\right]$. H\"older's inequality gives
\[
\norm{J^j}_{L^{\frac{p}{j}}(\Omega)}\leq M^{\frac{j}{p}}.
\]

For the product,
\[
\norm{JK}_{L^{\frac{p}{3}}(\Omega)}
 \leq\norm J_{L^p(\Omega)}\norm K_{L^p(\Omega)}\leq M^{\frac{2}{p}}.
\]
The product estimate also uses the embedding $L^{\frac{p}{2}}(\Omega)\subset L^{\frac{p}{3}}(\Omega)$ on a probability space. The terms involving only $K$ or $L$ satisfy the same integrability requirement. Minkowski's inequality then bounds the difference quotients uniformly in $L^{\frac{p}{j}}(\Omega)$. Since $\tfrac{p}{j}>1$, they are uniformly integrable. Pathwise $C^3$ regularity and convergence in expectation justify each differentiation and give \eqref{sup:eq:derivative-transfer}. The comparison identity follows under the hypotheses of Lemma~\ref{sup:lem:comparison}. See \citet{BretonPrivault2024}.
\subsubsection{Integrability of stable density derivatives}
\label{sup:sec:stable-derivatives}
For $1<\beta<2$ and $k\in\{1,2\}$, define
\[
 \mapdef{m_k}{\R}{\mathbb C}{\xi}{(i\xi)^ke^{-|\xi|^\beta}}.
\]
Define $a:\R\to(0,1]$ by $a(\xi)\coloneqq e^{-|\xi|^\beta}$. For $\xi\ne0$,
\[
a'(\xi)=-\beta\operatorname{sgn}(\xi)|\xi|^{\beta-1}a(\xi).
\]

Differentiating once more gives
\[
a''(\xi)=\bigl[\beta^2|\xi|^{2\beta-2}
              -\beta(\beta-1)|\xi|^{\beta-2}\bigr]a(\xi).
\]
Hence
\[
m_1''(\xi)=i\bigl(2a'(\xi)+\xi a''(\xi)\bigr)
                  =O(|\xi|^{\beta-1}),
\]
and
\[
m_2''(\xi)=-2a(\xi)-4\xi a'(\xi)-\xi^2a''(\xi)
                  =-2+O(|\xi|^\beta).
\]
Both estimates hold as $\xi\to0$. Both $m_k$ and $m_k'$ are continuous at zero. At infinity, the functions and their derivatives decay exponentially, so $m_k,m_k''\in L^1(\R)$. In the Fourier convention $\widehat f(\xi)=\int_\R e^{-i\xi x}f(x)\dd x$,
\[
 D^kp_1^{(\beta)}(x)=\tfrac1{2\pi}\int_\R e^{i\xi x}m_k(\xi)\dd\xi
 =-\frac1{2\pi x^2}\int_\R e^{i\xi x}m_k''(\xi)\dd\xi
 \quad(x\ne0).
\]
Neither integration by parts produces a boundary term, including at zero, because $m_k,m_k'$ are continuous there. Therefore
\[
|(p_1^{(\beta)})^{(k)}(x)|
 \leq\tfrac1{2\pi}
       \min\{\norm{m_k}_1,|x|^{-2}\norm{m_k''}_1\},
\]
and
\[
\begin{aligned}
B_{k,\beta}\coloneqq \norm{(p_1^{(\beta)})^{(k)}}_1
 \leq\tfrac1{2\pi}\left(2\norm{m_k}_1
                  +2\norm{m_k''}_1\int_1^\infty x^{-2}\dd x\right)
 =\frac{\norm{m_k}_1+\norm{m_k''}_1}{\pi}<\infty.
\end{aligned}
\]
The stable scaling formula then gives $\|(p_s^{(\beta)})^{(k)}\|_1=B_{k,\beta}s^{-\frac{k}{\beta}}$.
\subsubsection{Physical-time interpolation}\label{sup:sec:physical-time}
For the additive models of Section~\ref{sec:terminal}, write $\lambda_i\coloneqq \Law(X_T^i)$. Let $P^i_{s,t}h(x)\coloneqq \E[h(x+X_t^i-X_s^i)]$ for Borel tests with finite expectations. Suppose $\widetilde Q_t\coloneqq P^0_{0,t}P^1_{t,T}$ has increment density $\widetilde\rho_t\in W^{2,1}(\R)$. For almost every $t\in[0,T]$, set
\[
r_b(t)\coloneqq |b_0(t)-b_1(t)|,
\qquad 
r_a(t)\coloneqq |a_0(t)-a_1(t)|,
\]
and
\[
r_\nu(t)\coloneqq d_{\FM}(y^2\nu_0(t),y^2\nu_1(t)).
\]
Then, whenever the right-hand side is finite,
\begin{equation}\label{sup:eq:physical-time}
\begin{aligned}
 W_1(\lambda_0,\lambda_1)
 \leq |x_0-x_1|+\int_0^Tr_b(t)\dd t
 &+\tfrac12\int_0^T\norm{\widetilde\rho_t'}_1r_a(t)\dd t\\
 &+\int_0^T\left( \tfrac12\norm{\widetilde\rho_t'}_1 +\tfrac16\norm{\widetilde\rho_t''}_1\right)r_\nu(t)\dd t.
\end{aligned}
\end{equation}
Indeed, define the additive full-compensation symbol $\psi_i:[0,T]\times\R\to\mathbb C$ by
\[
 \psi_i(t,\xi)\coloneqq ib_i(t)\xi-\tfrac12a_i(t)\xi^2
       +\int_{\R}(e^{i\xi y}-1-i\xi y)\nu_i(t,\dd y).
\]
For fixed $\xi\in\R$, the mixed exponent is absolutely continuous in $t$, with
\[
\widetilde\Psi_t(\xi)\coloneqq \int_0^t\psi_0(s,\xi)\dd s
                      +\int_t^T\psi_1(s,\xi)\dd s.
\]

Consequently,
\[
\partial_t\widetilde\Psi_t(\xi)=\psi_0(t,\xi)-\psi_1(t,\xi).
\]
Thus $\partial_te^{\widetilde\Psi_t}= (\psi_0-\psi_1)e^{\widetilde\Psi_t}$. Fourier inversion gives the generator comparison with $\widetilde Q_t h$. For $h:\R\to\R$ with $\Lip(h)\leq1$, derivative transfer gives $\|D\widetilde Q_t h\|_\infty\leq1$, $\|D^2\widetilde Q_t h\|_\infty\leq\|\widetilde\rho_t'\|_1$, and $\|D^3\widetilde Q_t h\|_\infty\leq\|\widetilde\rho_t''\|_1.$ Weighted Taylor testing and integration in $t$ yield \eqref{sup:eq:physical-time}. The initial shift costs $|x_0-x_1|$. Integrating the characteristics first preserves the temporal cancellations in \eqref{eq:time-cancellation}; taking absolute values at each time can lose them.
\Needspace{7\baselineskip}
\section*{\normalfont\scshape AI Assistance Declaration}
During the preparation of this manuscript, the author used AI tools for language editing and manuscript organization. The author takes responsibility for the mathematical statements, proofs, computations, references, and final content.


\begingroup
\small
\begin{thebibliography}{99}
\setlength{\itemsep}{0pt}
\setlength{\parsep}{0pt}
\setlength{\parskip}{0pt}

\bibitem[Villani(2009)]{Villani2009} C.~Villani. \newblock \emph{Optimal Transport: Old and New}. \newblock Springer, Berlin, 2009. \newblock \doi{10.1007/978-3-540-71050-9}.
\bibitem[Breton and Privault(2022)]{BretonPrivault2022} J.-C. Breton and N.~Privault. \newblock {Wasserstein} distance estimates for stochastic integrals by forward--backward stochastic calculus. \newblock \emph{Potential Analysis}, 56\penalty0 (1):\penalty0 1--20, 2022. \newblock \doi{10.1007/s11118-020-09874-0}.
\bibitem[Breton and Privault(2024)]{BretonPrivault2024} J.-C. Breton and N.~Privault. \newblock {Wasserstein} distance estimates for jump-diffusion processes. \newblock \emph{Stochastic Processes and their Applications}, 172:\penalty0 104334, 2024. \newblock \doi{10.1016/j.spa.2024.104334}.
\bibitem[Houdr{\'e} et~al.(1998)Houdr{\'e}, P{\'e}rez-Abreu, and Surgailis]{HoudrePerezAbreuSurgailis1998} C.~Houdr{\'e}, V.~P{\'e}rez-Abreu, and D.~Surgailis. \newblock Interpolation, correlation identities, and inequalities for infinitely divisible variables. \newblock \emph{Journal of Fourier Analysis and Applications}, 4:\penalty0 651--668, 1998. \newblock \doi{10.1007/BF02479672}.
\bibitem[Last(2014)]{Last2014} G.~Last. \newblock Perturbation analysis of {Poisson} processes. \newblock \emph{Bernoulli}, 20\penalty0 (2):\penalty0 486--513, 2014. \newblock \doi{10.3150/12-BEJ494}.
\bibitem[Pag{\`e}s and Wilbertz(2012)]{PagesWilbertz2012} G.~Pag{\`e}s and B.~Wilbertz. \newblock Intrinsic stationarity for vector quantization: Foundation of dual quantization. \newblock \emph{SIAM Journal on Numerical Analysis}, 50\penalty0 (2):\penalty0 747--780, 2012. \newblock \doi{10.1137/110827041}.
\bibitem[Goldfeld et~al.(2024)Goldfeld, Kato, Nietert, and Rioux]{GoldfeldEtAl2024} Z.~Goldfeld, K.~Kato, S.~Nietert, and G.~Rioux. \newblock Limit distribution theory for smooth {$p$}-{Wasserstein} distances. \newblock \emph{The Annals of Applied Probability}, 34\penalty0 (2):\penalty0 2447--2487, 2024. \newblock \doi{10.1214/23-AAP2028}.
\bibitem[Kang and Lim(2025)]{KangLim2025} W.~Y. Kang and T.~S. Lim. \newblock On optimal {Markovian} couplings of {L\'evy} processes. \newblock Preprint, 2025. \newblock \href{https://arxiv.org/abs/2509.23086}{arXiv:2509.23086}.
\bibitem[Lim and Ooi(2026)]{LimOoi2026} T.~S. Lim and R.~S. Ooi. \newblock Optimal couplings of {L\'evy} processes in the class of immersion couplings. \newblock Preprint, 2026. \newblock \href{https://arxiv.org/abs/2606.24290}{arXiv:2606.24290}.
\bibitem[Gairing et~al.(2015)Gairing, H{\"o}gele, Kosenkova, and Kulik]{GairingEtAl2015} J.~Gairing, M.~H{\"o}gele, T.~Kosenkova, and A.~Kulik. \newblock Coupling distances between {L\'evy} measures and applications to noise sensitivity of {SDE}. \newblock \emph{Stochastics and Dynamics}, 15\penalty0 (2):\penalty0 1550009, 2015. \newblock \doi{10.1142/S0219493715500094}.
\bibitem[Gairing et~al.(2018)Gairing, H{\"o}gele, and Kosenkova]{GairingEtAl2018} J.~Gairing, M.~H{\"o}gele, and T.~Kosenkova. \newblock Transportation distances and noise sensitivity of multiplicative {L\'evy} {SDE} with applications. \newblock \emph{Stochastic Processes and their Applications}, 128\penalty0 (7):\penalty0 2153--2178, 2018. \newblock \doi{10.1016/j.spa.2017.09.003}.
\bibitem[Mariucci and Rei{\ss}(2018)]{MariucciReiss2018} E.~Mariucci and M.~Rei{\ss}. \newblock {Wasserstein} and total variation distance between marginals of {L\'evy} processes. \newblock \emph{Electronic Journal of Statistics}, 12\penalty0 (2):\penalty0 2482--2514, 2018. \newblock \doi{10.1214/18-EJS1456}.
\bibitem[Alfonsi et~al.(2018)Alfonsi, Corbetta, and Jourdain]{AlfonsiCorbettaJourdain2018} A.~Alfonsi, J.~Corbetta, and B.~Jourdain. \newblock Evolution of the {Wasserstein} distance between the marginals of two {Markov} processes. \newblock \emph{Bernoulli}, 24\penalty0 (4A):\penalty0 2461--2498, 2018. \newblock \doi{10.3150/17-BEJ934}.
\bibitem[Sato(1999)]{Sato1999} K.-i. Sato. \newblock \emph{{L\'evy} Processes and Infinitely Divisible Distributions}. \newblock Cambridge University Press, Cambridge, 1999.
\bibitem[Applebaum(2009)]{Applebaum2009} D.~Applebaum. \newblock \emph{{L\'evy} Processes and Stochastic Calculus}. \newblock Cambridge University Press, Cambridge, second edition, 2009. \newblock \doi{10.1017/CBO9780511809781}.
\bibitem[Wang(2016)]{Wang2016} J.~Wang. \newblock {$L^p$}-{Wasserstein} distance for stochastic differential equations driven by {L\'evy} processes. \newblock \emph{Bernoulli}, 22\penalty0 (3):\penalty0 1598--1616, 2016. \newblock \doi{10.3150/15-BEJ705}.
\end{thebibliography}
\endgroup
\end{document}